\documentclass[11pt, oneside]{article}
\usepackage{graphicx}                           
\RequirePackage{amsthm,amsmath,natbib, amssymb, bm, enumerate, pgf, amscd,amsfonts}

\usepackage{epic,graphpap,graphics,float,tabularx}
\usepackage{url, epsfig, graphicx, color, times}
\usepackage{pgf,tikz}
\usetikzlibrary{arrows}

\numberwithin{equation}{section}
\theoremstyle{plain}
\newtheorem{thm}{Theorem}
\numberwithin{thm}{section}
\newtheorem{prop}{Proposition}
\numberwithin{prop}{section}
\newtheorem{lm}{Lemma}
\numberwithin{lm}{section}
\newtheorem{term}{Terminology}
\numberwithin{term}{section}

\numberwithin{claim}{section}
\newtheorem{cor}{Corollary} 
\numberwithin{cor}{section}

\numberwithin{problem}{section}
\theoremstyle{definition} 
\newtheorem{definition}{Definition} 
\numberwithin{definition}{section}
\newtheorem{example}{Example} 
\numberwithin{example}{section}

\numberwithin{conjecture}{section}

\numberwithin{condition}{section}
\theoremstyle{remark} 
\newtheorem{remark}{Remark}
\numberwithin{remark}{section}

\numberwithin{remark}{section}
\def\R{{\mathbb R}}

\def\N{{\mathbb N}}

\title{
\textsc{
Stochastic Integral Equations   for  Walsh Semimartingales} \thanks{~ We are grateful to Mykhaylo  Shkolnikov  for prompting us to think about angular dependence, and to Johannes  Ruf and Cameron Bruggeman for their critical readings and many suggestions.} 
} 

\vspace{2mm} 

\author{TOMOYUKI ICHIBA  
\thanks{~
Department of Statistics and Applied Probability, South Hall, University of California, Santa Barbara, CA 93106, USA (E-mail:    {\it ichiba@pstat.ucsb.edu}). Research supported in part by the National Science Foundation under grant NSF-DMS-13-13373.
          } 
\and IOANNIS KARATZAS
\thanks{~
Department of Mathematics,  Columbia University, New York, NY 10027 (E-mail: {\it ik@math.columbia.edu}), and       \textsc{Intech} Investment Management,  One Palmer Square, Suite 441, Princeton, NJ 08542, USA.  Research   supported in part by  the National Science Foundation under  grant   NSF-DMS-14-05210.
	}  
\and VILMOS PROKAJ 
\thanks{~ Department of Probability Theory and Statistics, E\"otv\"os Lor\'and University, 1117 Budapest, P\'azm\'any P\'eter s\'et\'any 1/C, Hungary (Email: {\it prokaj@cs.elte.hu}), and Department of Statistics \& Applied Probability, South Hall, University of California, Santa Barbara, CA 93106, USA.
}
\and MINGHAN YAN
\thanks{~
Department of Mathematics,  Columbia University, New York, NY 10027 (E-mail: {\it  my2379@math.columbia.edu}).
	} 
	\date{October 19, 2015} 
}	

\begin{document}
\maketitle

\begin{abstract}  
\noindent \small  
We construct   planar  semimartingales  that include   the  Walsh  Brownian motion as a  special case, and derive   Harrison-Shepp-type  equations and a change-of-variable formula in the spirit of Freidlin-Sheu for these so-called ``Walsh  semimartingales". We examine the solvability of the resulting system of stochastic integral equations. In appropriate Markovian settings  we study two types of connections  to   martingale problems,   questions of uniqueness in distribution for such processes, and a few examples.    
\end{abstract}



\vspace{0.8mm}

\section{Introduction and Summary}

We consider the following   questions: {\it What is a two-dimensional analogue of the   skew Brownian motion on the real line? If such a process  exists, what is the corresponding stochastic   integral equation that realizes its  construction   and describes its dynamics? Are there more general  planar semimartingales with similar skew-unfolding-type structure?}

In order to answer   the first question, \textsc{Walsh} (1978) introduced   a singular planar diffusion with these properties. This diffusion   is known now as the  \textsc{Walsh}  Brownian motion. In its description by \textsc{Barlow, Pitman \& Yor} (1989), ``started at a point in the plane away from the origin, this process moves like a standard Brownian motion along the ray joining the starting point and the origin ${\bm 0}$, until it reaches ${\bm 0}$. Then it is kicked away  from ${\bm 0}$ by an entrance law that makes the radial part of the diffusion a reflecting Brownian motion, while randomizing the angular part". The  \textsc{Walsh}  Brownian motion has been generalized to the so-called spider martingales, and has been studied by several researchers (among them 
\textsc{Barlow, Pitman \& Yor} (1989), 
\textsc{Tsirel'son} (1997),  
\textsc{Watanabe} (1999),
\textsc{Evans \& Sowers} (2003),   
 \textsc{Picard} (2005), 
\textsc{Freidlin \& Sheu} (2000), 
\textsc{Mansuy \& Yor} (2006), 
\textsc{Hajri} (2011), 
\textsc{Fitzsimmons \& Kuter} (2014), 
\textsc{Hajri \& Touhami} (2014), 
\textsc{Chen \& Fukushima} (2015)). In this paper we   construct a family of planar  semimartingales that includes the spider martingales and the \textsc{Walsh} Brownian motion as   special cases.

 There are several constructions of \textsc{Walsh}'s Brownian motions in terms of resolvents, infinitesimal generators, semigroups, and excursion theory. Our approach in this paper can be thought of as a   bridge  between excursion theory and stochastic  integral equations,  via the folding and unfolding of semimartingales. It is also  an attempt to study higher-dimensional analogues of the skew-\textsc{Tanaka} equation, and the semimartingale properties of planar processes that hit points.   
 
 \medskip
 \noindent
 {\it Preview:}  We   provide in Section \ref{sec: Results} a system of stochastic   equations (\ref{eq: skewTanaka}) that these planar  semimartingales satisfy.  This system  is a two-dimensional  analogue of the equation introduced by \textsc{Harrison \& Shepp} (1981) for the skew Brownian motion, and     answers  the  second and third questions stated above.  Based on this  integral-equation description,  we develop in Sections \ref{disc} and \ref{FS} a stochastic calculus,     and establish a \textsc{Freidlin-Sheu} type change-of-variable formula (\ref{GenFS1}),  for such \textsc{Walsh} semimartingales. We also develop a condition (\ref{Gammaless1}) closely analogous to that of \textsc{Harrison \& Shepp} (1981), for the solvability of our   system of equations (\ref{eq: skewTanaka}).  
 In Section \ref{Pf} we   examine by the method of \textsc{Prokaj} (2009) this two-dimensional \textsc{Harrison-Shepp} equation driven by a continuous semimartingale, as in \textsc{Ichiba \& Karatzas} (2014).  
 
 Pathwise uniqueness fails for the system of   (\ref{eq: skewTanaka}). For the \textsc{It\^o} diffusion case, we recast this system  in the form (\ref{eq: SDEMP}), and study its connections to an appropriate martingale problem  in Sections \ref{sec: MP} and \ref{sec: AD}. The well-posedness of this martingale problem, in the form of  conditions under which a weak solution exists  for the  system of (\ref{eq: SDEMP}) and is unique in distribution, is based on the stochastic calculus of Section \ref{FS}.

The \textsc{Walsh} Brownian motion constructed via the \textsc{Feller} semigroup, is then shown in Section \ref{subsec: WBM} to be a special case of our ``\textsc{Walsh} diffusion" framework. Another type of connection to   martingale problems is established in Section \ref{SMP}, allowing us to show  that \textsc{Walsh} diffusions are the only time-homogeneous and strongly  Markovian solutions of the system (\ref{eq: SDEMP}).  
 A notable difference from the \textsc{Harrison-Shepp} equation is also given there;  whereas in Section \ref{Ex}    we  study additional examples.  Some auxiliary results and proofs are provided in the  appendices, Sections    \ref{sec: App_Lem} and \ref{sec: App}.

\section{The Setting and Results} \label{sec: Results}

On a filtered probability space $\,( \widetilde{\Omega}, \widetilde{\mathcal F}, \widetilde{\mathbb P}) \,$, $\, \widetilde{\mathbb F}= \big\{ \widetilde{\mathcal F}({t})\big\}_{0 \le t < \infty} \,$ that satisfies the ``usual conditions" of right-continuity and augmentation by null sets,  we consider a real-valued, continuous semimartingale 
\begin{equation} 
\label{eq: U}
\,U(t) \, =\,  M(t) + V(t) \,, \quad \,\,0 \le t < \infty\,. 
\end{equation}
Here $\,M(\cdot)\,$ is a continuous local martingale and $\,V(\cdot)\,$ has finite variation on compact intervals;   we   assume that the initial position $\, U(0)\ge 0 \,$ is a given real number. We denote  by 
\begin{equation}
 \label{eq: S}
S( t) \, :=\, U  ( t) +  \Lambda ( t)\,, \qquad \text{where} \quad 
 \Lambda ( t)= \max_{0 \le s \le \, t} \big( - U(s)\big)^{+}\,, \qquad 0 \le t < \infty\,,
\end{equation}
the \textsc{Skorokhod} reflection (or ``folding") of $\,U(\cdot)\,$;  see, for instance,  section 3.6 in \textsc{Karatzas \& Shreve} (1991) for   relevant theory. In particular, the continuous, increasing process $\, \Lambda (\cdot)\,$ is flat off the zero set 
\begin{equation} 
\label{eq: Z_{S}}
 \mathfrak Z  \, :=\, \big\{ 0 \le t < \infty: \,S(t)=  0 \big\} \,.   
 \end{equation}
We shall impose the ``non-stickiness" condition   
\begin{equation}
 \label{eq: SC}
\text{Leb} ( \mathfrak Z) \,\equiv \,\int^{\infty}_{0} {\bf 1}_{\{ S(t) \, =\,  0\}}\, {\mathrm d} t \, =\,  0 \,.
\end{equation}

Let us recall the   {\it (right) local time} $\,L^{\Xi}(\cdot)\,$ accumulated at the origin during the time-interval $[0,T]$ by  a generic one-dimensional continuous semimartingale $\, \Xi (\cdot)\,$,  namely 
\begin{equation} 
\label{eq: LT}
L^{\Xi}(T) \, :=\, \lim_{\varepsilon \downarrow 0} \frac{1}{\, 2\,\varepsilon\, }\int^{T}_{0} {\bf 1}_{\{ 0 \le \Xi(t) < \varepsilon \}} \,  {\mathrm d} \langle \Xi   \rangle (t)\,, \qquad 0 \le T < \infty  \,.  
\end{equation}
From (\ref{eq: S}) and by analogy with Lemma 3.1.5 in \textsc{Picard} (2005),   we have   the \textsc{It\^o-Tanaka-}type equation 
\begin{equation}
\label{Pic}
S(\cdot) \,=\, S(0) + \int_0^{\, \cdot} \mathbf{ 1}_{ \{ S(t) >0 \}} \, \mathrm{d} U (t) + L^S (\cdot)\,, \qquad S(0) = U(0) \ge 0\,.
\end{equation}
On the other hand,  the theory of semimartingale local time (e.g., section 3.7 in \textsc{Karatzas \& Shreve} (1991)) gives  the properties    
 \begin{equation}
\label{A.1}
\int^{\infty}_{0} {\bf 1}_{\{S(t) \, =\,  0\}} {\mathrm d} \langle S \rangle (t) \, =\,  0 \, ,\qquad L^{S}(\cdot) \, =\,  \int^{\cdot}_{0} {\bf 1}_{\{S(t) \, =\,  0\}} {\mathrm d}  S  (t) \, .  
\end{equation}

 \subsection{The Main Result} 
 \label{sec: Main}

  Theorem \ref{prop: skewTanaka} below is the first  key result of this paper. It produces a planar ``skew-unfolding" $X(\cdot)  =  (X_{1}(\cdot), X_{2}(\cdot))^{\prime}\,$  for the  folding  $\,S(\cdot)\,$ of the given continuous semimartingale $\, U(\cdot)\,$.  This planar ``skew-unfolded" process has radial part   $\, || X(\cdot)|| = S(\cdot)\,$, and its motion away from the origin follows the one-dimensional dynamics of $\, S(\cdot)\,$  along rays emanating from the origin. Once at the origin, the process $X(\cdot)$ chooses the next ray for  its   voyage $($according to the dynamics of $\, S(\cdot))$     independently of its past history and in a random fashion, according to a given probability measure   on the collection of angles in $\, [0, 2 \pi)\,$. Whenever $\,S(\cdot)\,$ is a reflecting Brownian motion or, more generally, a reflecting diffusion, these one-dimensional dynamics away from the origin are of course diffusive.

In order to describe this skew-unfolding with some detail and rigor, we shall need appropriate notation. Let us  consider the unit circumference  $$\, \mathfrak S   :=   \big\{(z_{1}, z_{2})^{\prime} : z_{1}^{2} + z_{2}^{2} = 1\big\}\,.$$ Here and throughout the paper, vectors are columns and the superscript $\,^{\prime}\,$ denotes transposition.  For every point $\,x \, :=\, (x_{1}, x_{2})^{\prime} \in \mathbb R^{2}\,$ we  introduce the mapping  $\, \mathfrak f     =  \big(\mathfrak f_{1} , \mathfrak f_{2}  \big)^{\prime} :\mathbb R^{2} \rightarrow \mathfrak S \cup \{{\bm 0}\}\,$   via   $\, \mathfrak f ({\bm 0})   :=   {\bm 0}\,$ and 
\begin{equation}
 \label{eq: f} 
 \mathfrak f(x) \, :=\,  \frac{ x}{\, \lVert x \rVert\, } \, =\,  \big( \cos( \text{arg}(x)) \, , \,  \sin (\text{arg}(x)) \, \big)^{\prime} \, ; \qquad x \in 
 E   :=  	\mathbb R^{2} \setminus \{ {\bm 0} \} 
\end{equation} 
with the notation $\,{\bm 0}:= (0,0)'\,$ and  with $\, \text{arg}(x) \in [0, 2\pi) \,$ denoting the {\it argument of the vector} $\,x \in \R^2 \setminus \{ \bm 0 \}\,$ in its polar co\"ordinates. 
We   fix   a probability measure $\, {\bm \mu} \,$ on the collection $\, \mathcal B(\mathfrak S)\,$ of Borel subsets of the unit circumference $\,  \mathfrak S \,$, and consider also its expression 
\begin{equation}
\label{nu_mu}
\,  {\bm \nu}  ( {\mathrm d} \theta )  \, :=\,  {\bm \mu} ( {\mathrm d} z )\,, \quad z = \big(\cos (\theta), \, \sin (\theta)\big)' \in \mathfrak S\, , \,\,\,\theta \in [0, 2\pi) 
\end{equation}
in polar co\"ordinates. We introduce the real constants    
\begin{equation}
\label{alpha}
\alpha_{i}^{(\pm)} \, :=\, \int_{\mathfrak S} \big({\mathfrak f}_{i}(z)\big)^{\pm} {\bm \mu } ({\mathrm d} z) \, , \quad \gamma_i \,:=\, \alpha_{i}^{(+)} - \alpha_{i}^{(-)}\,= \int_{\mathfrak S}  {\mathfrak f}_{i}(z) \,  {\bm \mu } ({\mathrm d} z) \,, \quad i \, =\, 1, 2  
\end{equation}
as well as  the vector on the unit disc 
\begin{equation}
\label{gamma}
 {\bm \gamma} \, := \, \big( \gamma_{1}, \gamma_{2}\big)^{\prime}  
 \,=\, \left(  \int_{0}^{2\pi} \cos (\theta) \,{\bm \nu}({\mathrm d} \theta )\,, \int_{0}^{2\pi} \sin (\theta) \, {\bm \nu}({\mathrm d} \theta )\right)^{\prime}.
\end{equation}
 
 Finally, we fix     a   vector $\, {\rm x} \, :=\, ({\rm x}_{1}, {\rm x}_{2})' \in \mathbb R^{2}\,$ with $\,{\rm x}_i = \mathfrak{ f}_i (\mathrm{x})\, S(0)\,, \,\, i=1, 2\,$.
 
\begin{thm} 
\label{prop: skewTanaka} {\bf Construction of \textsc{Walsh} Semimartingales:} Consider the \textsc{Skorokhod} reflection $\,S(\cdot)\,$ of the continuous semimartingale $ U(\cdot)\,$ as in (\ref{eq: U}) --    
(\ref{eq: SC}), and fix a vector $\, \mathrm x = (\mathrm{x}_{1}, \mathrm{x}_{2})^{\prime} \in \mathbb R^{2}  \,$   as above.

On a suitable enlargement $\, ( {\Omega}, {\mathcal F}, \mathbb P) \,$, $\, \mathbb F := \{\mathcal F(t)\}_{0 \le t < \infty}\,$ of the filtered probability space $\, ( \widetilde{\Omega}, \widetilde{\mathcal F} , \widetilde{\mathbb P})\,$, $\, \widetilde{\mathbb F}\,$ with a measure preserving map $\,{\bm \pi} : {\Omega} \to \widetilde{\Omega}\,$,  there exists a planar    
continuous  semimartingale $\,X(\cdot)  :=  (X_{1}(\cdot), X_{2}(\cdot))^{\prime}\,$       which solves   
the system of stochastic  integral equations
\begin{equation} 
\label{eq: skewTanaka} 
X_{i}(T) \, =\,  {\rm x}_{i} + \int^{T}_{0} \mathfrak f_{i}\big(X(t)\big) \,{\mathrm d} S(t) +      \Big( \alpha_{i}^{(+)}-  \alpha_{i}^{(-)} \Big) 
L^{ S }(T) \, , \qquad 0 \le T < \infty  
\end{equation}
for $\, i=1, 2\,$  and  whose radial part is 
\begin{equation} 
\label{length} 
\Vert X(\cdot) \Vert\,:=\,\sqrt{X_{1}^{2}(\cdot) + X_{2}^{2}(\cdot)\,} \, =\,  S (\cdot) \, .
\end{equation}

This continuous  semimartingale $\,X(\cdot)  :=  (X_{1}(\cdot), X_{2}(\cdot))^{\prime}\,$ has the following properties: 

\smallskip
\noindent
{\bf (i)} With $\,\mathrm{x} \in \R^2 \setminus \{ \mathbf{ 0}\}\,$  and $\, 
\,\tau(s)   :=  \inf \big\{ t > s : X(t) \, =\, {\bm 0} \big\} \,$ 
 the first time it reaches the origin after time $\,s \ge 0\,, $    this process $\, X(\cdot)\,$ satisfies for every $\, s \in (0, \infty) \,$, $\, B \in \mathcal B ( \mathfrak S)\,$ and for Lebesgue almost every $\, t \in (0, \infty)\,$   the properties 
 \begin{equation} 
\label{eq: DIST3}
 \mathfrak f(X(s))\, = \,{\mathfrak f}(\mathrm x)\, , \quad \mathbb P-\text{a.e. on}~~\{ \tau (0) > s\}\,,
 \end{equation}
\begin{equation} 
\label{eq: DIST}
\mathbb P \big(\, \mathfrak f(X(\tau (s) + t))   \in   B \,\big| \, {{\cal F}} ^{X}(\tau(s)) \big) \, =\,  {\bm \mu}(B)\, , \quad \mathbb P-\text{a.e. on}~~\{ \tau (s) < \infty\}\,. 
\end{equation}
 
\medskip
 \noindent
      {\bf (ii)} \,
The local times at the origin of the component processes $\, X_i (\cdot)\,$ are given as   
\begin{equation} 
\label{eq: LTXS}
 \,L^{X_{i}}(\cdot)  
  \, \equiv \, \alpha_{i}^{(+)}\, L^{\,|| X ||}(\cdot)  \,, \qquad i=1, 2
\end{equation}    
and are thus flat off the random set $\,\mathfrak Z\,$ 
 in (\ref{eq: Z_{S}}) which has    zero \textsc{Lebesgue} measure by (\ref{eq: SC}); in particular, 
\begin{equation} 
\label{eq: zeroLeb}
\int^{\infty}_{0} {\bf 1}_{\{X(t) \, =\,  {\bm 0} \}}  \,  {\mathrm d} t \, \equiv\,  0 \,.  
\end{equation}
      {\bf (iii)} \,Finally, for every $\, A \in \mathcal B ([0, 2\pi)) ,$ the semimartingale local time at the origin of the ``thinned" process  
             $\,R^{A}(\cdot) :=   \lVert X (\cdot) \rVert \cdot {\bf 1}_A \big(\text{arg}\big(X(\cdot)\big) \big)  \,$         is given by 
\begin{equation}
\label{eq: localDist}
L^{R^{A}}(\cdot) \, \equiv \, {\bm \nu} (A) \, L^{ \,\lVert X \rVert}(\cdot) \,.  
\end{equation}
\end{thm}

\begin{term}
{\rm 
The process $\, X(\cdot)\,$ constructed in the above Theorem can be thought of as a   planar skew-unfolding  of the \textsc{Skorokhod} reflection $\,S(\cdot)\,$ of the driving continuous semimartingale $\,U(\cdot)\,;$ we shall  call   it  a}      \textsc{Walsh} semimartingale with ``driver"  $\,U(\cdot)$ $(${\rm and} ``folded driver"  $ S(\cdot)).$   {\rm We shall call $\, X(\cdot)\,$ a}   \textsc{Walsh} diffusion, {\rm whenever the folded driver $\,S (\cdot) \,$ is an \textsc{It\^o} diffusion with  reflection at the origin. 

 With the family of increasing processes $\, L^{R^{A}}(\cdot)\,$, $\, A \in \mathcal B \big([0, 2\pi)\big)\,$ in (\ref{eq: localDist}), we shall find it conceptually convenient to associate a random product measure $\,{\bm \Lambda}^{X}  ( {\mathrm d} t , {\mathrm d} \theta)  \,$ on $\, [0, \infty) \times [0, 2\pi)\,$ via} 
\begin{equation} 
\label{eq: RM} 
{\bm \Lambda}^{X} \big([0, t) \times A\big) \, :=\, L^{R^{A}}(t) \, =\, {\bm \nu} (A) \, L^{ \lVert X \rVert}(t) \, ; \qquad 0 \le t < \infty \, , \quad A \in \mathcal B \big([0, 2\pi)\big)\, . 
\end{equation}
\end{term}

\section{Discussion and Ramifications}
 \label{disc}

An intuitive interpretation of the stochastic integral equations (\ref{eq: skewTanaka}) with the property (\ref{length}) is as follows: We  first ``fold" the driving semimartingale $\,U(\cdot)\,$ to get its \textsc{Skorokhod} reflection $ S(\cdot) $ as in (\ref{eq: S}) and then,  starting  from the point $\, \mathrm x = (\mathrm{x}_{1}, \mathrm{x}_{2})^{\prime} \in \mathbb R^{2} \setminus \{ {\bm 0} \}\,$ with $\,{\rm x}_i = \mathfrak{ f}_i (\mathrm{x})\, S(0)\,, \,\, i=1, 2\,$  and up until the time $\, \tau (0) \,$ of Theorem \ref{prop: skewTanaka}(i), we   run  the planar process  $\,X (\cdot) = ( X_{1}(\cdot),  \, X_{2}(\cdot))'\,$ according to the     integral equation  
\begin{equation} 
\label{eq: skewTanaka_flat}
X_i (\cdot)\,=\, {\rm x}_{i} + \int^{\cdot}_{0} {\mathfrak f}_{i}\big(X( t)\big)\, {\mathrm d} S( t) \,, \quad \text{for}\,\,\,\,\,i=1,2  
\end{equation} 
 \newpage
\noindent on $ [0, \tau(0))$. This is the equation to which (\ref{eq: skewTanaka}) reduces on the interval $\, [0, \tau(0)) $. By the definition of the function $\, {\mathfrak f}= ({\mathfrak f}_{1} , {\mathfrak f}_{2} )^{\prime}\,$ of (\ref{eq: f}), {\it the motion of the two-dimensional process $\,X(\cdot) = ( X_1(\cdot), X_2(\cdot))' \,$ during the time-interval $\, [0, \tau(0))\,$ is along the ray that connects the origin ${\bm 0}$ to the starting point $\, \mathrm x\,$.  }

Here is an argument  for this claim, which proceeds by applying  \textsc{It\^o}'s rule to the stochastic integral equation (\ref{eq: skewTanaka}) in conjunction with (\ref{length}): Given $\, \varepsilon \in (0,  \lVert \mathrm x \rVert)   $, let us define the stopping time $\, \sigma_{\varepsilon} \, :=\, \inf\{ t > 0 :  S(t) \, =\, \lVert X (t) \rVert  \le \varepsilon \} \,$. Let us recall that the local time $\, L^{S} (\cdot)\,$ is flat off  the random set $\, \mathfrak Z\,$ in (\ref{eq: Z_{S}}); thus  (\ref{eq: skewTanaka}) reduces to (\ref{eq: skewTanaka_flat})  on $\,[ 0, \sigma_{\varepsilon}] \, $. Applying \textsc{It\^o}'s rule, we observe
\[
\hspace{-3.5cm} \frac{\, X_{i}(t \wedge \sigma_{\varepsilon})\, }{ \, \lVert X(t \wedge \sigma_{\varepsilon})\rVert\, } \, =\,  \frac{\, {\rm x}_{i}\, }{\, \lVert \mathrm x \rVert\, } + \int^{t\wedge \sigma_{\varepsilon}}_{0} \frac{ \, {\mathrm d} X_{i}(s)\, }{\, \lVert X(s) \rVert\, } -  \int^{t\wedge \sigma_{\varepsilon}}_{0} \frac{ \, X_{i}(s) \, }{ \lVert X(s)\rVert^{2}\, } \,{\mathrm d} \lVert X(s) \rVert   
\] 
\begin{equation} \label{eq: ratio}
\hspace{3.9cm} {} +  \int^{t\wedge \sigma_{\varepsilon}}_{0} \frac{ \, X_{i}(s) \, }{ \, \lVert X(s) \rVert^{3}\, } \,{\mathrm d} \langle \lVert X \rVert\rangle(s) -  \int^{t\wedge \sigma_{\varepsilon}}_{0} \frac{1}{\, \lVert X(s) \rVert^{2}\, } \,{\mathrm d} \langle X_{i}, \lVert X \rVert\rangle (s)     
\end{equation}
for $\, t \ge 0 \,$ and $\,i \, =\, 1, 2\,$. Because of the definition of $\,\mathfrak f(\cdot)\,$, (\ref{eq: skewTanaka_flat}) and (\ref{length}),    
we obtain  
\[
X_{i}(t) \, =\,  \lVert X(t) \rVert \, \mathfrak f_{i}(X(t))\,, \qquad 
{\mathrm d} X_{i} (t) \, =\, \mathfrak f _{i}(X(t)) {\mathrm d} S(t)  
\, =\, \frac{\, X_{i}(t)\, }{\, \lVert X(t) \rVert\, } \, {\mathrm d} \lVert X(t) \rVert \, , 
\]
\[
\langle X_{i}, \lVert X\rVert\rangle(t) \, =\,  \int^{t}_{0} \mathfrak f_{i}(X(s) )\,  {\mathrm d} \langle \lVert X\rVert\rangle(s) \, =\, \int^{t}_{0} \frac{\, X_{i}(s) \, }{\, \lVert X(s) \rVert\, }  \, {\mathrm d} \langle \lVert X\rVert\rangle(s) 
\]
on $\,[0, \sigma_{\varepsilon}] \,$, for $\,i \, =\, 1, 2\,$. Substituting these relations into (\ref{eq: ratio}), we deduce 
\[
\mathfrak f_{i} (X(t )) \, =\, \frac{\, X_{i}(t )\, }{ \, \lVert X(t )\rVert\, } \, =\,  \frac{\, {\rm x}_{i}\, }{\, \lVert \mathrm x \rVert\, } 
\, =\, \mathfrak f_{i} (\mathrm x) \quad \text{ on } \, \, \,\,[0, \sigma_{\varepsilon}] \, , \, \, i \, =\, 1, 2 \, . 
\]
Since $\,\varepsilon > 0\,$ is arbitrary, this concludes the proof of the above claim, in accordance  with (\ref{eq: DIST3}).  

\smallskip
Now, every time the planar process  $\,X(\cdot)\,$ visits the origin, the direction of the next ray for its $\,S(\cdot)$-governed  motion is instantaneously chosen  at random    according to the probability distribution $\,{\bm \mu} $, the ``spinning measure" of the process $\,X(\cdot) \,$, in a manner described in more detail later. If the origin  is visited infinitely often during a  time-interval of finite length, it is not surprising that this random choice should lead to the accumulation of  local time at the origin, as indicated in the equations   (\ref{eq: skewTanaka}). It follows from (\ref{eq: zeroLeb}) that the set of    times spent by   $\,X(\cdot)\,$   at the origin has zero Lebesgue measure.  The process continues to move then along the newly chosen ray, its motion governed by the stochastic integral equations of  (\ref{eq: skewTanaka_flat}) just described, as long as it stays away from the origin.  The path $\,t \mapsto X(t) \,$ is, with probability one,  continuous in the topology induced by the tree-metric (French railway metric) on the plane, namely 
\begin{equation} 
\label{eq: tree}
\varrho (x,y) \,:=\, (r_1 + r_2) \,  \mathbf{ 1}_{ \{ \theta_1 \neq \theta_2 \} } + |r_1 - r_2| \,  \mathbf{ 1}_{ \{ \theta_1 = \theta_2 \} }\,, \qquad x= (r_1, \theta_1)\,, \,\,y= (r_2, \theta_2)\,. 
\end{equation} 
The reader may find it useful at this  juncture to think of a   roundhouse at the origin,   of the spokes of a bicycle wheel  -- or of the Aeolian winds of Homeric lore, that blow the raft of Odysseus  in all directions at once.    
 
 \subsection{Spider   Semimartingales}
 \label{Sp}
 
 Suppose that the   measure $\, {\bm \mu}  \,$   charges only a finite   number $\,m\,$ of points on the unit circumference. 
 We can think then of the planar process $\, X(\cdot)\,$  in Theorem \ref{prop: skewTanaka}  as a   {\it Spider Semimartingale,} whose radial part     $\, \Vert X(\cdot) \Vert = S(\cdot) \,$ is the \textsc{Skorokhod} reflection of the driver  $\, U(\cdot)$  according to (\ref{length}).    
 
When the driving semimartingale $\,U(\cdot)\,$ is Brownian motion,     the process $\,X(\cdot)\,$ of Theorem \ref{prop: skewTanaka} becomes   the original {\it \textsc{Walsh} Brownian Motion} (as constructed, for instance, in \textsc{Barlow, Pitman \& Yor} (1989)) with roundhouse singularity in a multipole field; this will be shown in Proposition \ref{prop: ID} below. When $\,m=2\,$ and $\, {\bm \nu} (\{ 0\}) = \alpha \in (0,1)\,$, $\, {\bm \nu} (\{ \pi\}) = 1- \alpha\,$, this construction recovers the familiar {\it Skew Brownian Motion}, introduced in \textsc{It\^o \& McKean} (1963) and studied by  \textsc{Walsh} (1978) and by \textsc{Harrison \& Shepp} (1981).

\subsection{Generalized  Skew-\textsc{Tanaka} and \textsc{Harrison-Shepp} Equations}
 \label{La}

 In the context of Theorem \ref{prop: skewTanaka} (in particular, with the property  (\ref{length})),   the equations of (\ref{eq: skewTanaka})  can be cast in   equivalent forms, now driven by the  original semimartingale $\,U (\cdot)$, as follows:   
\begin{equation} 
\label{eq: skewTanaka_2} 
X_{i}(\cdot) \, =\,  {\rm x}_{i} + \int^{\,\cdot}_{0} \mathfrak f_{i}\big(X(t)\big) \,{\mathrm d} U(t) +   \gamma_i\, L^{|| X||}(\cdot) \,, \qquad  i=1, 2  \,,
\end{equation} 
\begin{equation} 
\label{eq: skewTanaka_1} 
X_{i}(\cdot) \, =\,  {\rm x}_{i} + \int^{\,\cdot}_{0} \mathfrak f_{i}\big(X(t)\big) \,{\mathrm d} U(t) +  \left( 1- \frac{\, \alpha_{i}^{(-)}\,}{\, \alpha_{i}^{(+)}\, }\right) L^{ X_{i} }(\cdot) \, ,  \quad  i=1, 2   
\end{equation}
 (the latter when $\,\alpha_{i}^{(+)}>0$).   This last system    (\ref{eq: skewTanaka_1}) is a two-dimensional version of the  {\it skew-\textsc{Tanaka} equation} studied by \textsc{Ichiba \& Karatzas} (2014).   
 
 The system of equations (\ref{eq: skewTanaka_2}) 
, on the other hand,  can be thought of as a planar semimartingale version of the  \textsc{Harrison \& Shepp} (1981) equation for the skew Brownian motion. 
For two fixed real constants $\,\gamma_1, \gamma_2\,$, and a folded driver  $\, S(\cdot)\,$ that satisfies the condition 
 \begin{equation}
\label{eq: Pos_Loc_Time}
\mathbb{P} \big( L^S (\infty) >0 \big) \,>\,0 
\end{equation}
(e.g., reflecting Brownian motion),  we have the following necessary and sufficient condition (\ref{Gammaless1}), for the  solvability of the system (\ref{eq: skewTanaka_2})  subject to the condition (\ref{length}). The requirement (\ref{Gammaless1}) is a two-dimensional analogue of the condition in \textsc{Harrison \& Shepp} (1981) for the solvability of the stochastic equation that characterizes the skew Brownian motion. The proof of Theorem \ref{IFF} right below, is given in section \ref{Pf}.

\begin{thm} {\bf A Generalized  \textsc{Harrison-Shepp} System of Equations:}  
 \label{IFF}
Consider  a continuous  semimartingale $U(\cdot)$ along with its \textsc{Skorokhod} reflection $S(\cdot)$ as in section \ref{sec: Main},   two real numbers $\,\gamma_1, \gamma_2\,$, and  a   vector $\, {\rm x} \, :=\, ({\rm x}_{1}, {\rm x}_{2})' \in \mathbb R^{2}\,$ with $\,{\rm x}_i = \mathfrak{ f}_i (\mathrm{x})\, S(0)\,, \,\, i=1, 2\,$. 
 
 \smallskip
 \noindent
 (i) Suppose that   
 $\,\gamma_1, \gamma_2\,$ satisfy the condition
 \begin{equation}
 \label{Gammaless1}
  \,\gamma_{1}^{2}+\gamma_{2}^{2}  \, \leq  \,1\, .
 \end{equation}
  There exists  then a continuous planar semimartingale    $X(\cdot)=  (X_{1}(\cdot), X_{2}(\cdot))^{\prime} $ that satisfies the system     (\ref{eq: skewTanaka_2})  and the condition  (\ref{length}).
 
  \smallskip
  \noindent
 (ii) Conversely, suppose that   (\ref{eq: Pos_Loc_Time}) holds, and that there exists a continuous planar semimartingale    $X(\cdot)=  (X_{1}(\cdot), X_{2}(\cdot))^{\prime} $ that satisfies  the system      (\ref{eq: skewTanaka_2}) and the condition   (\ref{length}). Then (\ref{Gammaless1}) is satisfied by $\,\gamma_1, \gamma_2\,$. 
\end{thm}

\subsection{Open Questions} 

\noindent $\,\bullet \,$ It would be of considerable interest to extend the methodology of this paper to a situation with   an entire  family $\, U(\cdot\,; \mathbf{ z})\,,~\mathbf{ z} \in \mathfrak{ S}\,$  of semimartigales so that, when the point $\,\mathbf{ z}\,$ is selected on the unit circumference by the spinning measure $\, {\bm \mu}\,$, the motion along the corresponding ray is according to the \textsc{Skorokhod} reflection $\, S(\cdot\,; \mathbf{ z})  \,$ of this semimartingale $\, U(\cdot\,; \mathbf{ z})\,$. 
Some  results in this vein are obtained in section \ref{sec: AD}, in the context of the  diffusion case and by the method of scale function and time-change.

\noindent $\,\bullet \,$ What are the descriptive statistics of the \textsc{Walsh} semimartingale? For example, what is the area of the convex hull of its path $\,\{ X(s) , \,\,0 \le s \le T\}\,$ for some time $\,T\,$? 

\section{A  \textsc{Freidlin-Sheu-}type Formula}
 \label{FS}

 Let us consider now a twice continuously differentiable function $\, g : \R^2 \rightarrow  \R\,$. If $\,X(\cdot)\,$ is a continuous planar semimartingale that satisfies the system of equations (\ref{eq: skewTanaka}) with the property (\ref{length}),   an application of \textsc{It\^o}'s rule with the notation of (\ref{eq: f})  gives
 \newpage  
      $$
 g (X(T))\, =\,g (\mathrm{x}) \,+ \int_0^T \left( \sum_{i=1}^2 D_i g (X(t)) \,\mathfrak{f}_i (X(t)) \right) \mathrm{d} S(t)\, + \,\sum_{i=1}^2 D_i g (\mathbf{ 0})\, 
 \gamma_i \cdot L^S(T) 
 $$
 $$
~~~~~~~~~~~~  \qquad \qquad \qquad ~~~~ + \,\frac{1}{\,2\,} \int_0^T \left( \sum_{i=1}^2 \sum_{j=1}^2 D_{ij}^2 g (X(t)) \,\mathfrak{f}_i (X(t)) \,\mathfrak{f}_j (X(t)) \right) \mathrm{d} \langle S \rangle (t)\,, \quad 0 \le T < \infty\,.
 $$
 We define now the functions
\begin{equation}
\label{primes}
  G^{\prime} (x) := \sum_{i=1}^2 D_i g (x)\,\mathfrak{f}_i (x)\,, \qquad  G^{\prime \prime} (x) := \sum_{i=1}^2 \sum_{j=1}^2 D_{ij}^2 g (x) \,\mathfrak{f}_i (x) \,\mathfrak{f}_j (x)
\end{equation}
on the punctured plane $\,E= \R^2 \setminus \{ \mathbf{ 0}\}$, and  consider them   as the first and second derivatives,  
respectively, of the function $\,g(\cdot)\,$ in its radial argument $\, r = \sqrt{x_1^2 + x_2^2\,}\,$. 
 With this notation and that of (\ref{alpha}), (\ref{gamma}), the above decomposition can be written in the   \textsc{Freidlin-Sheu} (2000) form
\begin{equation}
\label{GenFS}
 g (X(\cdot)) = g (\mathrm{x}) + \int_0^{\, \cdot} \mathbf{ 1}_{\{ X(t) \neq \mathbf{ 0}\}} \left( G^{\prime} (X(t)) \,    \mathrm{d} S(t)  + \frac{1}{\,2\,} \, G^{\prime \prime} (X(t)) \, \mathrm{d} \langle S \rangle (t) \right) +\sum_{i=1}^2 \,\gamma_i\,D_i g (\mathbf{ 0})\cdot L^S(\cdot)\,.
\end{equation}
We  also note that the   constant  $\,\sum_{i=1}^2 \gamma_i\,D_i g (\mathbf{ 0})  \,$  which multiplies the local time term in (\ref{GenFS}), can be cast as 
\begin{equation}
\label{H2}
\sum_{i=1}^2 \,\gamma_i\,D_i g (\mathbf{ 0})\, =\, \int_{0}^{2\pi} h(\theta) {\bm \nu}({\mathrm d} \theta)\,,\quad \,\,\,\text{the integral   of } \quad 
\, h(\theta) \, :=\, \lim_{ \lVert x \rVert \downarrow 0} G^{\prime} (x) \Big|_{ \, \text{arg}(x) \, =\, \theta}\,
\end{equation}
  with respect to the spinning   measure 
 expressed here in polar co\"ordinates,  as in \textsc{Hajri \& Touhami} (2014). 
 
 


\subsection{A Generalization of the Change-of-Variable Formula (\ref{GenFS})}
  \label{FSa}

Let us   try to refine   the above considerations.  It is   clear that, along the paths of the process $\, X(\cdot)\,$ constructed in Theorem \ref{prop: skewTanaka}, only derivatives of the form indicated in (\ref{primes})  -- i.e., radial -- appear in the     \textsc{Freidlin-Sheu-}like   formula (\ref{GenFS}). This suggests that the smoothness assumption in (\ref{GenFS}) 
can be relaxed, as follows.

 \begin{definition}
  \label{def: D}
  We consider    the class $\, \mathfrak{D}\, $ of  \textsc{Borel}-measurable  functions $ \, g:    \mathbb R^{2}   \rightarrow \R\,$  with    the   properties:   
 
\noindent
 (i)~~~~   they are  continuous in the topology induced by the tree-metric (\ref{eq: tree}) on the plane;  
 
 \noindent 
(ii)   ~~\,for every $\,\theta \in [0, 2 \pi)\,$, the function $\,  r \,\longmapsto \,g_{\theta}(r)   :=  g(r, \theta)    \,$ 
 is twice continuously differentiable on $\, ~~~~~\,~~~~(0, \infty)\,$ and has 
 finite first and second right-derivatives at the origin;    
 
  \smallskip
\noindent  (iii) ~  the resulting functions $\, (r, \theta) \mapsto g^{\prime}_\theta (r)\,$ and  $\, (r, \theta) \mapsto g^{\prime \prime}_\theta (r)\,$ are \textsc{Borel} measurable; and 

 \smallskip
\noindent  (iv)     $\,\, \,\sup_{\, 0 <r<K  \atop \theta \in [0, 2 \pi)} \big( | g^{\prime}_\theta (r)|+ | g^{\prime \prime}_\theta (r)| \big) < \infty\,$ holds for all $\, K \in (0, \infty)\,$.  

\smallskip 

\noindent Here we consider \textsc{Borel} sets with respect to the  Euclidean topology.  We introduce also   the subclasses
\begin{equation} 
\label{eq: Dmu} 
\mathfrak D^{\bm \mu} \, :=\,  \Big \{ \, g \in \mathfrak D \, : \, \int_{0}^{2\pi} g^{\prime}_\theta (0+) \, {\bm \nu} ({\mathrm d} \theta) \, =\,  0 \,\Big\}\,, \quad \mathfrak D^{\bm \mu}_{+} \, :=\,  \Big \{ \, g \in \mathfrak D \, : \, \int_{0}^{2\pi} g^{\prime}_\theta (0+) \, {\bm \nu} ({\mathrm d} \theta) \, \ge\,  0    
\,\Big\}\,.
\end{equation}
\end{definition}

 \begin{definition}
  \label{def: G}
For every given function  $ \, g :    \mathbb R^{2}   
\rightarrow \R\,$ in   $\,\mathfrak{D} $ we    set  by analogy with (\ref{primes}):  
 $\, G^{\prime} (x)  :=  g^{\prime}_\theta (r)\,$ and  $ \, G^{\prime \prime} (x)  :=  g^{\prime \prime}_\theta (r)\,$ for $\, x=(r, \theta)\,$ with $\,  r>0\,.$    
\end{definition}

   With this notation in place, we can formulate our second major result. 
   
     \newpage

 \smallskip
 \begin{thm} 
\label{Gen_FS} {\bf A Generalized \textsc{Freidlin-Sheu} Formula:}     With the above notation, every continuous  semimartingale $\,X(\cdot)  =  (X_{1}(\cdot), X_{2}(\cdot))^{\prime}\,$ which solves the system of equations (\ref{eq: skewTanaka}) and satisfies the properties (\ref{length}) and  (\ref{eq: localDist}),  also 
satisfies for every $\, g \in \mathfrak{D}\,$  the generalized \textsc{Freidlin-Sheu} identity  
\[
g (X(\cdot)) = g (\mathrm{x}) \,+ \int_0^{\, \cdot} \mathbf{ 1}_{\{ X(t) \neq \mathbf{ 0}\}} \left( G^{\prime} (X(t)) \,    \mathrm{d} S(t)  + \frac{1}{\,2\,} \, G^{\prime \prime} (X(t)) \, \mathrm{d} \langle S \rangle (t) \right) +\Big(\int_{0}^{2\pi} g^{\prime}_\theta (0+) \,  {\bm \nu}({\mathrm d} \theta) \Big) L^S(\cdot) 
\]
\begin{equation}
\label{GenFS1}
 \,\,\,\,\, \,\,\,\,\, = g (\mathrm{x}) \,+ \int_0^{\, \cdot} \mathbf{ 1}_{\{ X(t) \neq \mathbf{ 0}\}} \left( G^{\prime} (X(t)) \,    \mathrm{d} U(t)  + \frac{1}{\,2\,} \, G^{\prime \prime} (X(t)) \, \mathrm{d} \langle U \rangle (t) \right) +\Big(\int_{0}^{2\pi} g^{\prime}_\theta (0+) \,  {\bm \nu}({\mathrm d} \theta) \Big) L^S(\cdot) \,.
\end{equation}
\end{thm}
In particular, the continuous semimartingale $\,X(\cdot) \,$ of Theorem \ref{prop: skewTanaka} satisfies  
(\ref{GenFS1}).   

\subsection{Slope-Averaging Martingales}  

For any given  bounded, measurable    $\,\varphi : [0, 2\pi) \to \mathbb R \,$, let us define the functions
\begin{equation} 
\label{eq: gh} 
  h_{(\varphi)} (x) \, :=\,  \big(\varphi (\text{arg} (x)) - \mathbb E[ \varphi ( \text{arg}({\bm \xi}_{1}))] \big) \cdot {\bf 1}_{\{x \neq {\bm 0}\}}\,,\qquad \,\,\,g_{(\varphi)}(x) \, :=\, \lVert x \rVert \cdot h_{(\varphi)}(x)     
\end{equation}
for $\,x \in \mathbb R^{2}\,$, where $\,{\bm \xi}_{1}\,$ is an $\,{\mathfrak S}\,$-valued random variable with distribution $\,{\bm \mu}\,$ as in (\ref{eq: exp}). Such functions were first introduced  by \textsc{Barlow, Pitman \& Yor} (1989),  in their study of the \textsc{Walsh} Brownian motion. Using polar co\"ordinates, we observe that   $\, g_{(\varphi)} (x) \equiv  g_{(\varphi)}(r, \theta) \,$ belongs to the class $\, \mathfrak{D}\, $ and    satisfies
\[
G'_{(\varphi)} (x) \,\equiv \,\big(g_{(\varphi)}\big)_{\theta}^{\prime}(r) \, =\, h_{(\varphi)} (r, \theta)\, , \qquad G''_{(\varphi)} (x)\, \equiv \,\big(g_{(\varphi)}\big)_{\theta}^{\prime\prime}(r) \, =\, 0 \, , \qquad 
\int^{2\pi}_{0} h_{(\varphi)}(r, \theta) \, {\bm \nu} ( {\mathrm d} \theta ) \, =\,  0 \, .
\]
Here $\, ^{\prime}\,$ denotes differentiation with respect to $\,r\in(0,\infty)\,$. 

Direct application   of Theorem \ref{Gen_FS} gives the following result. 

 \begin{prop} 
\label{prop: MART} 
Assume that $ U(\cdot) $ in (\ref{eq: U}) is a continuous local martingale, and  construct its  \textsc{Skorokhod} reflection $  S(\cdot) $ as in (\ref{eq: S}). Consider any continuous  semimartingale $\,X(\cdot)  :=  (X_{1}(\cdot), X_{2}(\cdot))^{\prime}\,$ which satisfies the system of equations (\ref{eq: skewTanaka}), along  with  the properties (\ref{length}) and  (\ref{eq: localDist}). 
 
\smallskip
\noindent
(i)\,  For any  function $\, g : \mathbb R^{2} 
\rightarrow \mathbb R\,$ in the class $\, \mathfrak{D}^{\bm \mu}\,$ of (\ref{eq: Dmu}), 
the process 
$$ 
g(X(\cdot)) - g(X(0))  
- \frac{1}{\,2\,}\int_0^{\, \cdot } G^{\prime \prime} (X( t)) \,\mathbf{ 1}_{ \{ X( t) \neq \mathbf{ 0}\}}\, {\mathrm d} \langle U \rangle (t) 
$$
  is a continuous local martingale  with quadratic variation   \,$\,   \int_0^{\, \cdot } \big( G^{\prime  } (X( t)) \big)^2\,\mathbf{ 1}_{ \{ X( t) \neq \mathbf{ 0}\}}\, {\mathrm d}  \langle U \rangle (t) \,$.

\noindent
(ii)\, For any given  bounded, measurable function $\,\varphi : [0, 2\pi) \to \mathbb R \, $ and with the notation of (\ref{eq: gh}),  the process  below   is a continuous local martingale:
 $$\,
g_{(\varphi)}\big(X(\cdot)\big) \, =\, \big| \big| X(\cdot) \big| \big| \, h_{(\varphi)}(X(\cdot))    \, =\,  g_{(\varphi)}( \mathrm{x} ) + \int^{\,\cdot}_{0}  h_{(\varphi)}(X(t)) \, {\mathrm d} U(t) \,.
$$       
   \end{prop}

 \section{The Proofs of Theorems \ref{prop: skewTanaka}, \ref{Gen_FS} and  \ref{IFF}}
 \label{Pf}

The way we construct a process $X(\cdot)$ 
which  satisfies the equation (\ref{eq: skewTanaka}) is via ``folding and unfolding of semimartingales",  with additional randomness   coming from a sequence $\, {\bm \xi}_{1}, {\bm \xi}_{2}, \ldots \,$ of $\,\mathfrak S-$valued, I.I.D. random variables. These have common probability distribution $\, {\bm \mu}   \,$ on $\,\mathfrak S\,$, such that the components of the random vector $\, {\bm \xi}_{1} \, :=\, \big(\xi_{1,1}, \xi_{1,2}\big)^{\prime} \,$ have expectations that are matched with the parameter vector $\, \big(\alpha_{1}^{(+)}, \alpha_{1}^{(-)}, \alpha_{2}^{(+)}, \alpha_{2}^{(-)}\big) \in [0,1]^{4}\,$ in  (\ref{alpha}), (\ref{eq: skewTanaka}) as  
 \begin{equation} 
\label{eq: exp}
\, \mathbb E\big( \xi_{1,i}^{\pm} \big) \, =\,  \alpha_{i}^{(\pm)}\,, \quad \mathbb E\big( \xi_{1,i} \big) \, =\,  \alpha_{i}^{(+)} - \alpha_{i}^{(-)} \,=\, \gamma_i\,, \quad \mathbb E\big( \lvert \xi_{1,i} \rvert \big) \, =\,  \alpha_{i}^{(+)} + \alpha_{i}^{(-)} \, ; \quad i \, =\, 1, 2\,.  
\end{equation}
  \newpage

 \noindent
 {\bf Proof of Theorem \ref{prop: skewTanaka}:}
For simplicity, we consider the case $ \,\mathrm{x}_{1} = \mathrm{x}_{2} = 0\,$ first. Following \textsc{Prokaj} (2009) and \textsc{Ichiba \& Karatzas} (2014), we   enlarge the original probability space by means of the above  sequence $ \big\{ {\bm \xi}_{k}  \big\}_{k \in \mathbb N}\,$ of $\,\mathfrak S\,$-valued, I.I.D.  random variables. These are independent of the $\, \sigma-$algebra $\, \widetilde{\mathcal F}(\infty)   :=  \bigvee_{0 \le t < \infty} \widetilde{\mathcal F}(t)\,$ and  have expectation $\, \mathbb{E} ({\bm \xi}_1) = {\bm \gamma}\,$ as in (\ref{eq: exp}),  (\ref{gamma}).

\medskip
\noindent
$\bullet~$ 
Let us decompose the nonnegative half-line into the zero set $\,\mathfrak Z     \,$ of $\,S(\cdot) \,$ as in (\ref{eq: Z_{S}}) on the one hand,   and the countable collection $\, \{\mathcal C_{k}\}_{k \in \mathbb N} \,$ of open disjoint components   of $\,[0, \infty) \setminus \mathfrak Z \,$ on the other. Each of these components represents an excursion  interval away from the origin for the \textsc{Skorokhod} reflection process $\,S(\cdot)\,$ in (\ref{eq: S}). Here we enumerate these countably-many   excursion intervals $\, \{\mathcal C_{k}\}_{k \in \mathbb N} \,$  in a measurable manner, so that $\,\{ t \in \mathcal C_{k}\} \in \widetilde{\mathcal F}(\infty)\,$ holds for all $\, t \ge 0\,$, $\,k \in \mathbb N\,$. For notational simplicity, we declare also  $\,\mathcal C_{0}  :=  \mathfrak Z\, , \,\,\,\, {\bm \xi}_{0}  :=  \mathbf{ 0}\, .$  We shall  denote 
\begin{equation} 
\label{eq: ZX}
Z(t) \, :=\,  \sum_{k\in \mathbb{N}_0}  \, {\bm  \xi}_{k} \cdot {\bf 1}_{\mathcal C_{k}}(t) \, , \qquad X(t) \, :=\, Z(t) S(t) \,, \qquad \mathcal F^{Z}(t) \, :=\, \sigma( Z(s) , 0 \le s \le t) \,  
\end{equation}
for $\, 0 \le t < \infty \,$ and introduce the enlarged filtration $\, {\mathbb F}   :=  \big\{ {\mathcal F}(t) \big\}_{0 \le t < \infty} \, $ via  $\,  {\mathcal F}(t)   :=  \widetilde{\mathcal F}(t) \vee \mathcal F^{Z}(t) \, $.

 \smallskip
This procedure corresponds exactly to the program   outlined by \textsc{J.B. Walsh} in the appendix to his 1978 paper, as follows: {\sl ``The idea is to take each excursion of (reflecting Brownian motion) and, instead of giving a random sign, to assign it a random variable with a given distribution in $[0,2 \pi)$, and to do so independently for each excursion"}. Of course the process $\,S(\cdot)\,$ is one-dimensional, while  $\,Z(\cdot)    = \big(Z_{1}(\cdot), Z_{2}(\cdot)\big)^{\prime}\,$  and $\, X(\cdot)   =  \big(X_{1}(\cdot), X_{2}(\cdot)\big)^{\prime} \,$  are two-dimensional processes with $\,\,\mathfrak f(X(\cdot)) = \big(\mathfrak f_{1}(X(\cdot)), ~\mathfrak f_{2}(X(\cdot))\big)^{\prime}= Z(\cdot) \,$: 
\begin{equation} 
\label{eq: fXZ}
\mathfrak f_{i}(X(\cdot)) \, =\,  Z_{i}(\cdot)\, , \quad i \, =\,  1, 2  
\end{equation}
with the functions $\, \mathfrak f_{i}(\cdot)\,$    as defined in (\ref{eq: f}). In particular, the zero set of (\ref{eq: Z_{S}}) is
\begin{equation} 
\label{eq: Z_{S}a}
 \mathfrak Z \,= \, \big\{t \ge 0: S(t)= 0\big\} = \big\{t \ge 0: Z(t) = {\bm 0} \big\} = \big\{t \ge 0: X(t) =  {\bm 0} \big\} \, .  
\end{equation} 
We can also think of the vector process $  X(\cdot)  
 $ of (\ref{eq: ZX}) as expressed in its polar co\"ordinates  
 \begin{equation} \label{eq: Theta}
    S(t) = \sqrt{X_{1}^{2}( t) + X_{2}^{2}( t)\,} \qquad \text{and} \qquad \Theta ( t) =\text{arg} \big(Z( t)\big) = \sum_{k \in \mathbb N } \text{arg} ( {\bm \xi}_{k} ) \cdot {\bf 1}_{\mathcal C_{k}}( t) \,   
    \end{equation} 
    from (\ref{eq: ZX}). We   shall see presently that this process $\, X(\cdot)\,$ satisfies the system of equations  (\ref{eq: skewTanaka}).   

\medskip
\noindent
$\bullet~$ 
We claim that, because of independence and of the way the probability space was enlarged, {\it both processes $\,U(\cdot) \,$ and $\,S(\cdot) \,$ are continuous $\,{\mathbb F}-$semimartingales}. This claim can be established as in the proof  of Proposition 2 in \textsc{Prokaj} (2009); see also  Proposition 3.1  in \textsc{Ichiba \& Karatzas} (2014).

\medskip
\noindent
$\bullet~$ 
In order to describe the dynamics of the process $\,X(\cdot)\,$ defined in (\ref{eq: ZX}), we approximate the process $\,Z(\cdot)\,$ also defined there  by a family of processes $\,Z^{\varepsilon}(\cdot)\,$ with finite first variation over compact intervals indexed by $\,\varepsilon \in (0,1)$,  as  follows.  We  define the sequence of stopping times $\,\tau_{0}^{\varepsilon} : = \inf \big\{ t\geq 0 : \Vert X(t) \Vert = 0 \big\}$  and 
\begin{equation} 
\label{eq: tau rec}
\tau^{\varepsilon}_{2\ell+1} \, :=\,  \inf \big\{ t > \tau^{\varepsilon}_{2\ell} : \Vert X(t) \Vert \geq \varepsilon \big\} \, , \quad \tau^{\varepsilon}_{2\ell+2} \, :=\,  \inf \big\{ t > \tau^{\varepsilon}_{2\ell+1} : \Vert X(t) \Vert \, =\, 0 \big\} \, ; \quad \ell \in \mathbb N_{0} 
\end{equation}
recursively. We also introduce a piecewise-constant process $\, Z^{\varepsilon} (\cdot) \, :=\, \big(Z^{\varepsilon}_{1}(\cdot), Z^{\varepsilon}_{2}(\cdot)\big)^{\prime}\,$ with 
\begin{equation} 
\label{eq: Ze}
Z^{\varepsilon}(t) \, :=\,  \sum_{\ell \in \mathbb N_{0}} \,Z(t) \,{\bf 1}_{[\tau^{\varepsilon}_{2\ell +1}, \tau^{\varepsilon}_{2\ell +2})}(t)  \,  =\,  \sum_{(k, \ell) \in \mathbb N_{0}^2} \,{\bm \xi}_k  \,{\bf 1}_{\mathcal{C}_k \cap [\tau^{\varepsilon}_{2\ell +1}, \tau^{\varepsilon}_{2\ell +2})}(t)\, , \qquad 0 \le t < \infty \, , 
\end{equation}
i.e.,  constant on each of the ``downcrossing intervals" $\,[\tau^{\varepsilon}_{2\ell +1}, \tau^{\varepsilon}_{2\ell +2})\,$. For this process, the product rule gives 
\begin{equation} 
\label{eq: ZeS}
X^{\varepsilon}(T) \, :=\, Z^{\varepsilon}(T) S(T) \, =\,  \int^{T}_{0} Z^{\varepsilon}(t) {\mathrm d} S(t) + \int^{T}_{0} S(t) {\mathrm d} Z^{\varepsilon}(t) \, , \qquad 0 \le T < \infty \, . 
\end{equation}
  \newpage
Passing to the limit  as $\,\varepsilon \downarrow 0\,$  and using (\ref{eq: exp})-(\ref{eq: fXZ}), as well as the characterization of the local time $\,L^{S}(\cdot) \,$ of the semimartingale $\,S(\cdot)\,$ in terms of  the number of its downcrossings, we obtain the decomposition   
\begin{equation} 
\label{eq: XZS} 
X(T) \, =\, Z(T) S(T)  \, =\,  \int^{T}_{0} Z(t)\, {\mathrm d} S(t) + \mathbb E[ {\bm \xi}_{1}] \,L^{S} (T) \, =\, \int^{T}_{0} \mathfrak f (X(t)) \,{\mathrm d} S(t) + {\bm \gamma} \, L^{S} (T)  \, 
\end{equation}
in the notation of (\ref{gamma}).    Indeed,  the second term on the right-hand side of (\ref{eq: ZeS}) can be   estimated by the strong law of large numbers and Theorem VI.1.10 in \textsc{Revuz \& Yor} (1999): namely, we have   
\begin{equation} \label{eq: SZe}
\begin{split}
\int^{T}_{0} S(t) {\mathrm d} Z^{\varepsilon}(t) \, &=\, \sum_{\{\ell : \,\tau^{\varepsilon}_{2\ell + 1} < T\}} S(\tau^{\varepsilon}_{2 \ell + 1}) Z^{\varepsilon}(\tau^{\varepsilon}_{2\ell + 1}) \, =\, \varepsilon \sum_{j=1}^{N(T, \varepsilon)} {\bm \xi}_{\ell_{j}} + O(\varepsilon)\\ \, &=\, \varepsilon \, N(T, \varepsilon) \cdot \frac{1}{\, N(T, \varepsilon)\, }  \sum_{j=1}^{N(T, \varepsilon)} {\bm \xi}_{\ell_{j}}+O(\varepsilon) \, \xrightarrow[\varepsilon \downarrow 0]{} \,  L^{S}(T) \cdot \mathbb E [ {\bm \xi}_{1}] \,  
\end{split}
\end{equation}
in probability. Here $\,\big\{ {\bm \xi}_{\ell_{j}}\big\}_{j=1}^{N(T, \varepsilon)}\,$ is an enumeration of   $\,Z^{\varepsilon}(\tau^{\varepsilon}_{2 \ell + 1})\,$, and    $\,N(T, \varepsilon)   :=  \sharp \big\{ \ell \in \mathbb{N}: \tau^{\varepsilon}_{2 \ell } < T \big\}\, $    the number of downcrossings of the interval $\,(0, \varepsilon)\,$ that the process  $\,S(\cdot)\,$ has completed  during $\,[0,T)\,$. We deduce from (\ref{eq: XZS}), in particular,  that the process $\,X(\cdot)\,$ is a continuous planar $\,{\mathbb F}-$semimartingale.

\smallskip
By analogy with (\ref{eq: XZS}), we can approximate the process $\, \lvert Z_{i} (\cdot) \rvert \,$ by $\, \lvert Z_{i} ^{\varepsilon}(\cdot) \rvert\,$, the absolute value of each of the components $\,Z_{i}^{\varepsilon}(\cdot)\,$ of the piecewise-constant process in (\ref{eq: Ze});  passing to the limit as $\, \varepsilon \downarrow 0\,$,  we obtain   
\begin{equation} 
\label{eq: |X||ZS|} 
\,\,\, \big \lvert X_{i}(T) \big \rvert \, =\, \big \lvert Z_{i}(T)  \big \rvert S(T)  \, =\,  \int^{T}_{0} \big  \lvert Z_{i}(t) \big \rvert \, {\mathrm d} S(t) + \mathbb E \big( \lvert {\xi}_{1, i} \rvert \big) L^{S} (T)  \, , \qquad 0 \le T < \infty   
\end{equation} 
for $\, i=1, 2\,$. We appeal now to  Exercise VI (1.16) $3^{o})$ of \textsc{Revuz \& Yor} (1999); recalling the form of $\,S(\cdot)\,$ in (\ref{eq: S}) along with   (\ref{eq: Z_{S}a})  we deduce that, 
with the normalization of (\ref{eq: LT}), the continuous, nonnegative semimartingale $\, \lvert X_{i}(\cdot) \rvert\,, \,\, i=1, 2\, $ with the decomposition (\ref{eq: |X||ZS|}) has local time at the origin 
\begin{equation} 
\label{eq: L|X|}
 L^{ \lvert X_{i}\rvert}(\cdot) \, =\,\int_0^{\,\cdot} \mathbf{ 1}_{\{ X_i (t) =0\}}\,\big[\, \big| Z_i (t)\big|\,     \mathrm{d}  S(t)   + \big(\alpha_{i}^{(+)} + \alpha_{i}^{(-)}\big) \, \mathrm{d} L^{S}(t)\, \big]
 \, =\, \big(\alpha_{i}^{(+)} + \alpha_{i}^{(-)}\big) \, L^{S}(\cdot)\,. 
 \end{equation}

\smallskip
\noindent
$\bullet~$
At this point,  we need to identify the local times $\,L^{X_{i}}(\cdot)\,$ of each component $\,X_{i}(\cdot)\,$  in terms of the local time $\,L^{S}(\cdot)\,$. Since $\,X_{i}(\cdot) \, =\, Z_{i}(\cdot) S(\cdot)\,$ is a continuous semimartingale for $\,i \, =\, 1, 2\,$, we recall the decomposition (\ref{eq: XZS}) and properties of semimartingale local time and  obtain  the string of identities 
\[
\,2 \, L^{X_{i}}(\cdot) - L^{ \lvert X_{i}\rvert}(\cdot) \, =\,  \int^{\cdot}_{0} {\bf 1}_{\{X_{i}(t) \, =\, 0\}} \,{\mathrm d} X_{i}(t) \, =\, \int^{\cdot}_{0} {\bf 1}_{\{X_{i}(t) \, =\, 0\}} \big[ Z_{i}(t) {\mathrm d} S(t) + \mathbb E ( \xi_{1,i}) {\mathrm d} L^{S}(t) \big]
\]
\begin{equation} 
\label{eq: diff}
\qquad \qquad   
\qquad \qquad \qquad \qquad  
\qquad \qquad \, =\, \mathbb E ( \xi_{1,i})  \, L^{S}(\cdot) \, =\, \big(\alpha_{i}^{(+)} - \alpha_{i}^{(-)}\big) \, L^{S}(\cdot)  
\end{equation}
 (cf. subsection 2.1 in \textsc{Ichiba, Karatzas \& Prokaj} (2013)).  Thus, combining with (\ref{eq: L|X|}),   we deduce
\begin{equation} 
\label{eq: LXLS}\,\,
2 \, L^{X_{i}}(\cdot) \, =\, \big( \mathbb E ( \lvert \xi_{1,i}\rvert )+ \mathbb E ( \xi_{1,i} ) \big) L^{S}(\cdot) \, =\, 2 \,\mathbb E( \xi_{1,i}^{+}) \,L^{S}(\cdot) \, =\, 2\, \alpha_{i}^{(+)} L^{S}(\cdot)\,, \quad i \, =\, 1, 2\,,   
\end{equation}
\noindent
i.e., property (\ref{eq: LTXS}). The equations (\ref{eq: skewTanaka}) and   (\ref{length}), (\ref{eq: zeroLeb}) follow now from (\ref{eq: SC}), (\ref{alpha}), (\ref{eq: Theta}) and (\ref{eq: XZS}).

\smallskip
\noindent
$\bullet~$
The property (\ref{eq: localDist}) can be shown by an approximation   similar in spirit and manner to that just carried out in the proof of (\ref{eq: LTXS}). We take now throughout the ``thinned" sequence $\,{\bm \xi}_{k}^{A} \, :=\,  
{\bf 1}_A ( \text{arg}({\bm \xi}_{k}) )\,$, $\,k \in \mathbb N_{0}\, $  in place of $\,{\bm \xi}_{k} \,$, $\,k \in \mathbb N_{0}\,;$ and in lieu of $\,Z(\cdot)\,$ and $\,S(\cdot)\,$   in (\ref{eq: ZX}), respectively, the processes  
$$\,
Z^{(A)}(\cdot) \, :=\, \sum_{k\in \mathbb N_{0}} \, {\bm \xi}_{k}^{A}\, \cdot  {\bf 1}_{\mathcal C_{k}}(\cdot)
\qquad \text{and} \qquad \,
R^{A}(\cdot) \, =\, \big \lVert X (\cdot) \big \rVert \cdot {\bf 1}_A \big(\text{arg}(X(\cdot)\big)   \, =\, S(\cdot) \,Z^{(A)}(\cdot)  \,.
$$
  \noindent
 $\bullet~$ 
When the initial value $\, \mathrm x \, =\, (\mathrm{x}_{1}, \mathrm{x}_{2})^{\prime}\,$  is not the origin, we define 
\begin{equation*} 
\label{reduced2}
 Z(0) \, :=\,  {\mathfrak f}(X(0)) \, =\, {\mathfrak f}( {\mathrm x})\qquad \text{and}\qquad X(t) \, :=\,   
 Z(0) S(t) \,, \quad \text{for}\,\,\,\, 0 \le t < \tau (0) \,,
 \end{equation*}
   \newpage
   \noindent
  i.e., until the process $\,X(\cdot) \,$ first attains the origin, very much in accordance with (\ref{eq: skewTanaka_flat}). Here $\, \tau(0) \,$ is defined as in Theorem \ref{prop: skewTanaka}(i). The so-constructed process $\,X(\cdot)\,$ satisfies the stochastic differential equation (\ref{eq: skewTanaka_flat}), to which (\ref{eq: skewTanaka}) reduces on the interval $\,[0, \tau(0))\,$ as in the discussion at the start of  section \ref{disc}. On  the interval  $\,[ \tau(0), \infty)\,$  we use the recipe (\ref{eq: ZX}) above, to construct  $\,X(\cdot)\,$ starting from the origin.

With these considerations  we obtain $\, \{ {\mathfrak f}(X(s)) =   {\mathfrak f}({\mathrm x}) , \, \, s < \tau(0) \} = \{ s< \tau (0) \} \, $,  mod. $\mathbb P$,   and hence we verify (\ref{eq: DIST3}).  Moreover, for every $\, (s,t) \in (0, \infty)^{2} \,$, there exists  by   construction an $\, \widetilde{\mathcal F} (\infty)-$measurable random index $\, \kappa_{0} (s,t) : \Omega \rightarrow \N \,$ such that we have, either $\, \tau (s) + t \in \mathcal C_{\kappa_{0} (s,t)}\,$,   or $\, \tau(s) + t \in \mathfrak Z \,$   on $\, \{ \tau (s) < \infty\} \,$. If $\,\tau(s) + t \in \mathfrak Z\,$ and $\,\tau(s) < \infty \,$, then $\, \mathfrak f ( X(\tau(s) + t)) \, =\, {\bm 0 } \,$. By (\ref{eq: SC}) and the construction of $\,X(\cdot) \,$ we obtain $\,\mathbb P ( \mathfrak f (X(\tau(s) + t)) = {\bm 0} ) \, =\, \mathbb  P ( S(\tau(s) + t) = 0 ) \, =\,  0 \,$ for a.e. $\,t \in (0, \infty)\,$. Therefore,  
\[
\big\{{\mathfrak f}(X(\tau (s) + t)) \in B,\,\, \tau (s)< \infty \big\} =   \Big\{ \sum_{k \in \mathbb N_{0}} {\bm \xi}_{k} \, {\bf 1}_{\mathcal C_{k}}( \tau (s) + t) \in B \, , \,\,\tau (s)< \infty \Big\}
=  \big\{ {\bm \xi}_{\kappa_{0} (s,t)} \in B \, , \,\,\tau (s) < \infty  \big\}   
\]
\noindent
holds mod. $\mathbb P\, $ for every   $\, B \in \mathcal B (\mathfrak S) \,$ and almost every $\,t \in (0, \infty)\,$. We conclude that (\ref{eq: DIST}) holds, namely  
\[
\mathbb P
\big(\, {\mathfrak f}(X(\tau  (s) + t)) \in B\,\big| \, {\cal F}^{X} (\tau(s)) \big) \, =\, \mathbb P 
\big( {\bm \xi}_{\kappa_{0}(s, t)} \in B\,\big| \, {\cal F}^{X} (\tau(s)) \big) \]
\[
\, =\, \mathbb E \Big [ \mathbb P \big[  {\bm \xi}_{\kappa_{0}(s, t)} \in B \, \big \vert \, \widetilde{\mathcal F}(\infty) \vee \mathcal F^{Z}(\tau(s))  \big] \, \big \vert \, \mathcal F^{X}(\tau(s))   \Big] 
\,=\, \mathbb E
\big[ \, \mathbb P
( {\bm \xi}_{1} \in B )  \,\big| \, {\cal F}^{X} (\tau(s)) \,\big] \, =\,  {\bm \mu} (B) \, ,  
\]
for every $\,s \in (0, \infty)\,$, $\, B \in {\mathcal B}({\mathfrak S})\,$ and almost every $\,t \in (0, \infty)\,$. We have used here the definitions of $\, \mathcal F^{X}(\cdot) \subseteq \widetilde{\mathcal F} (\cdot) \vee \mathcal F^{Z}(\cdot)\,$ and $\, Z(\cdot) \,$ in (\ref{eq: ZX}), the $\, \widetilde{\mathcal F}(\infty)-$measurability of the stopping time $\,\tau (s)  $ and of the random index $\, \kappa_{0} (s,t) \,$, and  the independence between $\, \widetilde{\mathcal F}(\infty) \,$ and the sequence   $  \big\{{\bm \xi}_{k} \big\}_{ k \in \mathbb N}\,$ of I.I.D. random variables. 
This completes the proof of Theorem \ref{prop: skewTanaka}. \qed

 \medskip
 \noindent
 {\bf Proof of Theorem \ref{Gen_FS}:} Let us fix a function  $\, g:    \mathbb R^{2}  \rightarrow \R\,$ in the class $\,\mathfrak{D}\,$ as in the statement of the theorem,   and recall  the notation established in Definitions \ref{def: D}, \ref{def: G}. Consider also a continuous planar semimartingale $\,X(\cdot) \,$ satisfying the equations of (\ref{eq: skewTanaka}) along with the properties  (\ref{length}) and  (\ref{eq: localDist}).   With $\,\{\tau_{k}^{\varepsilon}\}_{k \in \mathbb N_{0}}\,$   defined as in (\ref{eq: tau rec}), and with $\, \tau_{-1}^{\varepsilon} \equiv 0\,$ and $\, \mathbb N_{-1} :=  \mathbb N_{0} \cup \{-1\}\,$, the value $\,g (X(T)) \,$ is decomposed into 
\begin{equation} 
\label{eq: gXe}
g(X(T)) = g({\mathrm x}) + \sum_{\ell \in \mathbb N_{-1}} \big ( g(X(T \wedge \tau^\varepsilon_{2\ell+2})) - g(X(T \wedge \tau^\varepsilon_{2\ell+1})) 
\big) 
+ \sum_{\ell \in \mathbb N_{0}}   \big ( g(X(T \wedge \tau^\varepsilon_{2\ell+1})) - g(X(T \wedge \tau^\varepsilon_{2\ell})) \big) \, .
\end{equation}

\noindent
$\bullet~$ We recall from the discussion at the beginning of section \ref{disc}, that the  process $X(\cdot)$ moves along the ray that connects $\,{\bm 0}\,$ to the starting point $\,\mathrm x \neq {\bm 0}\,$, during the time-interval $\,[0, \tau (0))=[\tau_{-1}^{\varepsilon}, \tau_{0}^{\varepsilon})\,$. 
In a similar manner, the processes  $\, \mathfrak f_{i}(X(\cdot)) \,$ are constant on every interval $\, [\tau_{2 \ell+1}^{\varepsilon}, \tau_{2 \ell+2}^{\varepsilon})\, $ for $\, \ell \in \mathbb N_{0}\,$, $\,i \, =\, 1, 2\,$. 

The first summation in (\ref{eq: gXe}) can thus be rewritten as    
\[
\sum_{\ell \in \mathbb N_{-1}}  \big ( g(X(T \wedge \tau^\varepsilon_{2\ell+2})) - g(X(T \wedge \tau^\varepsilon_{2\ell+1}))  \big) \, =\, \sum_{\ell \in \mathbb N_{-1}}  \Big ( g_{\theta}(S(T \wedge \tau^\varepsilon_{2\ell+2})) - g_{\theta}(S(T \wedge \tau^\varepsilon_{2\ell+1})) \Big)\Big|_{\theta \, =\, \Theta 
(T \wedge \tau^\varepsilon_{2 \ell +1})} \,  
\]
\begin{equation*} 
\, =\, \sum_{\ell \in \mathbb N_{-1}} \int^{T\wedge \tau^{\varepsilon}_{2\ell+2}}_{T\wedge \tau^{\varepsilon}_{2\ell+1}} \Big( g^{\prime}_{\theta}(S(t)) \,  {\mathrm d} S(t) + \frac{1}{\, 2\, } \, g^{\prime\prime}_{\theta}(S(t)) \, {\mathrm d} \langle S\rangle (t) \Big) \Big|_{\theta \, =\, \Theta 
(t)}  
\end{equation*}
\[
~~~~~\, =\, \int^{T}_{0} \Big( \sum_{\ell \in \mathbb N_{-1}} {\bf 1}_{(\tau^\varepsilon_{2 \ell + 1}, \, \tau^\varepsilon_{2\ell+2}) }(t)\Big) \Big( G^{\prime}(X(t)) \,{\mathrm d} S(t) + \frac{1}{\, 2\, } \,G^{\prime\prime} (X(t)) \, {\mathrm d}   \langle S \rangle(t) \Big) .
\]
We have set here $\,\Theta(\cdot)  :=  \text{arg} (X(\cdot)) \,$,  and   applied \textsc{It\^o}'s rule (Problem 3.7.3 in \textsc{Karatzas \& Shreve} (1991)) to the process $\,g_{\theta}(S(\cdot)) \,$. Letting $\,\varepsilon \downarrow 0\,$, we obtain in the limit 
\[
\int^{T}_{0} {\bf 1}_{\{X(t) \neq \mathbf{ 0}\}} \Big( G^{\prime}(X(t)) \,  {\mathrm d} S(t) + \frac{1}{\, 2\, } \,G^{\prime\prime} (X(t)) \, {\mathrm d}   \langle S \rangle(t) \Big) \qquad \qquad  \qquad \qquad 
\]
\noindent
  \begin{equation}
\label{eq: gXe3}
 \qquad \qquad  \qquad \qquad
= \int^{T}_{0} {\bf 1}_{\{X(t) \neq \mathbf{ 0}\}} \Big( G^{\prime}(X(t)) \,  {\mathrm d} U(t) + \frac{1}{\, 2\, } \,G^{\prime\prime} (X(t)) \, {\mathrm d}   \langle U \rangle(t) \Big)\,. 
\end{equation}
\newpage

\noindent
$\bullet~$
For the second summation in (\ref{eq: gXe}), we observe $\, g({\bm 0}) = g_{\theta} (0) \,$ by definition and hence 
\begin{equation} 
\label{eq: gXe2}
\sum_{\ell \in \mathbb N_{0}}  \big ( g(X(T \wedge \tau^\varepsilon_{2\ell+1})) - g(X(T \wedge \tau^\varepsilon_{2\ell}) ) \big)   
\end{equation}
\[
  = \!\! \!\! \sum_{\{\ell \, : \, \tau^{\varepsilon}_{2\ell + 1} < T\}} \big( g_{\theta} ( S(\tau^{\varepsilon}_{2\ell+1}) )  - g_{\theta} (0) \big) \Big\vert_{\theta \, =\, \Theta(\tau^{\varepsilon}_{2\ell+1})} \, + O(\varepsilon) \, = \sum_{\{\ell \, : \, \tau^{\varepsilon}_{2\ell + 1} < T\}} \big( g_{\theta} ( \varepsilon)  - g_{\theta} (0) \big)\Big\vert_{\theta \, =\, \Theta(\tau^{\varepsilon}_{2\ell + 1})} + O(\varepsilon)
\]
\[
  =\, \sum_{\{\ell \, : \, \tau^{\varepsilon}_{2\ell + 1} < T\}} \Big( \varepsilon g^{\prime}_{\theta}(0+) +  \int^{\varepsilon}_{0} (\varepsilon - u) g_{\theta}^{\prime\prime}(u) {\mathrm d} u  \Big) \Big\vert_{\theta \, =\, \Theta(\tau^{\varepsilon}_{2\ell + 1})} + O(\varepsilon) \xrightarrow[\varepsilon \downarrow 0 ]{} \, L^{S}(T)    \int_{0}^{2\pi} g^{\prime}_{\theta}(0+) {\bm \nu}({\mathrm d} \theta)  
\]

\medskip
\noindent
in probability. Indeed,  by analogy with  (\ref{eq: SZe}) we can verify 
\[
\Bigg \lvert \sum_{\{\ell \, : \, \tau^{\varepsilon}_{2\ell + 1} < T\}}\Big(  \int^{\varepsilon}_{0} (\varepsilon - u) \,g_{\theta}^{\prime\prime}(u) {\mathrm d} u \Big) \bigg \vert_{\theta \, =\, \Theta 
(\tau^{\varepsilon}_{2\ell + 1})}   \Bigg| \, \le \,c \sum_{ \{\ell \, : \, \tau^{\varepsilon}_{2\ell + 1} < T\}} \varepsilon^{2} \, =\,  c \, \varepsilon \cdot \big( \varepsilon\, N(T, \varepsilon) + O(\varepsilon) \big) \, \xrightarrow[\varepsilon \downarrow 0]{} 0
\]
in probability, where $\,c \, :=\, \sup_{\theta \in \text{supp} ({\bm \mu})} \max_{0 \le u \le 1} \big(g^{\prime\prime}_{\theta}(u) \, / \, 2\big)< + \infty\,$ by assumption. 

 \smallskip
  \noindent
$\bullet~$
We also check that  for every $\, A \in \mathcal B ([0, 2\pi)) \, $ we have, on account of the property (\ref{eq: localDist})  for the process $\, R^{A}(\cdot) \, =\, \lVert X (\cdot) \rVert \, {\bf 1}_A ( \Theta(\cdot))\,$, the convergence
\[
\sum_{\{\ell \, :\, \tau^{\varepsilon}_{2\ell+1}  < T\}} \varepsilon \, {\bf 1}_{\{\Theta(\tau^{\varepsilon}_{2\ell+1}) \in A\}}  \, =\, \sum_{\{\ell \, :\, \tau^{\varepsilon}_{2\ell+1}  < T\}} S(\tau^{\varepsilon}_{2\ell+1}) \, {\bf 1}_{\{\Theta(\tau^{\varepsilon}_{2\ell+1}) \in A\}} \,= \!\! \sum_{\{\ell \, :\, \tau^{\varepsilon}_{2\ell+1}  < T\}}\!\! \lVert X (\tau^{\varepsilon}_{2\ell+1}) \rVert \, {\bf 1}_{\{\Theta(\tau^{\varepsilon}_{2\ell+1}) \in A\}} 
\]
\[
\, =\, \sum_{\{\ell \, :\, \widetilde{\tau}^{\varepsilon}_{2\ell+1}  < T\}} R^{A}( \widetilde{\tau}^{\varepsilon}_{2\ell+1}) \, =\,  \varepsilon \,\widetilde{N}(T, \varepsilon) +O(\varepsilon) \, \xrightarrow[\varepsilon \downarrow 0]{} \, L^{R^{A}}(T) =\, {\bm \nu} (A) \, L^{S}(T) \, 
\]
in probability. Here we define $\, \widetilde{\tau}^{\varepsilon}_{0} \, :=\, \inf \big\{t \geq 0  \, : \, R^{A}(t) \, =\, 0 \big\}\,$,  and recursively $$ \widetilde{\tau}^{\varepsilon}_{2\ell+1} \, :=\, \inf \big\{t >  \widetilde{\tau}^{\varepsilon}_{2\ell} \, : \, R^{A}(t) \geq \varepsilon \big\}\, , 
\qquad  \widetilde{\tau}^{\varepsilon}_{2\ell+2} \, :=\, \inf \big\{t >  \widetilde{\tau}^{\varepsilon}_{2\ell+1} \, : \, R^{A}(t) \, =\, 0 \big\}\,$$ for $\,\ell \in \mathbb N_{0}\,$,  and   denote by  $\, \widetilde{N}(\varepsilon, T)\,$ the number of downcrossings of the interval $\,(0, \varepsilon) \,$ that the process $\, R^{A} (\cdot)\,$ has completed during the interval $ [0, T) $  $($please note that  we count here the number   of downcrossings corresponding to the rays in the directions in the subset $\,A\,$ of $\,[0, 2\pi))$.  Thus,  approximating the function $\,\theta \mapsto  g^{\prime}_{\theta}(0+)\,$ by   indicators   $\, \theta \mapsto {\bf 1}_A (\theta )\,$,  $\,A \in \mathcal B([0, 2\pi))\,,$ we verify the convergence   
\begin{equation}
\label{007}
\sum_{\{\ell \, :\, \tau^{\varepsilon}_{2\ell+1}  < T\}} \varepsilon \, g^{\prime}_{\Theta(\tau^{\varepsilon}_{2\ell + 1})}(0+) \,\,\xrightarrow[\varepsilon \downarrow 0 ]{} \, \,L^{S}(T)\int_{0}^{2\pi} g^{\prime}_{\theta}(0+) {\bm \nu}({\mathrm d} \theta) \,, \quad \text{ in probability.}
\end{equation}
$\bullet~$
Therefore, the limit of the expression in (\ref{eq: gXe}) is the sum of the limits   of the expressions in (\ref{eq: gXe3}) and (\ref{eq: gXe2}), and we conclude   with 
\[
g(X(T)) = g({\mathrm x})  +  \Big(\int_{0}^{2\pi} g^{\prime}_{\theta}(0+) {\bm \nu}({\mathrm d} \theta)  \Big) L^{S}(T) + \int^{T}_{0} {\bf 1}_{\{X(t) \neq \mathbf{ 0} \}} \Big( G^{\prime}(X(t)) {\mathrm d} S(t) + \frac{1}{\, 2\, } G^{\prime\prime} (X(t))   {\mathrm d}   \langle S \rangle(t) \Big),
\]
which establishes (\ref{GenFS1})  and proves the first claim in Theorem \ref{Gen_FS}; the second   follows then   directly. \qed

\medskip
\noindent
{\bf Proof of Theorem \ref{IFF}:  (i)} Assume $\, \gamma_{1}^{2}+\gamma_{2}^{2}\leq 1\,$, and consider the vector  $\, {\bm \gamma}:= (\gamma_{1}, \gamma_{2})^{\prime} \in \R^2\,$. Then  we define the  probability measure   $\, {\bm \mu} := \big( (1+\beta) / 2 \big) \,\delta_{z_{0}}   + \big( (1-\beta) / 2 \big) \,\delta_{-z_{0}}  \,$ on $\, (\mathfrak S, \mathcal B(\mathfrak S))\,$ with $\, \beta := \Vert \bm \gamma \Vert \leq 1\,$ and $\, z_{0}:=   {\bm \gamma} / \beta  \in \mathfrak S\,$ provided that $\, \beta \neq 0\,$ (if $\, \beta=0\,$, we simply pick up an arbitrary $\, z_{0} \in \mathfrak S$), 
  and note 
\[
\int_{ \mathfrak S}  \mathfrak f (z) \,  {\bm \mu} (\mathrm{d} z) \, = \,\int_{ \mathfrak S}  z \,  {\bm \mu} (\mathrm{d} z)\, = \,\frac{1+\beta}{2}\, z_{0} + \frac{1-\beta}{2} \,(-z_{0})\, = \,\beta z_{0}\, =\, \bm \gamma \,.
\]
Thus, if we take the  process $\, S(\cdot) \,$ in section 2 as the ``folded driver"  and $\, \bm \mu \,$ as the ``spinning measure", Theorem 2.1   constructs a   continuous planar semimartingale    $X(\cdot)=  (X_{1}(\cdot), X_{2}(\cdot))^{\prime} $ that satisfies the condition (\ref{length})  and the system of equations (\ref{eq: skewTanaka})  -- thus also the system    (\ref{eq: skewTanaka_2}).

 \noindent
{\bf (ii)}  Suppose now that  (\ref{eq: Pos_Loc_Time})  holds, and that there exists a continuous   semimartingale    $X(\cdot)=  (X_{1}(\cdot), X_{2}(\cdot))^{\prime} $ which satisfies        (\ref{length}) and the system of equations (\ref{eq: skewTanaka_2}), thus also of (\ref{eq: skewTanaka}). For every $\, \varepsilon > 0\,$, we define    
$\,  \tau^{\varepsilon}_{-1}\equiv 0\,$ and $\, \big\{ \tau^{\varepsilon}_{m} \big\}_{m \in \N_0}\,$ as in (\ref{eq: tau rec}). 
Following the idea in the proof of Theorem \ref{Gen_FS}, we write 
\begin{equation*} 
X(T) = {\mathrm x} + \sum_{\ell \in \mathbb N_{-1}} \big ( X(T \wedge \tau^\varepsilon_{2\ell+2}) - X(T \wedge \tau^\varepsilon_{2\ell+1}) 
\big) 
+ \sum_{\ell \in \mathbb N_{0}}   \big ( X(T \wedge \tau^\varepsilon_{2\ell+1}) - X(T \wedge \tau^\varepsilon_{2\ell}) \big) \, .
\end{equation*}

\noindent
Then as $\, \varepsilon \downarrow 0 \,$,    on account of (\ref{eq: skewTanaka}) and  in the same manner as in the proof of Theorem \ref{Gen_FS}, the first summation in the above expression converges in probability to $\,\int^{T}_{0} \mathfrak f\big(X(t)\big) \,{\mathrm d} U(t)\,$. Thus, the second summation converges in probability to $\, \bm \gamma \, L^{ \,||X||}(T)\,$, thanks to (\ref{eq: skewTanaka_2}).  This  implies the convergence  in probability
$$ 
\sum_{\ell \in \mathbb N_{0}}   \big ( X(T \wedge \tau^\varepsilon_{2\ell+1}) - X(T \wedge \tau^\varepsilon_{2\ell}) \big) \,=\,\sum_{\ell =0}^{N(T, \varepsilon)-1} \varepsilon \, {\mathfrak f}(X(\tau_{2\ell+1}^{\varepsilon})) +O(\varepsilon)\,\xrightarrow[\varepsilon \downarrow 0] {}\, {\bm \gamma }L^{ \,||X||}(T)\,.
 $$ 
 We also have the convergence in probability $\, \varepsilon \, N(T, \varepsilon)  \rightarrow L^{\Vert X \Vert} (T)\, $ as $\,
 \varepsilon \downarrow 0\,$ by Theorem VI.1.10 in \textsc{Revuz \& Yor} (1999), where $\,N(T, \varepsilon) \, :=\, \sharp \big\{ \ell \in \mathbb{N} : \tau^{\varepsilon}_{2 \ell } < T \big\}\,$.   Therefore, 
  we have $$ \qquad \frac{1}{N(T, \varepsilon)} \sum_{\ell =0}^{N(T, \varepsilon)-1} {\mathfrak f} \big(X(\tau_{2\ell+1}^{\varepsilon})\big)\,\,\xrightarrow[\varepsilon \downarrow 0] {}\, \,{\bm \gamma }\, \quad \qquad \text{in probability, \, on the event $\, \big\{ L^{ \,||X||}(T) >0 \big\}\,. $}  $$
  Now $\, \Vert \bm \gamma \Vert \leq 1 \,$ follows from $\, \Vert \mathfrak f (\cdot)  \Vert \leq 1 \,$, since we can select a sufficiently large $T \in (0, \infty)$  such that $\,\mathbb{P} \big( L^{ \,||X||}(T) >0 \big) \,>\,0 $ (thanks to  (\ref{eq: Pos_Loc_Time}) and (\ref{length})).    \qed

\section{\textsc{Walsh} Diffusions and the Associated Martingale Problems}
\label{sec: MP}

 We cannot expect pathwise uniqueness, therefore neither can we expect  strength, to hold for the equations of (\ref{eq: skewTanaka}) or (\ref{eq: skewTanaka_2}). Any such lingering hope is dashed by the realization that, when $\,U(\cdot)\,$ is standard Brownian motion, thus $\,S(\cdot)\,$ a reflecting Brownian motion, the process $\,X(\cdot)\,$ constructed in Theorem \ref{prop: skewTanaka} is the \textsc{Walsh} Brownian motion -- a process whose filtration cannot be generated by {\it any} Brownian motion of {\it any} dimension. For this result see the celebrated  paper by \textsc{Tsirel'son} (1997), as well as $\,$Proposition \ref{prop: ID} below  and  \textsc{Mansuy \& Yor} (2006), pages 103-116. In light of these observations,  it is natural to ask whether the next best thing, that is, {\it uniqueness in distribution}, might hold for these equations under appropriate conditions. We try in this section to provide some affirmative answers to this question, when the folded driving semimartingale $\, S(\cdot)\,$ is a reflected diffusion; the main results appear in Propositions \ref{prop: MP1} and \ref{Cor: MP1}.

\subsection{The Folded Driving Semimartingale  as a Reflected Diffusion}   
\label{sec_511}

Let us start by  considering the canonical space $\,\Omega_1    :=  C([0, \infty); [0, \infty)) \,$ of nonnegative, continuous functions on $\,[0, \infty)\,$. We   endow this space with the usual topology of uniform convergence over compact intervals  and with the  $\sigma-$algebra $\, \mathcal F_1    :=  \mathcal B (\Omega_1 ) \,$ of its   \textsc{Borel} sets. We   consider also    the  filtration $\, \mathbb F_{1}   :=  \{\mathcal F_{1}(t)  \}_{0 \le t < \infty } \,$ generated by its co\"ordinate mapping, i.e.,   $\,\mathcal F_{1}(t) = \sigma \big( \omega_1 (s),\, \, 0 \le s \le t \big) \,$. 

Given  \textsc{Borel}-measurable co\"efficients $\,{\bm b} : [0, \infty) \to \mathbb R\,$ and $\,{\bm \sigma} : [0, \infty) \to \R \setminus \{ 0\}\,$ and setting $\,{\bm a} (\cdot) := {\bm \sigma}^{2}(\cdot)\,$, we define the process    
\begin{equation} 
\label{eq: MP1}
K^{\psi}  (\cdot\, ; \omega_1) \, :=\, \psi (\omega_{1}(\cdot)) - \psi (\omega_{1}(0)) -  \int^{\cdot}_{0} \mathcal G \psi (\omega_{1} (t))\cdot {\bf 1}_{\{ \,   \omega_{1} (t)  > 0 \}} \, {\mathrm d} t \, ,  
\end{equation}
 where
\[
\mathcal G  \psi (r) \, :=\, {\bm b} (r) \, \psi^{\prime}(r) + \frac{1}{\, 2\, } \,{\bm a} (r) \,\psi^{\prime\prime}(r) \, ; \, \, \, \, \,\,r \in [0, \infty)\, , \, \,\,\,\psi \in C^{2}_{0}\big( [0, \infty); \mathbb R\big) \, . 
\]

\subsubsection{Local Submartingale Problem for a Reflected Diffusion}

In the manner of \textsc{Stroock \& Varadhan} (1971), we formulate the  {\bf Local   Submartingale Problem associated with the pair   (${\bm \sigma}, {\bm b}$)\,}   as follows.

\smallskip
 {\it For every given $\,x \in [ 0, \infty) ,$ to find a   probability measure $\, {\mathbb Q}^{\bullet}\,$ on the     space $\,(\Omega_{1}, \mathcal F_{1}) \,$, under which: \\ (i) $\,\, \,\omega_{1}(0) \, =\,  x \,$ and $\, \,\int_0^\infty \mathbf{ 1}_{ \{ \omega_1 (t) =0\}} \, \mathrm{d}t =0\,\,$ hold $\,\mathbb Q^{\bullet}-\,$a.e.; and moreover, \\ (ii) \, for every function $\,\psi \in C^{2} ( [0, \infty); \mathbb R) \,$ with $\,\psi^{\prime}(0+) \ge 0\,$, the process $\, K^{\psi} (\cdot)\,$ is a continuous local submartingale, and    a continuous 
 local martingale whenever $\,\psi^{\prime}(0+) \, =\, 0\,,$ with respect to the filtration}  $\, \mathbb F_{1}^{\bullet} = \big\{ \mathcal F_{1}^{\bullet}(t)\big\}_{  0 \le t < \infty }\,$   with $\,  \mathcal F_{1}^{\bullet}(t) \, :=\, \mathcal F_{1}^{\circ}(t+) \,.$      
 
 \smallskip
Here we have   denoted by $\,\mathbb F_{1}^{\circ}:=\{ \mathcal F_{1}^{\circ}(t),  \, 0 \le t < \infty \}\,$   the augmentation of $\, \mathbb F_{1}\,$ under $\, {\mathbb Q}^{\bullet}\,$.  As usual, we shall say that this problem is {\it well-posed,} if it admits exactly one solution. For a recent study of the  well-posedness of submartingale problems for obliquely reflected diffusions, in domains with piecewise smooth boundaries, see \textsc{Kang \& Ramanan} (2014), where a random measure similar to (\ref{eq: RM}) is derived.

\subsection{A Local Martingale Problem for the Planar  Diffusion}   
\label{sec_512}

Next, we consider the canonical space $\,\Omega_{2}   := C([0, \infty); \mathbb R^{2}) \,$ of $\,\mathbb R^{2}-$valued continuous functions on $\, [0, \infty)\,$ endowed with the  $\sigma-$algebra $\, \mathcal F_{2}   :=  \mathcal B (\Omega_{2}) \,$ of its   \textsc{Borel} sets. We  consider also its co\"ordinate mapping  and the natural filtration $\, \mathbb F_{2} \, :=\, \{\mathcal F_{2}(t)  \}_{0 \le t < \infty } \,$ with $\,\mathcal F_{2}(t) = \sigma \big( \omega_2 (s),\, \, 0 \le s \le t \big) $.  We recall, here and in what follows, the Definitions \ref{def: D} and \ref{def: G}.

Given a  probability measure $\,{\bm \mu}\,$ on the   \textsc{Borel} subsets of the unit circumference $\,\mathfrak S\,$,  and  \textsc{Borel}-measurable functions $\,{\bm b} : [0, \infty) \to \mathbb R\,,$  $\,{\bm \sigma} : [0, \infty) \to \R \setminus \{ 0\}\,$ as in subsection \ref{sec_511},  we define  for every function $\,g \in \mathfrak D_{}\,$ the process 
\begin{equation} 
\label{eq: MP2}
M^{g} (\cdot\,;\omega_2) \, :=\, g(\omega_{2}(\cdot)) - g(\omega_{2}(0)) - \int^{\cdot}_{0} \mathcal L  \, g (\omega_{2} (t)) \cdot {\bf 1}_{\{ \, \lVert \omega_{2} (t) \rVert > 0 \}} \, {\mathrm d} t    \,, \qquad \,\,\,\text{where}
\end{equation}
\[
\mathcal L  \,g (x) \, :=\, {\bm b} ( \lVert x \rVert ) \, G^{\prime}(x) + \frac{1}{\, 2\, } \,{\bm a} ( \lVert x \rVert ) \,G^{\prime\prime}(x) \, ; \, \, \, \, \,\,\,x \in \mathbb R^{2}\,   .   
\]

\subsubsection{The Local Martingale Problem}
\label{sec: LMP}
 
Motivated by the generalized \textsc{Freidlin-Sheu} formula (\ref{GenFS1}) in Theorem \ref{Gen_FS}, we   formulate now the   {\bf Local Martingale Problem associated with the triple  (${\bm \sigma}, {\bm b}, {\bm \mu}$) }  as follows. 

\smallskip
{\it For every fixed $\, \mathrm x \in \mathbb R^{2}\,,$ to find a probability measure $\, {\mathbb Q} \,$ on  the canonical space $\,(\Omega_{2}, \mathcal F_{2}) \,$,  such that: \\ (i) $\, \omega_{2}(0) \, =\, \mathrm x \,$ holds $\, \mathbb Q-\,$a.e.; \\ 
(ii)     the analogue of the ``non-stickiness"  property (\ref{eq: zeroLeb}) holds, namely 
\begin{equation} 
\label{eq: zeroLeb2}
\, \int_0^\infty \mathbf{ 1}_{ \{ \omega_2 (t) \, =\, \mathbf{ 0}\}} \, \mathrm{d}t =0\,, \qquad \, {\mathbb Q}-a.e.;
\end{equation}
(iii) for every  function $\,g\,$ in   $\, \mathfrak D_{+}^{{\bm \mu}} $ $($respectively, $\, \mathfrak D^{{\bm \mu}} )$ as in (\ref{eq: Dmu}), the process $\,M^{g} (\cdot\,;\omega_2) \,$ of (\ref{eq: MP2}) is a continuous local submartingale (resp., martingale) with respect to the filtration  $\, \mathbb F_{2}^{\bullet}   :=  \{ \mathcal F_{2}^{\bullet}(t)  \}_{   0 \le t < \infty} \,$.   
 }  

  \medskip
Here we have set $\,\mathcal F_{2}^{\bullet}(t)  :=   \mathcal F_{2}^{\circ}(t+)\,$, and   denoted by $\,\mathbb F_{2}^{\circ}= \big\{ \mathcal F_{2}^{\circ}(t) \big\}_{0 \le t < \infty }\,$  the $\, {\mathbb Q}-$augmentation of the filtration $\, \mathbb F_{2}\,$.  Again,   this problem is called ``well-posed" if it   admits exactly one solution.

\smallskip
\noindent
$\bullet~$
The  theory  of the \textsc{Stroock \& Varadhan} martingale problem  is extended in Proposition \ref{prop: SDEMP} right below, for a  continuous planar semimartingale $X(\cdot)$ that satisfies the   properties    (\ref{eq: zeroLeb}), (\ref{eq: localDist}) and, with  co\"efficients $\,  \gamma_i\,,\, i= 1, 2\,$   given through (\ref{gamma}) and (\ref{alpha}), the system of stochastic integral equations  
\begin{equation} 
\label{eq: SDEMP}
X_{i}(\cdot) \, =\, X_{i}(0) + \int^{\,\cdot}_{0} {\mathfrak f}_{i}(X(t)) \Big[ {\bm b} \big( \lVert X(t) \rVert \big) {\mathrm d} t\, +\, {\bm \sigma} \big( \lVert X(t) \rVert \big) {\mathrm d} W(t) \Big] +   
\gamma_i\, L^{ \lVert X\rVert}(\cdot) \,, \quad i=1, 2\,.  
\end{equation}

\begin{prop} 
\label{prop: SDEMP}   {\bf Stochastic Equations for \textsc{Walsh} Diffusions:} 
 (a)\, For every weak   solution $\,(X, W)\,$, $\,( {\Omega}, {\mathcal F}, \mathbb P)\,$, $\, {\mathbb F}  = \{ {\mathcal F} (t) \}_{   0 \le t < \infty}\,$ to the system of stochastic equations   (\ref{eq: SDEMP}), we have   
 \begin{equation} 
\label{eq: ||X||} 
\lVert X(\cdot) \rVert \, =\,  \lVert X(0) \rVert + \int^{\cdot}_{0} {\bf 1}_{\{ \lVert X(t) \rVert > 0 \}} \, \Big( {\bm b} ( \lVert X(t) \rVert) {\mathrm d} t + {\bm \sigma} ( \lVert X(t) \rVert) {\mathrm d} W(t)\Big) + L^{||X||} (\cdot) \, ; 
\end{equation}
  and if this weak solution also satisfies the conditions (\ref{eq: zeroLeb})-(\ref{eq: localDist}), then it  induces a solution to the local martingale problem associated with the triple $\,({\bm \sigma}, {\bm b}, {\bm \mu})\,.$

 \smallskip
\noindent
(b)\, Conversely, every solution to the local martingale problem associated with the triple $\,({\bm \sigma}, {\bm b}, {\bm \mu})\, $ induces a weak   solution  to the system     (\ref{eq: SDEMP}) which satisfies the properties (\ref{eq: zeroLeb}), (\ref{eq: localDist}). The state process $X(\cdot)$ in  this weak solution solves also the system of stochastic equations (\ref{eq: skewTanaka}) with   ``folded driver" $\, S(\cdot)  = \lVert X(\cdot) \rVert \,$.  

\smallskip
\noindent
(c)\, Uniqueness holds  for the local martingale problem associated with   $\,({\bm \sigma}, {\bm b}, {\bm \mu}),$ if and only if uniqueness in distribution holds for the system of  (\ref{eq: SDEMP})   
subject to the conditions     
(\ref{eq: zeroLeb}), (\ref{eq: localDist}).  
\end{prop}

\noindent
{\it Proof of Part (a):}  We first validate (\ref{eq: ||X||}) for any weak solution to (\ref{eq: SDEMP}). From (\ref{eq: SDEMP}) we see 
\[
\int^{T}_{0}  \big( \vert {\mathfrak f}_{i}(X(t)) {\bm b} ( \lVert X(t) \rVert) \vert + {\mathfrak f}_{i}^{2} (X(t)) {\bm a} ( \lVert X(t) \rVert) \big)\, {\mathrm d} t  \,< \,\infty \, , \,\,\,\,\quad i=1,2\,, \,\, \,\,\,\, 0 \leq T < \infty\, .
\]
Since $\, {\mathfrak f}_{1}^{2}({\rm x}) + {\mathfrak f}_{2}^{2} ({\rm x}) =1 \,$ and $\, \vert {\mathfrak f}_{1}({\rm x}) \vert + \vert {\mathfrak f}_{2} ({\rm x})\vert \geq 1 \,$ hold for any $\, {\rm x} \in \mathbb R^{2} \setminus \{ \bm 0 \} \,$, we obtain then
\begin{equation}
\label{ineq: finite}
\int^{T}_{0} {\bf 1}_{\{ || X (t) ||  > 0 \}}\, \big( \vert {\bm b} ( \lVert X(t) \rVert) \vert + {\bm a} ( \lVert X(t) \rVert) \big) {\mathrm d} t < \infty \, , \,\,\,\,\,\,\, 0 \leq T < \infty\, . 
\end{equation}

Let us recall the stopping time $\, \sigma_{\varepsilon}= \inf\{ t > 0 :  \lVert X (t) \rVert  \le \varepsilon \} \,$ for every $\,\varepsilon > 0\,$. Since the function $\, x \mapsto  \Vert x \Vert = \sqrt{x_{1}^{2}+x_{2}^{2}\,} \,$ is smooth on $\, \mathbb R^{2} \setminus \{ \bm 0 \} \,$, we get the following from (\ref{eq: SDEMP}) by \textsc{It\^o}'s formula: 

\begin{equation}
\label{||X||sigmavarepsilon}
\lVert X(\cdot \wedge \sigma_{\varepsilon})\rVert\ = \Vert X(0) \Vert + \int_{0}^{\cdot \wedge \sigma_{\varepsilon}} \, \Big( {\bm b} \big( \lVert X(t) \rVert \big) \,{\mathrm d} t + {\bm \sigma} \big( \lVert X(t) \rVert \big) \,{\mathrm d} W(t)\Big) \, .  
\end{equation}
With $\, \tau(0):= \inf\{ t \geq 0 :  \lVert X (t) \rVert  = 0  \} = \lim_{\varepsilon \downarrow 0 } \sigma_{\varepsilon} \,$, we let $\, \varepsilon \downarrow 0 \,$ in (\ref{||X||sigmavarepsilon}) and obtain (from (\ref{ineq: finite}))
\[
\lVert X(\cdot \wedge \tau (0))\rVert\ = \Vert X(0) \Vert + \int_{0}^{\cdot} {\bf 1}_{(0, \, \tau(0)) }(t)\, \Big( {\bm b} \big( \lVert X(t) \rVert \big) \,{\mathrm d} t + {\bm \sigma} \big( \lVert X(t) \rVert \big) \,{\mathrm d} W(t)\Big) \, .
\]
Recall now the stopping times $\,\big\{ \tau_{m}^{\varepsilon} \,, \, m \in \N_0\big\} \,$ defined in (\ref{eq: tau rec}). In the same manner as above we obtain 
$$
\, \lVert X(\cdot \wedge \tau_{2 \ell + 2}^{\varepsilon})\rVert - \lVert X(\cdot \wedge \tau_{2 \ell + 1}^{\varepsilon})\rVert = \int_{0}^{\cdot} {\bf 1}_{(\tau^\varepsilon_{2 \ell + 1}, \, \tau^\varepsilon_{2\ell+2}) }(t)\, \Big( {\bm b} \big( \lVert X(t) \rVert \big) \,{\mathrm d} t + {\bm \sigma} \big( \lVert X(t) \rVert \big) \,{\mathrm d} W(t)\Big) \, 
$$
for $\, \ell \in \mathbb N_{-1} =  \mathbb N_{0} \cup \{-1\}\,$ with $\, \tau_{-1}^{\varepsilon} \equiv 0\,$. We decompose $\, \Vert X(T) \Vert \,$ as in the proof of Theorem \ref{Gen_FS}:
\begin{equation}\label{Decompose||X||}
\Vert X(T) \Vert = \Vert X(0) \Vert + \sum_{\ell \in \mathbb N_{-1}} \big ( \Vert X(T \wedge \tau^\varepsilon_{2\ell+2}) \Vert - \Vert X(T \wedge \tau^\varepsilon_{2\ell+1}) \Vert \big) 
+ \sum_{\ell \in \mathbb N_{0}}   \big ( \Vert X(T \wedge \tau^\varepsilon_{2\ell+1}) \Vert - \Vert X(T \wedge \tau^\varepsilon_{2\ell}) \Vert \big) \, .
\end{equation}
With the previous considerations, 
 letting $\, \varepsilon \downarrow 0 \,$,  
 we obtain 
 (\ref{eq: ||X||}) for the radial process $\, \lVert X(\cdot) \rVert\,$.   The continuous  semimartingale $X(\cdot)$ thus solves also the system (\ref{eq: skewTanaka}) with the ``folded driver" $\, S(\cdot) = \lVert X(\cdot) \rVert  \,$.   
 \newpage
 
  Suppose now that the properties (\ref{eq: zeroLeb})-(\ref{eq: localDist}) are also satisfied by the weak solution we have posited. Thanks to Theorem \ref{Gen_FS}, for every  given function $\, g \in \mathfrak D_{+}^{{\bm \mu}}  \,$ $($resp., $\, g \in \mathfrak D^{{\bm \mu}}  )$, the process  $\,M^{g} (\cdot \,; X) \,$ as in (\ref{eq: MP2}) is then a local submartingale (resp., martingale).  The property (\ref{eq: zeroLeb2}) comes from (\ref{eq: zeroLeb}). Consequently, a  solution $\,{\mathbb Q} \,$ to the local martingale problem  associated with the triple $\,({\bm \sigma}, {\bm b}, {\bm \mu}) \,$ is given by     the  probability  measure $\, {\mathbb Q} = \mathbb P X^{-1}\,$  induced by the process $\,X(\cdot) \,$  on the canonical space $\, (\Omega_{2}, \mathcal F_{2})\,$.  \qed

\medskip
\noindent
{\it Proof of Part (b):} 
Conversely, suppose that the local martingale problem  associated with the triple $\,({\bm \sigma}, {\bm b}, {\bm \mu}) \,$ has a   solution $\, \mathbb Q\,$. We recall the notation in (\ref{gamma}) and define on the canonical space the processes
\begin{equation} \label{eq: XMPSDE}
X(\cdot) \, \equiv \,  \big(X_{1}(\cdot), X_{2}(\cdot)\big)^{\prime} \, :=\,   \big( \, \lVert \omega_{2} (\cdot) \rVert \, {\mathfrak f}_{1} ( \omega_{2}(\cdot))  \, , \,\, \lVert \omega_{2} (\cdot)\rVert \, {\mathfrak f}_{2} ( \omega_{2}(\cdot)) \,\big)^{\prime}   , 
\end{equation}
\begin{equation}\label{eq: M_i}
M_{i}(\cdot)  \, :=\, X_{i}(\cdot) - X_{i}(0) - \int^{\cdot}_{0} {\bm b}( \lVert X(t) \rVert) {\mathfrak f}_{i}( X(t)) {\mathrm d} t \,-\, \gamma_i \Big( \lVert X(\cdot) \rVert - \lVert X(0) \rVert - \int^{\cdot}_{0} {\bm b} (\lVert X(t) \rVert)\, \mathbf{ 1}_{ \{ \lVert X(t) \rVert>0 \}}  {\mathrm d} t \Big) , 
\end{equation}
\[
M_{i,k}(\cdot) \, :=\, g_{i,k}(X(\cdot)) - g_{i,k}(X(0)) - 2 \int^{\cdot}_{0} \lVert X(t) \rVert \, {\bm b} ( \lVert X(t) \rVert) \, \big({\mathfrak f}_{i}(X(t)) - \gamma_{i}\big) ( {\mathfrak f}_{k}(X(t)) - \gamma_{k}\big)\, \mathbf{ 1}_{ \{ \lVert X(t) \rVert>0 \}}  {\mathrm d} t  \,  
\]
\[
- \int^{\cdot}_{0} {\bm a} (\lVert X(t) \rVert )\big({\mathfrak f}_{i}(X(t)) - \gamma_{i}\big) ( {\mathfrak f}_{k}(X(t)) - \gamma_{k}\big)\, \mathbf{ 1}_{ \{ \lVert X(t) \rVert>0 \}}  {\mathrm d} t    
\]
for $\, 1 \le i, k \le 2\,$, as well as  
\begin{equation} 
\label{eq: XMPSDE.a}
M_{i,i}^\circ (\cdot) \, := \, X_{i}^2(\cdot) - X_{i}^2(0) - \int^{\cdot}_{0} {\mathfrak f}_{i}^2( X(t)) \, \Big[ \,2\, \lVert X(t) \rVert \, {\bm b}( \lVert X(t) \rVert)  + {\bm a}( \lVert X(t) \rVert) \,\Big] \, {\mathrm d} t   \,.
\end{equation}
$\bullet~$
 Here, as in Proposition \ref{prop: MART}, we consider the following  functions  in the family   $\,\mathfrak D^{{\bm \mu} }\,$ of (\ref{eq: Dmu}): 
\begin{equation}
\label{eq: G1}
g_{1}(x) \, :=\, r \big( \cos (\theta) - \gamma_1  \big) \, , \quad 
g_{2}(x) \, :=\, r \big( \sin (\theta) -\gamma_2 \big) \, , \quad 
g_{i,k}(x) \, :=\, g_{i}(x) \, g_{k}(x)   
\end{equation} 
for $\, x = (r, \theta) \in {\mathbb R}^{2}\,$, $\, 1 \le i, k \le 2\,$. We consider also the functions $\,g_{i,i}^{\circ} \in {\mathfrak D}^{\bm \mu}\,$ and $\,g_{3} \in {\mathfrak D}^{\bm \mu}_{+}\,$ defined by 
\begin{equation}
\label{eq: G2}
\, g_{1,1}^{\circ}(x) \, :=\, r^{2} \cos^{2} (\theta) \,, \quad \,g_{2,2}^{\circ}(x) \, :=\, r^{2} \sin^{2}(\theta) \,, \quad \, g_3 (x) := r\, ; \qquad x \in {\mathbb R}^{2} 
\, . 
\end{equation} 
\noindent
$\bullet~$  
We deduce then from (\ref{eq: MP2}) that the processes 
 $\,
M_{i}(\cdot) \equiv  M^{g_{i}}( \cdot\,;X)\,, \,\, 
M_{i,k}(\cdot) \equiv M^{g_{i,k}}( \cdot\,;X)\,$ as well as 
$\, M_{i,i}^\circ (\cdot)  \equiv  M^{g_{i,i}^\circ}( \cdot\,;X)    \, ,  
$  
are continuous local martingales for $\,1 \le i, k \le 2\,$;  and that so are the processes 
\[
M_{i,k}(\cdot) - g_{i}(X(0)) M_{k}(\cdot) - g_{k}(X(0)) M_{i}(\cdot)  - \int^{\cdot}_{0} \Big( \int^{t}_{0}  \big( {\mathfrak f}_{k}(X(s)) - \gamma_{k}\big) {\bm b}( \lVert X(s) \rVert) \, \mathbf{ 1}_{ \{ \lVert X(s) \rVert>0 \}}  {\mathrm d} s \Big) {\mathrm d} M_{i}(t)
~~~~~~~~~~~~~~~~
\]
\[
 ~~~~~~~~~~~~~~~~ - \int^{\cdot}_{0} \Big( \int^{t}_{0}  \big( {\mathfrak f}_{i}(X(s)) - \gamma_{i} \big) {\bm b}( \lVert X(s) \rVert) \, \mathbf{ 1}_{ \{ \lVert X(s) \rVert>0 \}}  {\mathrm d} s \Big) {\mathrm d} M_{k}(t) \, =\, M_{i}(\cdot) M_{k}(\cdot) - \int^{\cdot}_{0} r_{i,k}(t) \, {\mathrm d} t \, .
\]
 This way, we identify  for $\,1\le i , k \le 2\,$ the cross-variation structure 
\begin{equation} 
\label{eq: rikt}
\langle M_{i} , M_{k} \rangle (\cdot) = \int^{\cdot}_{0} r_{i,k}(t) {\mathrm d} t \,  ,\, \quad 
r_{i,k} ( t)   :=  {\bm a}( \lVert X( t) \rVert) \,\big( {\mathfrak f}_{i}(X( t)) - \gamma_i \big) \big( {\mathfrak f}_{k}(X( t)) - \gamma_k  \big)\,\mathbf{ 1}_{ \{ \lVert X(t) \rVert>0 \}}     \,.
\end{equation}

\noindent
$\bullet~$ 
We also observe  that the continuous process 
\begin{equation} \label{eq: N}
N(\cdot) \, :=\, M^{g_3 }  (\cdot\,;X) \,=\, \lVert X(\cdot) \rVert - \lVert X(0) \rVert - \int^{\cdot}_{0} {\bm b} (\lVert X(t) \rVert)\, \mathbf{ 1}_{ \{ \lVert X(t) \rVert>0  \}} \, {\mathrm d} t 
\end{equation}
is a  local submartingale; this way we obtain the semimartingale property of the radial process $\, || X(\cdot)||\,$.  
By the \textsc{Doob-Meyer} decomposition   (e.g., \textsc{Karatzas \& Shreve} (1991), Theorem 1.4.10), there exists then an adapted, continuous and increasing process $\, A(\cdot) \,$ such that 
\begin{equation} 
\label{eq: M3}
 M_{3}(\cdot)  \,:=\,  N(\cdot) - A(\cdot) \,=\,\lVert X(\cdot) \rVert - \lVert X(0) \rVert - \int^{\cdot}_{0} {\bm b} (\lVert X(t) \rVert)\, \mathbf{ 1}_{ \{ \lVert X(t) \rVert>0  \}} \, {\mathrm d} t - A(\cdot)
 \end{equation}
  is a continuous local martingale. {\it We claim that this increasing process is   $\, A(\cdot) = L^{||X||} (\cdot)\,$, the local time at the origin of the continuous, nonnegative semimartingale $\, || X(\cdot)||\,$.}
 \newpage
\noindent
$\bullet~$ 
In order to substantiate this claim, let us fix two arbitrary positive constants $\, c_{1} , c_{2} \,$ with $\, c_{1} < c_{2}\,$ and define a sequence of stopping times inductively, via $\, \sigma_{0}    :=   \inf \{ t \ge 0 : \lVert X(t) \rVert \, =\, c_{2} \} \,$  if $\, \lVert X(0) \rVert < c_{2}\,$  and $\, \sigma_{0}    :=0\,$  otherwise; as well as 
\[
\sigma_{2n+1} \, :=\,  \inf \{ t \ge \sigma_{2n} : \lVert X(t) \rVert \, =\, c_{1} \} \, , \quad \sigma_{2n+2} \, :=\,  \inf \{ t \ge \sigma_{2n+1} : \lVert X(t) \rVert \, =\, c_{2} \} \, ; \quad n \in \mathbb N_{0} \, . 
\] 
We note that  $\, \Vert X(t) \Vert \geq c_{1}\,$ holds for $\, t \in (\sigma_{2n}, \sigma_{2n+1})\,$; and conversely, that  $\, \Vert X(t) \Vert > c_{2}\,$ implies $\, t \in (\sigma_{2n}, \sigma_{2n+1})\,$ for some $\, n \in \mathbb N_{0}\,$. Thus,  by taking an appropriate smooth function $\, g_{4}  \in \mathfrak D^{\bm \mu}\,$ of the form $\, g_{4}(r, \theta) = \psi(r)\,$ where $\, \psi : [0, \infty) \to [0, \infty) \,$ is smooth with $\,\psi(r) = r\,$ for $\,r \geq c_{1}\,$, one can show that $  N(\cdot \wedge \sigma_{2n+1}) - N(\cdot \wedge \sigma_{2n})\, $ is a continuous local martingale. 

\smallskip
Then, since both processes $\, N(\cdot \wedge \sigma_{2n+1}) - A(\cdot \wedge \sigma_{2n+1}) \,$ and $\, N(\cdot \wedge \sigma_{2n}) - A(\cdot \wedge \sigma_{2n}) \,$ are continuous local martingales, so is $\,A(\cdot \wedge \sigma_{2n+1}) - A(\cdot \wedge \sigma_{2n}) \,$. But this last process  is of bounded variation, so $\,A(\cdot \wedge \sigma_{2n+1}) \equiv  A(\cdot \wedge \sigma_{2n})\,$ for every $\,n \in \mathbb N_{0}\,$. In other words, the process $\,A(\cdot)\,$ is flat on $\,[\sigma_{2n}, \sigma_{2n+1}]\,$ for every $n$. Therefore we have $\, 
\int^{\infty}_{0} {\bf 1}_{\{\lVert X (t) \rVert \in (c_{2}, \infty) \}}  \, {\mathrm d} A(t) \, \equiv\,  0 \,$, because $\,\lVert X (t) \rVert \in (c_{2}, \infty)\,$ implies $\, t\in (\sigma_{2n}, \sigma_{2n+1})\,$ for some $\, n \in \mathbb N_{0}\,$. Since $\, c_{2} > 0\,$ can be chosen arbitrarily small, we obtain
\begin{equation} \label{eq: A}
\, A(\cdot)  \, =\,  \int^{\cdot}_{0} {\bf 1}_{\{ \lVert X(t) \rVert \, =\, 0\}} \, {\mathrm d} A(t) \, , \quad \text{ and } \quad \int^{\cdot}_{0} \lVert X(t) \rVert   {\mathrm d} A(t) \, =\,  0 \, . 
\end{equation}
In conjunction with (\ref{eq: N})--(\ref{eq: A}),  the characterization      $\, L^{ \lVert X\rVert}(\cdot) =  \int^{\cdot}_{0} {\bf 1}_{\{ \lVert X (t) \rVert =0 \}} {\mathrm d} \lVert X(t) \rVert  \, $ for the  local time of a continuous, nonnegative semimartingale such as $||X(\cdot)||\,$,   establishes then the claim,  since  
$$
L^{||X||} (\cdot)\,=\int^{\cdot}_{0} \mathbf{ 1}_{ \{ \lVert X(t) \rVert =0 \}} \Big(  {\bm b} (\lVert X(t) \rVert)\,\mathbf{ 1}_{ \{ \lVert X(t) \rVert>0 \}}\,   {\mathrm d} t + {\mathrm d} A(t) \Big) \, =\,  A(\cdot) \,.
$$

\noindent
$\bullet~$ 
We return   to the computation of the cross-variations $\, \langle M_{i}, M_{k}\rangle(\cdot)\,$ for $\, 1 \le i, k \le 3\,$.  
Recalling (\ref{eq: M3}), (\ref{eq: A})  and $\,  \langle \lVert X \rVert \rangle(\cdot) =  \langle M_{3}\rangle(\cdot) =  \langle N\rangle(\cdot)  $, an  application of \textsc{It\^o}'s rule to $\, \lVert X (\cdot) \rVert^{2}\,$ gives 
\[
\lVert X(\cdot) \rVert^{2} -  \lVert X (0) \rVert^{2} -  2 \int^{\cdot}_{0}  \lVert X(t) \rVert{\bm b}( \lVert X(t) \rVert) \, {\bf 1}_{\{ \lVert X(t) \rVert > 0 \}} {\mathrm d} t - \langle  N \rangle (\cdot) \, =\, 2  \int^{\cdot}_{0}  \lVert X(t) \rVert \, {\mathrm d}   M_{3}(t)    \,.
\]
Combining the last identity with (\ref{eq: XMPSDE.a}),  we observe that 
\begin{equation} 
\label{eq: M1122}
M_{1,1}^\circ (\cdot) + M_{2,2}^\circ (\cdot) 
- 2 \int^{\cdot}_{0} \lVert X(t) \rVert {\mathrm d} M_{3}(t)\, =\,  \langle N \rangle(\cdot) - \int^{\cdot}_{0}  {\bm a}( \lVert X(t) \rVert) \, {\bf 1}_{\{ \lVert X(t) \rVert>0\}}  \,{\mathrm d} t 
\end{equation}
is both a local martingale and a continuous process of bounded variation; thus we identify 
\begin{equation} 
\label{eq: <||X||>} 
\langle \lVert X \rVert \rangle(\cdot) = \langle N\rangle (\cdot)  = \langle M_{3}\rangle(\cdot) =  \int^{\cdot}_{0} r_{3,3}(t) {\mathrm d} t \quad \, \text{where}\quad   r_{3,3}(t)   :=     {\bm a}( \lVert X(t) \rVert)\,{\bf 1}_{\{ \lVert X(t) \rVert>0  \}}  \, . 
\end{equation} 
By analogy with the derivation of (\ref{eq: M1122}), and taking (\ref{eq: M_i}) into account,
 we observe that 
\[
M_{i,i}^\circ (\cdot)- 2 \int^{\cdot}_{0} X_{i}(t) {\mathrm d} M_{i}(t) -2 \, \gamma_{i} \int^{\cdot}_{0} X_{i}(t) {\mathrm d} M_{3}(t) \, =\,  \langle X_{i}\rangle (\cdot) - \int^{\cdot}_{0} {\bm a}( \lVert X (t) \rVert) \, \big[{\mathfrak f}_{i}(X(t))\big]^{2} {\mathrm d} t 
\]
is both a local martingale and a continuous process of bounded variation for $\,i = 1, 2\,$; thus we identify 
\begin{equation} \label{eq: Xi} 
\, \langle X_{i} \rangle(\cdot) \, =\, \int^{\cdot}_{0} {\bm a}( \lVert X (t) \rVert) \, \big[ \, {\mathfrak f}_{i}(X(t))\big]^{2} \, {\mathrm d} t \,, \quad i=1,2 \,.
\end{equation} 
It follows   from (\ref{eq: XMPSDE}) that   $\, \langle M_{i}\rangle(\cdot) =  \langle X_{i}\rangle(\cdot) + \gamma_{i}^{2} \langle \lVert X\rVert\rangle(\cdot) - 2 \gamma_{i} \langle X_{i}, \lVert X \rVert\rangle(\cdot)\,$; and in conjunction with (\ref{eq: <||X||>}), (\ref{eq: Xi}), (\ref{eq: rikt}), this gives  $\, \, \langle X_{i}, \lVert X \rVert \rangle (\cdot) \, =\, \int^{\cdot}_{0} {\bf 1}_{\{X(t) \neq {\bm 0}\}} {\bm a} ( \lVert X (t) \rVert  )\, {\mathfrak f}_{i}( X(t)) \,{\mathrm d} t\, $. Hence, with  $\,r_{i,3}(t) \, \equiv \, r_{3,i}(t) \, :=\, {\bm a}( \lVert X( t) \rVert) \,\big( {\mathfrak f}_{i}(X( t)) - \gamma_i \big) \,\mathbf{ 1}_{ \{ X(t) \neq \mathbf{ 0}\}}  \,$   for $\,i \, =\, 1, 2\,$, we obtain  
\[
\, \langle  M_{i}, M_{3}\rangle (\cdot) \, \equiv \, \langle  M_{3}, M_{i} \rangle (\cdot) \, =\,  \langle X_{i}, \lVert X\rVert \rangle (\cdot) - \gamma_{i} \langle \lVert X \rVert \rangle (\cdot) \, = \int^{\cdot}_{0} r_{i,3}(t) {\mathrm d} t \,.
\]
\noindent $\bullet~$  
We have now computed all elements of the $\,(3\times 3) \,$ matrix $\, ({\mathrm d} \langle M_{i}, M_{k}\rangle (t) \, / \, {\mathrm d} t)_{1\le i, k \le 3} \, =\,  ( r_{i,k}(t))_{1\le i, k \le 3} \,$;   we observe also that this matrix is of rank 1, on $  \{ t \ge 0: X(t) \neq {\bm 0} \}  $.  By Theorem 3.4.2  and     Proposition 5.4.6  in  \textsc{Karatzas \& Shreve} (1991),    there exists an extension of the original probability space, and on it \\ (i) a three-dimensional standard Brownian motion $\, \widetilde{W}(\cdot) \, =\, \big( \widetilde{W}_{1}(\cdot), \widetilde{W}_{2}(\cdot) , \widetilde{W}_{3}(\cdot) \big)^{\prime}\,$, \\ (ii) a one-dimensional standard Brownian motion $\, {W}(\cdot) \,$, and \\ (iii) measurable, adapted, matrix-valued processes $\,  (\rho_{i,k}(\cdot) )_{1\le i,k \le 3} \,$ with  $\, \int^{T}_{0}[\rho_{i,k}(t)]^{2} {\mathrm d} t < \infty \,$,  \\ such that we have  the representations 
\begin{equation} 
\label{eq: Mi}
M_{i}(\cdot) \, =\, \sum_{k=1}^{3} \int^{\cdot}_{0} \rho_{i,k}(t) \, {\mathrm d} \widetilde{W}_{k}(t) \, =   \int^{\cdot}_{0} {\bm \sigma} ( \lVert X(t) \rVert) \,\big( \mathfrak f_{i}(X(t)) -  \gamma_{i} \big)\,\mathbf{ 1}_{ \{ X(t) \neq \mathbf{ 0}\}}\, {\mathrm d} W(t)    \,, \qquad i=1,2    
\end{equation}
and  $\,M_{3}(\cdot) \, = \int^{\cdot}_{0} {\bm \sigma} ( \lVert X(t) \rVert ) \, {\bf 1}_{\{ X(t) \neq {\bm 0}\}}\, {\mathrm d} W(t) \,$.  Substituting this into the decomposition $\, N(\cdot) = M_{3}(\cdot) + L^{||X||}(\cdot) \,$ and then into (\ref{eq: N}), we obtain the stochastic equation (\ref{eq: ||X||})  for the radial process $\, || X(\cdot) ||\,$. Substituting the expressions of (\ref{eq: Mi}), (\ref{eq: ||X||}) into $\,M_{i}(\cdot) \,$ in (\ref{eq: XMPSDE}) for $\,i =1, 2\,$,  
we observe that the process $ X(\cdot)  $ defined in (\ref{eq: XMPSDE}) satisfies the system of   (\ref{eq: SDEMP}). It  follows from    
(\ref{eq: zeroLeb2}) that $ X(\cdot)  $ satisfies the property  (\ref{eq: zeroLeb}) .

\medskip 
\noindent 
$\bullet~$ Finally, for every set $\, A \in \mathcal B ( [0, 2\pi) ) \,,$  we  consider the functions 
\begin{equation} 
\label{eq: G3}
g_{5}(x) \, :=\, g_{5}(r, \theta) \, =\,  r \big( {\bf 1}_{  A } (\theta) - {\bm \nu} (A)\big)\, \qquad \text{and} \qquad \, g_{6}(x) \, :=\, g_{6}(r, \theta) \, =\,  r \,{\bf 1}_{  A } (\theta)
\end{equation} 
in polar co\"ordinates. Since $\,g_{5} \in {\mathfrak D}^{{\bm \mu}}\,$ and $\,g_{6} \in {\mathfrak D}^{{\bm \mu}}_{+}\,$  we obtain that  the process
\begin{equation} 
\label{eq: Mg4}
g_{5}(X(\cdot)) -  g_{5}(X(0)) - \int^{\cdot}_{0}  {\bm b} ( \lVert X(t) \rVert) \big( {\bf 1}_{ \{ \text{arg} (X(t)) \, \in\,  A\} } - {\bm \nu} (A) \big){\bf 1}_{\{ \lVert X(t) \rVert > 0\}} {\mathrm d} t 
\end{equation}
is a continuous local martingale, and that the process 
\[
g_{6} (X(\cdot)) - g_{6}(X(0)) -  \int^{\cdot}_{0}  {\bm b} ( \lVert X(t) \rVert)  {\bf 1}_{ \{ \text{arg} (X(t)) \, \in \, A\} \cap \{ \lVert X(t) \rVert > 0\}} {\mathrm d} t 
\]
 is a continuous local submartingale. 

Repeating an  argument similar to the one deployed above, we identify  $\,{\bm \nu} (A) L^{ \lVert X\rVert } (\cdot) \,$  as the local time $\,L^{R_{A}}(\cdot) \,$ at the origin for the continuous, non-negative semimartingale  $\,R^{A}(\cdot) \, :=\, g_{6}(X(\cdot)) \, $. Indeed,  
\begin{equation} 
\label{eq: Mg5}
R^{A} (\cdot) - R^{A} (0) -  \int^{\cdot}_{0}  {\bm b} ( \lVert X(t) \rVert)  {\bf 1}_{ \{ \text{arg} (X(t)) \in A\} \cap \{ \lVert X(t) \rVert > 0\}} {\mathrm d} t  - L^{R^{A} }(\cdot) 
\end{equation}
is a continuous local martingale. 
Moreover, on account of (\ref{eq: M3}), we see that 
\begin{equation} 
\label{eq: nuM}
{\bm \nu} (A) \Big ( \lVert X (\cdot) \rVert - \lVert X(0) \rVert - \int^{\cdot}_{0} {\bm b} ( \lVert X(t) \rVert) {\bf 1}_{\{ \lVert X(t) \rVert > 0 \}} {\mathrm d} t - L^{ \lVert X \rVert}(\cdot) \Big) 
\end{equation}
is also a continuous local martingale. Subtracting (\ref{eq: Mg5})  from (\ref{eq: Mg4}) and adding (\ref{eq: nuM}),  we deduce  that the finite variation process $\,  L^{R^{A}}(\cdot) - {\bm \nu} (A) L^{ \lVert X \rVert}(\cdot) \, $   is a continuous local martingale, and hence identically zero, i.e., \,$
L^{R^{A}} (\cdot)   \equiv   {\bm \nu} (A) \,L^{ \lVert X\rVert } (\cdot)  
\,$ as in (\ref{eq: localDist}).   
 
We conclude from this analysis, that the system of equations (\ref{eq:
  SDEMP}) admits a weak solution with the properties (\ref{eq:
  zeroLeb}) and (\ref{eq: localDist}).   This proves Part (b). Part
(c) is now evident. \qed 

\smallskip 
\begin{remark} \label{remark: 6.1}
Looking back to the definition of the above local martingale problem for the planar diffusion, we recall Definition \ref{def: D} and observe  that the following statements (i)-(ii) are equivalent: 

\noindent (i) For every $\,g \in \mathfrak D^{\bm \mu}_{+}\,, $ the process $\,M^{g}(\cdot \,; \omega_{2}) \,$ is a continuous local submartingale;  

\noindent (ii)  For every $\,g \in \mathfrak D^{\bm \mu}\,, $ the process $\,M^{g}(\cdot \,; \omega_{2}) \,$ is a continuous local martingale, and the process  $\, M^{g_{3}}(\cdot \,; \omega_{2}) \,$ is a continuous local submartingale, where $\,g_{3}(x) =  \lVert x \rVert = r \,$ is defined in (\ref{eq: G2}). 

\smallskip
If (i) holds, $\, M^{g_{3}}(\cdot \,; \omega_{2}) \,$ is a continuous local submartingale since $\, g_{3} (x)  = \lVert x \rVert\,$ belongs to $\, \mathfrak D^{\mu}_{+}\,$. For every $\,g \in \mathfrak D^{\bm \mu}\,$ we have $\, g \in \mathfrak D^{\bm \mu}_{+} \,$ and $\, - g \in \mathfrak D^{\mu}_{+}\,$,   hence both $\,M^{g}(\cdot \,; \omega_{2})\,$ and    $\, M^{-g}(\cdot \,; \omega_{2}) = - M^{g}(\cdot \,; \omega_{2}) \,$ are continuous local submartingales.  Thus $\, M^{g}(\cdot\,; \omega_{2}) \,$ is a continuous local martingale, and (ii) follows.

  \newpage

Next, let us assume (ii). Every $\, g \in \mathfrak D^{\bm \mu}_{+}\,$ can then be decomposed as $\, g     :=  g_{(1)} + g_{(2)}\,$, where $\,g_{(1)}(x) = c \lVert x \rVert   \,$ with $\,c   :=  \int^{2\pi}_{0} g_{\theta}^{\prime}(0+) {\bm \nu} ({\mathrm d} \theta) \ge 0 \,$ and $\,g_{(2)} (\cdot) \, :=\, g(\cdot) - g_{(1)}(\cdot) \in \mathfrak D^{\bm \mu} \,$. Thus the above condition (ii) implies that $\,M^{g_{(1)}}(\cdot\,; \omega_{2}) = c \lVert \omega_{2} (\cdot) \rVert\,$ is a local submartingale and $\, M^{g_{(2)}}(\cdot \,; \omega_{2}) \,$ is a local martingale, and hence $\, M^{g}(\cdot\,; \omega_{2}) = M^{g_{(1)}}(\cdot\,; \omega_{2}) + M^{g_{(2)}}(\cdot\, ; \omega_{2}) \,$ is a local submartingale, and (i) follows. 
\end{remark}

 \begin{remark} In the proof of Proposition \ref{prop: SDEMP}(b) let us define a random measure $\, \mathbf{ R}(t, {\mathrm d} z)\, $ on $\, \mathcal B (\mathfrak S) \, $ for every $\,t \ge 0\,$ with $\, \mathbf{ R}(t, {\mathrm d} z ) \equiv \mathbf{ R}(t, {\mathrm d} \theta) \,$ for $\, \theta \in [0, 2\pi)\,$, $\, z \in \mathfrak S\,$ via $\, g_{6}(X(t)) = R^{A}(t) = \int_{A} \mathbf{ R}(t, {\mathrm d} \theta)\,$, $\,A \in \mathcal B([0, 2\pi))\,$. Then for every $\, t \ge 0 \,$ we may write $\,  X(t) \, =\,  \int_{\mathfrak S} z \, \mathbf{ R}(t, {\mathrm d} z)  \, $. 
 \end{remark}

\subsection{Well-Posedness}   
\label{sec_513}

We conjecture that,  if the local submartingale problem  associated with the pair $\, ({\bm \sigma}, {\bm b}) \,$    is well-posed, then the same is true for the local martingale problem  associated with the triple $\, ({\bm \sigma}, {\bm b},  {\bm \mu}) \,$.   

The result that follows settles this conjecture in the affirmative,     for the driftless case $\, {\bm b} \equiv \mathbf{ 0}\,$.  Proposition \ref{Cor: MP1} then deals    with the case of a drift $\, {\bm b} = {\bm \sigma} {\bm c}\,$ with $\, {\bm c}: \R_+ \to \R\,$ bounded and measurable.

\begin{prop} {\bf Well-Posedness for Driftless \textsc{Walsh} Diffusions:} 
\label{prop: MP1}
Suppose that  \\ (i)  the drift $\,{\bm b} \,$ is identically equal to  zero;  and that \\ (ii) the reciprocal of the dispersion co\"efficient $\, {\bm \sigma} : [0, \infty) \rightarrow \mathbb{R} \setminus \{ 0\}\,$ is   locally square-integrable; i.e.,   
\begin{equation} 
\label{eq: ES}
\int_K \,  \frac{{\mathrm d} y}{\, {\bm \sigma}^{2} (   y) \, }  \,<\,    \infty  \,  \qquad \text{holds for every compact set } \,K  \subset [0, \infty)\,. 
\end{equation}
Then the local submartingale problem of   subsection \ref{sec_511}, associated with the pair $\, ({\bm \sigma}  , \mathbf{ 0}) \,$,   is  well-posed. 

Moreover,  
the local martingale problem of subsection \ref{sec_512}  associated with the triple $\, ({\bm \sigma}, \mathbf{ 0},  {\bm \mu}) \,$   is also   well-posed;  and uniqueness in distribution holds,  subject to the properties in    
(\ref{eq: zeroLeb}) and (\ref{eq: localDist}),   for the corresponding system of stochastic integral equations in (\ref{eq: SDEMP})  
with $\, {\bm b} \equiv 0\,$,  namely,  
\begin{equation}
\label{5.5}
 X_i (\cdot) \, =\, X_i (0)+\int_0^{\, \cdot} {\mathfrak f}_i(X(t)) \, {\bm \sigma} \big( \lVert X( t) \rVert \big) \, {\mathrm d} W(t) +     
 \gamma_i \, L^{ \,||X|| }  (\cdot) \, , \quad      \quad i=1, 2  \,.
\end{equation}
 \end{prop}

\noindent
{\it Proof of Existence:}  Let us consider the stochastic integral equation 
 \begin{equation} 
 \label{eq: REFL}
 S(\cdot) \, =\,  r + \int^{\cdot}_{0} {\bm \sigma} ( S(t)) \, {\mathrm d} W(t) + L^{ S}(\cdot) \,   
 \end{equation}
driven by one-dimensional Brownian motion $\,W(\cdot)$. 
 It is shown in \textsc{Schmidt} (1989) that, under   
 (\ref{eq: ES}), this equation (\ref{eq: REFL}) has a non-negative, unique-in-distribution weak solution;    equivalently,   the  \textsc{Stroock \& Varadhan} (1971) local   submartingale problem associated with   $\, ({\bm \sigma}, \mathbf{ 0}) \,$ for $\, K^{\psi}(\cdot) \,$   in (\ref{eq: MP1})  is well-posed. 
 
 Let us also verify the property (\ref{eq: SC}).  From Exercise 3.7.10 in \textsc{Karatzas \& Shreve} (1991),   we get
$$
0 \, =\int_0^\infty \, \mathbf{ 1}_{ \{ S(t) =0\} }\, \mathrm{d} \langle S \rangle (t) \, =\int_0^\infty \, \mathbf{ 1}_{ \{ S(t) =0\} }\,{\bm \sigma}^2 (S(t))\, \mathrm{d} t \,, \qquad \text{thus also} \qquad \int_0^\infty \, \mathbf{ 1}_{ \{ S(t) =0\} } \, \mathrm{d} t =0\,
$$
 because $\, {\bm \sigma} (\cdot)\,$ never vanishes. It follows then from Theorem \ref{prop: skewTanaka} that, on a suitably enlarged  probability space, we may construct from this reflected diffusion  $\,  S (\cdot)\,$   a continuous, planar semimartingale $\,X(\cdot) \,$ which satisfies $\, ||X(\cdot)|| = S(\cdot)\,$, the system of  equations (\ref{5.5}), and the properties (\ref{eq: DIST3})--(\ref{eq: localDist}).  On the strength  of Proposition \ref{prop: SDEMP}(a), the local martingale problem associated with the triple $\,({\bm \sigma}, \mathbf{ 0}, {\bm \mu}) \,$  admits a solution. \qed

 \medskip
\noindent
{\it Proof of Uniqueness:} We adopt the idea of proof in Theorem 3.2 of \textsc{Barlow, Pitman \& Yor} (1989). 
Suppose   there are {\it two} solutions $\,\mathbb Q_{j} \,$, $\,j = 1, 2\,$ to this local martingale problem associated with the triple $\,({\bm \sigma}, \mathbf{ 0}, {\bm \mu})\,$. Let us take an arbitrary  set $\,A \in \mathcal B ([0, 2\pi))\,$  and consider the  functions $\, h_{A}(\cdot) \,$ and $\,g_{A}(\cdot) \,$ defined as in (\ref{eq: gh}) for the  indicator $\, \varphi = \mathbf{ 1}_A\,$, namely  
\begin{equation} 
\label{eq: HA}
 h_{A}(x) \, :=\, \big({\bf 1}_{\{\text{arg} (x) \in A\}} - {\bm \nu}(A) \big) \cdot  {\bf 1}_{\{\lVert x \rVert > 0\}} =\, \big( {\bm \nu}(A^c)\,{\bf 1}_{\{\text{arg} (x) \in A\}} - {\bm \nu}(A)\,{\bf 1}_{\{\text{arg} (x) \in A^c\}} \big) \cdot  {\bf 1}_{\{\lVert x \rVert > 0\}}  \,,\,\,\,
\end{equation}
\begin{equation} 
\label{eq: GA}
g_{A}(x) \,: =\,    \lVert x   \lVert \, h_{A}(x) \,, \qquad x \in \mathbb R^{2}\,.
\end{equation}
  \newpage  
The above function $\, \, g_{A}(\cdot)\,$ belongs to the family $\,\mathfrak D^{{\bm \mu}}\,$ in (\ref{eq: Dmu}), as does the function $\, [g_{A}(\cdot)]^{2}\,$. By assumption and Proposition \ref{prop: MART}, the process $\, M_A  (\cdot )  :=   g_{A}(\omega_{2}(\cdot)) \,$ is then a $\,\mathbb Q_{j}-$local martingale,  with   
 \[
\langle M_A \rangle (T) \, =\, \big\langle g_{A}(\omega_{2}(\cdot) ) \big\rangle (T) \, =\, \int^{T}_{0} \big[h_{A}(\omega_{2}(t)) \big]^{2} {\bm a}( \lVert \omega_{2}(t) \rVert ) \, {\mathrm d} t \, ; \quad   0 \le T < \infty\,, \,\,\,\,j \, =\, 1, 2 \,.
 \] 
 Note that $\,\omega_2 (\cdot)\,$ and $\, \Vert \omega_{2}(\cdot) \Vert \,$ solve  in the weak sense the equations (\ref{5.5}) and (\ref{eq: REFL}), respectively. The argument at the beginning of Section \ref{disc} implies that $\,\omega_2 (\cdot)\,$ stays on the same ray on each of its excursion away from the origin. Moreover, $\, \Vert \omega_{2}(\cdot) \Vert \,$ is strongly Markovian with respect to the filtration $\,\mathbb F_{2}\,$, and its distribution is uniquely determined. 
 
 \smallskip
 \noindent
 $\bullet~$ 
Let us assume $\, 0 < {\bm \nu } (A) < 1\,$ first. We note that  $ g_{A}(x) > 0 $, if $  \text{arg} (x) \in A $; $\,g_{A}(x) < 0 \,$ if $\,\text{arg} (x) \in A^{c}\,$; and $\,g_{A}(x) = 0\,$ if $\, x =  {\bm 0}\,$. 
   It is also easy to verify that  the process  
\begin{equation} 
\label{eq: UA0}
 U_{A}(\cdot) \, :=\, \int^{\cdot}_{0} \Big( \frac{1}{\, {\bm \nu} (A^{c})\, } \cdot {\bf 1}_{\{ g_{A}(\omega_{2}(t)) > 0\}} + \frac{1}{\, {\bm \nu}(A)\, } \cdot {\bf 1}_{\{ g_{A}(\omega_{2}(t)) \le 0\}}\Big) \cdot \frac{{\mathrm d} M_A (t)}{\, {\bm \sigma} ( \lVert \omega_{2}(t) \rVert ) \, }\,      
\end{equation}
is a continuous $\,\mathbb Q_{j}-$local martingale with $\, \langle U_{A}\rangle (t)   =  t \,$ for $\, t \ge 0 \,$; i.e.,   a $\,\mathbb Q_{j}-$Brownian motion for $\,j \, =\, 1, 2\,$.   
The  probability distribution of the process $\, M_A (\cdot) = g_{A}(\omega_2 (\cdot)) \,$ is then determined uniquely and {\it independently of the solution  $\,\mathbb Q_{j} \,$, $\,j = 1, 2\,$ to the local martingale problem}. This is because, under the assumption (\ref{eq: ES}) on the dispersion co\"efficient and thanks to the theory of \textsc{Engelbert \& Schmidt} (1984), the stochastic differential equation driven by the Brownian motion $\,U_{A}(\cdot)\,$ and   derived from (\ref{eq: UA0}),    
\begin{equation} 
\label{eq: MgA}
{\mathrm d} M_A (t) \, =\, {\bm \varrho} \big(M_A (t)\big) \, {\mathrm d} U_{A}(t) \, , \quad 0 \le t < \infty   
\end{equation}
with $\,c_{0} \, :=\, {\bm \nu}(A^{c})\,$, $\,c_{1} \, :=\, {\bm \nu}(A)\,$ and the new dispersion function 
\begin{equation} 
\label{eq: SigmaTilde}
{\bm \varrho} (x) \, :=\, c_{0} \cdot {\bm \sigma} \Big( \frac{ \, x\, }{ c_{0}} \Big) \cdot {\bf 1}_{\{ x > 0\}} +  c_{1}\cdot {\bm \sigma} \Big( {} - \frac{ \, x\, }{ \, c_{1}\, }  \Big) \cdot {\bf 1}_{\{ x \le 0\}}  \, ; \quad x \in \mathbb R \,,
\end{equation}
admits a weak solution, which is unique in the sense of the probability distribution.  This follows from Theorem 5.5.7 in \textsc{Karatzas \& Shreve} (1991), and from the fact that   the reciprocal of the  function $\, {\bm \varrho} (\cdot)\,$ inherits the local square-integrability property (\ref{eq: ES}) of the reciprocal of   $\, {\bm \sigma} (\cdot)\,$. Moreover,   $\, M_A  (\cdot )  =   g_{A}(\omega_{2}(\cdot)) \,$ is strongly Markovian with respect to the filtration $\,\mathbb F_{2}\,$ (cf. the proof of Lemma \ref{LM ||X||} in Section \ref{SMP}). Therefore, for an arbitrary $\, C \in {\cal B} ((0, \infty))\,$, recalling (\ref{eq: HA}) and (\ref{eq: GA}), we have 
$$
\mathbb{Q}_j \big( \, || \omega_2 (t) || \in C\,, \, \text{arg} ( \omega_2 (t) ) \in A\, \big| \,   \mathcal F_{2}(s) \,\big) \, =\, \mathbb{Q}_j \big( \, g_{A}(\omega_{2}(t)) \in {\bm \nu} (A^c)\, C\, \big| \,   \mathcal F_{2}(s) \,\big)\, 
$$
$$
= \, \mathbb{Q}_j \big( \, g_{A}(\omega_{2}(t)) \in {\bm \nu} (A^c)\, C \, \big| \,   g_{A}(\omega_{2}(s)) \,\big),   ~~~~~ ~~~\quad \quad ~~ ~~ ~0 \leq s < t < \infty , ~ ~ j = 1, 2 .
$$
Since the distribution of the process $\, g_{A}(\omega_{2}(\cdot)) \,$ is uniquely determined, the above probability does not depend on $\, j =1,2 \,$. We conclude then that $\,\mathbb{Q}_j \big( \, || \omega_2 (t) || \in C\,, \, \text{arg} ( \omega_2 (t) ) \in A\, \big| \,   \mathcal F_{2}(s) \,\big)\,$ does not depend on $\,j=1,2\,$, if $\, 0 < {\bm \nu } (A) < 1\,$. 

 \smallskip
\noindent
$\bullet~$ 
For the resulting diffusion process in natural scale, we shall denote by
$$
p_A \big( s,t; y, B) \,:=\, \mathbb{Q}_j \big( g_A (\omega_2 (t) ) \in B \, \big| \, g_A (\omega_2 (s) ) = y \big) 
$$
$$
p_A^* \big( s,t; y, B) \,:=\, \mathbb{Q}_j \big( g_A (\omega_2 (t) ) \in B\,, \, \tau_s (\omega_2) >t \, \big| \, g_A (\omega_2 (s) ) = y \big)
$$
$$
q_A  \big( s,t; y ) \,:=\, p_A^* \big( s,t; y, \R) \,:=\,\mathbb{Q}_j \big(   \tau_s (\omega_2) >t \, \big| \, g_A (\omega_2 (s) ) = y \big)
$$
its transition, taboo-transition, and survival probabilities (for both $\,j=1, 2\,$ on the strength of uniqueness in distribution for (\ref{eq: MgA})). Here $\, 0 \leq s<t<\infty\,$, $\, y \in \R\,$ and $\, B \in {\cal B}\,$  are arbitrary, and we have denoted the first hitting time of the origin by 
$$\,
\tau_s (\omega_2) \,:=\, \inf \big\{ u \ge s \,:\, || \omega_2 (u)|| = 0 \big\}\, =\, \inf \big\{ u \ge s \,:\, g_{A}(\omega_2 (u)) = 0 \big\} 
 $$
 \noindent
For an arbitrary $\, C \in {\cal B} ((0, \infty))\,$ and recalling (\ref{eq: HA}), (\ref{eq: GA}), we have then the expression
$$
 \mathbb{Q}_j \big( || \omega_2 (t) || \in C\,, \, \text{arg} ( \omega_2 (t) ) \in A\,,   \, \tau_s (\omega_2) >t \, \big| \,   \omega_2 (s)   = x \big)\,=\, p^*_A \big( s,t; {\bm \nu} (A^c) \, || x||, {\bm \nu} (A^c)\, C \big) \, \mathbf{ 1}_A (\text{arg}(x))\,,
$$
 \newpage  
\noindent
whose right-hand side  does  not depend on $\,j=1,2\,$. Similarly, we observe that the transition probability 
$$
 \mathbb{Q}_j \big( || \omega_2 (t) || \in C\,, \, \text{arg} ( \omega_2 (t) ) \in A\,,   \, \tau_s (\omega_2) < t \, \big| \,   \omega_2 (s)   = x \big) 
$$
is given, with 
  $\, m := {\bm \nu } (A^c) \, \mathbf{ 1}_{A} (\text{arg}(x)) - {\bm \nu } (A )\, \mathbf{ 1}_{A^{c} } (\text{arg}(x))\,$, by the expression  
$$
\int_s^t  \mathbb{Q}_j \big( || \omega_2 (t) || \in C\,, \, \text{arg} ( \omega_2 (t) ) \in A   \, \big| \,   \tau_s (\omega_2)  = \theta  ,\,\omega_2 (s)   = x  \big)\cdot  \mathbb{Q}_j \big(   \tau_s (\omega_2)  \in \mathrm{d} \theta  \, \big| \,   \omega_2 (s)   = x \big)
$$
$$
~~~~~~~~~=\,    
\int_s^t  p_{A}\big(\theta, t; 0,  {\bm \nu}(A^{c}) C\big)\cdot \Big( - \mathrm{d}_\theta \,q_A \big( s,\theta\,; \,m\, ||x||  \big) \Big).
$$
\noindent
Once again, the expression  on the right-hand side  does  not depend on $\,j=1,2\,$. 

\smallskip
 \noindent
 $\bullet~$ 
Next, we  consider the case $\, \bm \nu (A) \in \{ 0 ,\,1\}$.  
%
%
Let  $\, \bm \nu(A) =0 \,$ first; then $\, g_{A}(x) = \Vert x \Vert \, {\bf 1}_{\{\text{arg}(x) \in A,\,\lVert x \rVert > 0\}}  ,$ and the process  $\, M_A (\cdot) = g_{A}(\omega_{2}(\cdot)) \,$ is a nonnegative, continuous $\,\mathbb Q_{j} -$local martingale, thus also a supermartingale -- so it stays at the origin $\, \mathbf{ 0}\,$ after hitting it  for the first time. It follows that with $\,\mathbb Q_{j} -$probability one, the angular part $\, \text{arg}(\omega_{2}(\cdot))  $ never again visits    the set $A\,,$ after the radial part $\,\Vert \omega_{2}(\cdot) \Vert\,$ first becomes zero. Thus for an arbitrary $\, C \in {\cal B} ((0, \infty))\,$ and for every $\, 0 \le s < t < \infty \,$, $\,x \in \mathbb R^{2}\,$, $\,j \, =\, 1, 2\,$ we have
$$
\, \mathbb{Q}_j \big( || \omega_2 (t) || \in C\,, \, \text{arg} ( \omega_2 (t) ) \in A\,,   \, \tau_s (\omega_2) <t \, \big| \,   \mathcal F_{2}(s) \big)\, =\, 0\, .
$$ 
If, on the other hand,  $\, \bm \nu(A) =1 \,$ holds, then $\, \bm \nu(A^{c}) =0 \,$ and   therefore 
$$ 
\, \mathbb{Q}_j \big( || \omega_2 (t) || \in C\,, \, \text{arg} ( \omega_2 (t) ) \in A^{c}\,,   \, \tau_s (\omega_2) <t \, \big| \,   \mathcal F_{2}(s) \big)\, =\, 0\, ,
$$  
which implies  
$$
\mathbb{Q}_j \big( || \omega_2 (t) || \in C, \, \text{arg} ( \omega_2 (t) ) \in A\,,   \, \tau_s (\omega_2) <t \, \big| \,   \mathcal F_{2}(s) \big) = \mathbb{Q}_j \big( || \omega_2 (t) || \in C,  \, \tau_s (\omega_2) <t \, \big| \,  \mathcal F_{2}(s) \big)
$$
$$
 = \mathbb{Q}_j \big( || \omega_2 (t) || \in C,  \, \tau_s (\omega_2) <t \, \big| \,   ||\omega_2 (s)||   \big).
$$
We have also the following in both cases:
$$
 \mathbb{Q}_j \big( || \omega_2 (t) || \in C\,, \, \text{arg} ( \omega_2 (t) ) \in A\,,   \, \tau_s (\omega_2) >t \, \big| \,   \mathcal F_{2}(s) \big)~~~~~~~~~~~~~~~~~\qquad 
 $$
$$
~~~~~~~~~~~~~~~~~~~\qquad  
\,=\, \mathbb{Q}_j \big( || \omega_2 (t) || \in C\,,  \, \tau_s (\omega_2) >t \, \big| \,   || \omega_2 (s) || \big)\, \mathbf{ 1}_A (\text{arg}(\omega_2 (s))) \, .
$$
Since the distribution of $\, \Vert \omega_{2}(\cdot) \Vert \,$ is uniquely determined and independent of $\,j=1,2\,$, we conclude that $\,\mathbb{Q}_j \big( || \omega_2 (t) || \in C\,, \, \text{arg} ( \omega_2 (t) ) \in A\, \big| \,   \mathcal F_{2} (s) \big)\,$ does not depend on $\,j=1,2\,$, if $\, {\bm \nu } (A) =0\,$ or $\, 1\,$.   

\smallskip
\noindent
$\bullet~$  Finally, we note that $\,\mathbb{Q}_j \big(   \omega_2 (t)   =0 \, \big| \,   \omega_2 (s)   = x \big)=\,\mathbb{Q}_j \big( || \omega_2 (t) || =0\, \big| \,   ||\omega_2 (s) ||  = ||x ||\big)\,$,   where the right-hand side is also uniquely determined.

\smallskip
 \noindent
 $\bullet~$
It is clear that the conditional distribution of $\,\omega_{2}(t)\,$ given $\, \mathcal F_{2} (s)\,$ is uniquely determined for $\, 0 \le s < t < \infty\,$. Standard arguments show then, that the finite-dimensional distributions of $\,\omega_{2}(\cdot) \,$ are uniquely determined. Therefore, the local martingale problem of the Proposition is well-posed.  \qed 
 
\begin{prop} {\bf Well-Posedness for  \textsc{Walsh} Diffusions with Drift:}  
\label{Cor: MP1}  
Under the setting of Proposition \ref{prop: MP1}, and in addition to the  assumptions imposed there,  let us consider another   function $\, {\bm c} : \mathbb R_{+}\to \mathbb R\,$ which is bounded and measurable. We denote by $\,{\mathbb Q}^{(\mathbf{ 0})}\,$ the solution to the local martingale problem of subsection \ref{sec_512} associated with the triple  $\, ({\bm \sigma}, \mathbf{ 0}, {\bm \mu}) \,$.

\smallskip
\noindent
{\it (i)} For every $\,T \in ( 0, \infty)\,$, the local martingale problem associated with the triple  $\, ({\bm \sigma}, {\bm \sigma} {\bm c}, {\bm \mu})\,$ for $\, M^{g}(t)\,$,  $\, 0 \le t \le T\,$ in (\ref{eq: MP2}), is then well posed, and its solution is given by the probability measure 
$\, \mathbb Q^{(  
{\bm c})}_{T}\, $ with 
\[
\frac{\,  {\mathrm d} \mathbb Q^{(  
{\bm c})}_{T}\, }{ \, {\mathrm d} \mathbb Q^{(\mathbf{ 0})}\, } \bigg \rvert_{\mathcal F_{2}^{\bullet}(t)}\, :=\, \exp \Big( \int^{t}_{0}   
{\bm c} ( \lVert \omega_{2}(u)\rVert ) \,{\mathrm d} W(u) - \frac{1}{\, 2\, } \int^{t}_{0}  
{\bm c}^{2} ( \lVert \omega_{2} (u) ) \rVert ) {\mathrm d} u \Big) \, ; \quad 0 \le t  \le T\, . 
\] 

\smallskip
\noindent
{\it (ii)}
Under the assumptions in {\bf (i)}, suppose that $\,\mathbb{Q}^{(  
{\bm c})}\,$ solves the local martingale problem associated with the triple $\, ({\bm \sigma}, {\bm \sigma} {\bm c}, {\bm \mu})\,$. Then there exists an $\, \mathbb{F}_{2}-$Brownian motion $\, B(\cdot)\,$, such that    every $ \mathbb{F}_{2}-$local martingale $ M(\cdot) $ with $M(0)=0$ can be  represented in the integral form   $\, M(\cdot)=\int_0^{\, \cdot} H(t) \, \mathrm{d} B (t)\, $  for some $\,\mathbb{F}_{2}-$progressively measurable and locally square-integrable process  $\, H(\cdot)\,$.
\end{prop}

\begin{proof} {\it (i)} This is a direct consequence of Propositions \ref{prop: SDEMP}, \ref{prop: MP1} and \textsc{Girsanov}'s change of measure. Indeed, it follows from Proposition \ref{prop: SDEMP} that, under $\,\mathbb Q^{(\mathbf{ 0})}\,$, the co\"ordinate process $\, \omega_{2}(\cdot) \,$ satisfies the system of stochastic integral equations   (\ref{5.5}), subject to (\ref{eq: zeroLeb}) and (\ref{eq: localDist}). Because of the boundedness of  the function    
$\,   
{\bm c} (\cdot)\,$, the measure $\, \mathbb Q^{( 
{\bm c})}_{T}\,$ just introduced is a probability. 

By \textsc{Girsanov}'s theorem (e.g., \textsc{Karatzas \& Shreve} (1991), Theorem 3.5.1) we see that for every fixed $\,T \in ( 0, \infty)\,$,   the process  
$
W^{( 
{\bm c})}(u)   :=   W(u) - \int^{u}_{0}    
{\bm c}  \big(\Vert \omega_{2}(t)\Vert\big) \, {\mathrm d} t \, ,\,\,\, 0 \le u \le T 
$
is standard Brownian motion under this probability measure $\, {\mathbb Q}^{(  
{\bm c})}_{T}$, and  thus    the co\"ordinate process $\, \omega_{2}(\cdot) \,$ satisfies on the time-horizon $\, [0,T]\,$ the system  of stochastic integral equations   
\[
  X_i(\cdot) \, =\, \mathrm{x}+ \int_0^{\, \cdot} {\mathfrak f}_i(X(t)) \, {\bm \sigma} \big( \lVert X( t) \rVert \big) \Big[ \, {\mathrm d} W^{( {\bm c})}(t) +  {\bm c}  \big( \lVert X( t) \rVert \big) \, {\mathrm d}t \, \Big] +    \gamma_i \, L^{\lVert X  \rVert}(\cdot) \,, \quad i=1, 2 \,.    
\]
Moreover, since the probability measure $\, {\mathbb Q}^{({\bm c})}_{T}$ is absolutely continuous with respect to $\, {\mathbb Q}^{({\bm 0})}$, we  obtain (\ref{eq: zeroLeb}) and (\ref{eq: localDist}) with $\, X(\cdot)\,$ replaced by $\, \omega_{2}(\cdot)\,$, a.e.\,under $\,\mathbb Q^{({\bm c})}_{T}\,$.
Thanks to Proposition \ref{prop: SDEMP} again, 
 $\, {\mathbb Q}^{({\bm c})}_{T}\,$ solves the  local martingale problem of subsection \ref{sec_512} associated with the triple $\,({\bm \sigma}, {\bm \sigma}{\bm c}, {\bm \mu}) \,$.  

Conversely, for any solution $\, {\mathbb Q}^{  
({\bm c})}_{T}\,$ to the  local martingale problem  associated with $\,({\bm \sigma}, {\bm \sigma}{\bm c}, {\bm \mu}) \,$ for  $\,M^{g}(t)\,$,  
$\, 0 \le t \le T\,$ as in (\ref{eq: MP2}), the probability measure $\,\mathbb Q^{(\mathbf{ 0})}\,$ defined via 
 \[
\frac{\,  {\mathrm d} \mathbb Q^{(  
\mathbf{ 0})}\, }{ \, {\mathrm d} \mathbb Q^{({\bm c})}_{T}\, } \bigg \rvert_{\mathcal F_{2}^{\bullet}(t)}\, :=\, \exp \Big( - \int^{t}_{0}   
{\bm c} ( \lVert \omega_{2}(u)\rVert ) \,{\mathrm d} W^{( 
{\bm c})}(u) - \frac{1}{\, 2\, } \int^{t}_{0}  
{\bm c}^{2} ( \lVert \omega_{2} (u) ) \rVert ) {\mathrm d} u \Big) \, ; \quad 0 \le t  \le T  
\] 
is  seen to solve the local martingale problem of subsection \ref{sec_512} associated with  the triple  $\, ({\bm \sigma}, \mathbf{ 0}, {\bm \mu}) \,$. Since this problem is well-posed, the same  holds for the local martingale problem associated with    $\, ({\bm  \sigma}, {\bm \sigma}{\bm c}, {\bm \mu}) \,$. 

\smallskip
\noindent 
{\it (ii)}\,\, From part {\it (i)} we know that $\,\mathbb{Q}^{(  
{\bm c})} \big\rvert_{\mathcal F_{2}^{\bullet}(T)} =\mathbb{Q}^{(  
{\bm c})}_{T} \, ,\,\, 0 \leq T < \infty\,,$ and $\, B(\cdot):= W^{( 
{\bm c})}(\cdot) \,$ is a standard Brownian motion under $\,\mathbb{Q}^{(  
{\bm c})}\,$. Since the local martingale problem in part {\it (i)} is well-posed, we can adapt the proof of Theorem 4.1 of \textsc{Barlow,  Pitman  \& Yor} (1989) to show that the required $\, H(\cdot)\,$ exists up to any finite time $\, T\,$  (see also \textsc{Jacod} (1977)), and can thus be   defined on all of $\,  [0, \infty)\,$. 
\end{proof}

\section{Martingale Characterization of the \textsc{Walsh} Brownian Motion} 
\label{subsec: WBM} 

We still have to show that, when $\, U(\cdot) \equiv B(\cdot)\,$ is standard Brownian motion, the construction of Theorem \ref{prop: skewTanaka} leads to the \textsc{Walsh} Brownian motion as defined, for instance, in \textsc{Barlow, Pitman \& Yor} (1989) or \textsc{Fitzsimmons \& Kuter} (2014). In the present section we establish this connection; cf. Proposition \ref{prop: ID}.

Following these sources, 
we may characterize the \textsc{Walsh} Brownian motion $\, {\bm W}(\cdot)\,$ in terms of its  \textsc{Feller} semigroup $\,\{\mathcal{P}_{t}\,, \, t \ge 0\}\,$ defined for $\,f \in C_{0}(\overline{E}) \,$ via 
\begin{equation} 
\label{eq: FSEMIG}
\begin{split}
\big[\mathcal{P}_{t} f \big](0, \theta) \, &:=\, T_{t}^{+} \overline{f}(0) \, , \\
\big[\mathcal{P}_{t} f\big](r, \theta) \, & :=\, T_{t}^{+} \overline{f}(r) + \big[T_{t}^{0} \big(f_{\theta} - \overline{f}\, \big)\big](r) \, ; \quad r > 0 \, , \, \, \theta \in [0, 2\pi) \, . 
\end{split}
\end{equation}  
 \noindent
  Here $\, \{T_{t}^{+}, \,0 \le t < \infty \} \,$ is the semigroup of  reflected Brownian motion on $\,[0, \infty)\,$, and $\, \{T_{t}^{0}, \,0 \le t < \infty \} \,$   the semigroup of Brownian motion on $\,[0, \infty)\,$ killed upon reaching  the origin. For the sake of simplicity, we use   polar co\"ordinates in the punctured plane $\,E\,$ of (\ref{eq: f}). Abusing notation slightly,   we  define  also
\begin{equation}
\label{abuse}
\overline{f}(r) \, :=\, \int_{[0, 2\pi)} f(r, \theta)   \, {\bm \nu}  ({\mathrm d} \theta) \, , \qquad f_{\theta}(r) \, :=\, f(r, \theta) \, ; \qquad (r, \theta) \in  \overline{E} \, ,  
\end{equation}
for $\,f \in C( \overline{E})\,$,  as in   (\ref{nu_mu}). Let us assume that $\,{\bm W}(0) \, =\, {\rm x} \in \mathbb R^{2}\,$. 
\textsc{Barlow, Pitman \& Yor} (1989) show that there is a \textsc{Feller} and strong \textsc{Markov}   process $\, {\bm W}(\cdot)$  with values in $\,\R^2\,$, continuous paths, and     $\,\{\mathcal{P}_{t}\,, \, 0 \le t < \infty \}\,$ as its semigroup. This is the process these authors call ``\textsc{Walsh} Brownian motion". They show that the radial part $\,|| {\bm W}(\cdot) ||\,$ is one-dimensional reflecting Brownian motion. For this planar process $\, {\bm W}(\cdot)\,,$ \textsc{Hajri \& Touhami} (2014) derive a version of the \textsc{Freidlin-Sheu} formula, that involves the standard, one-dimensional Brownian motion   of the filtration $\,\mathbb{F}^{{\bm W}} = \big\{ \mathcal{F}^{{\bm W}} (t) \big\}_{0 \le t < \infty}\,$, given by 
 \begin{equation}
\label{eq: new_BM}
 {\bm \beta}^{{\bm W}}(\cdot)  \,:=\, ||{\bm W}(\cdot)||- ||\mathrm{x}||- L^{||{\bm W||}}(\cdot)  \,.
\end{equation}

Here is an extension of Proposition 3.1 in \textsc{Barlow, Pitman \& Yor} (1989); we recall the Definitions \ref{def: D} and \ref{def: G}, as well as the notation of \eqref{eq: gh}.   
It shows that the \textsc{Walsh} Brownian motion with spinning measure $\,{\bm \mu} \,$, defined via the semigroup (\ref{eq: FSEMIG}),  generates a solution to the local martingale problem associated with the triple $\, (\bm 1, \bm 0, \bm \mu)\,$ (cf. Remark \ref{remark: 6.1}).

 \begin{prop} {\bf Properties of \textsc{Walsh} Brownian Motion:}
 \label{MartProb}
Let $\, {\bm W} (\cdot)\,$ be the \textsc{Walsh} Brownian motion defined via the semigroup (\ref{eq: FSEMIG})  and with spinning measure $\,{\bm \mu} \,$. Then: 
\\
 (i) The process $\, \Vert {\bm W}(\cdot) \Vert\,$ is reflecting Brownian motion; and 
$    {\bm W}(\cdot)   $ satisfies the properties in (\ref{eq: DIST3})--(\ref{eq: DIST}). 

\noindent
 (ii) For any   $\, g : \mathbb{R}^{2} \rightarrow \mathbb R\,$ 
in the class $\, \mathfrak{D}^{{\bm \mu}}\,$ of   (\ref{eq: Dmu}),  the continuous process   below is  a   local martingale:
$$
g({\bm W}(\cdot)) -g(\mathrm x)- \frac{1}{\,2\,}  \int_0^{\, \cdot } G^{\prime \prime} ({\bm W}( t)) \,\mathbf{ 1}_{ \{ {\bm W}( t) \neq \mathbf{ 0}\}}\, {\mathrm d}  t\,=\, \int_0^{\, \cdot } G^{\prime  } ({\bm W}( t)) \,\mathbf{ 1}_{ \{ {\bm W}( t) \neq \mathbf{ 0}\}}\, {\mathrm d} {\bm \beta}^{{\bm W}}( t)\,.
$$
  \end{prop}
 
 \noindent
 {\it Proof:} The claims of (i) are proved    in \textsc{Barlow, Pitman \& Yor} (1989). Claim (ii) follows  by  applying the \textsc{Freidlin-Sheu}-type formula of Theorem 1.2 in \textsc{Hajri \& Touhami} (2014)  
  to the process $\,g ({\bm W}(\cdot))$. We also note that, with the notation of \eqref{eq: gh},  both processes  below are continuous   
 martingales:
\[
{\cal M}^{{\bm W}}_{(\varphi)}(\cdot)=  g_{(\varphi)}\big({\bm W}(\cdot)\big)- g_{(\varphi)}({\mathrm x})=    \int^{\,\cdot}_{0} h_{(\varphi)} \big({\bm W}(t)\big) \,{\mathrm d}  {\bm \beta}^{{\bm W}}( t)  \,, \quad {\cal N}_{(\varphi)}^{{\bm W}}(\cdot) = \big({\cal M}_{(\varphi)}^{{\bm W}}(\cdot)\big)^2 - \langle {\cal M}_{(\varphi)}^{{\bm W}} \rangle (\cdot). \qed
\]
 
Our next result shows that, as we expected all along, \textsc{Walsh} semimartingales driven by Brownian motions $\, U(\cdot)\,$ are  \textsc{Walsh} Brownian motions defined via the semigroup (\ref{eq: FSEMIG}).  

\begin{prop} {\bf Stochastic Equations for \textsc{Walsh} Brownian Motions:}   
\label{prop: ID} 
Let us place ourselves in the context of Theorem  \ref{prop: skewTanaka},  and suppose that the semimartingale $\,U(\cdot)\equiv B(\cdot)\,$ of (\ref{eq: U})  is Brownian motion. Then the planar process $ X(\cdot) $ constructed there,  has the following properties: \\ (i) It is the unique-in-distribution weak solution, subject to the properties (\ref{eq: zeroLeb}), (\ref{eq: localDist}), of the system of stochastic integral equations in (\ref{eq: skewTanaka_2}), namely 
$\,\,
X_{i}(\cdot) =  {\rm x}_{i} + \int^{\, \cdot}_{0} \mathfrak f_{i}\big(X(t)\big) \,{\mathrm d} B(t) + \gamma_i \, L^{ \,||X||}(\cdot) \, ,  \,\,\,\, i=1,2  \,$.

\noindent
(ii)\, It is   a \textsc{Walsh} Brownian motion.

\noindent
 (iii)
Every $ \mathbb{F}^X-$local martingale $ M(\cdot) $ with $M(0)=0$   has an integral representation  $\, M(\cdot) = \int_0^{\, \cdot} H(t) \, \mathrm{d} B (t)\, $,  \,for some $\,\mathbb{F}^X-$progressively measurable and locally square-integrable process  $\, H(\cdot)\,$.   
\end{prop}
 
\begin{proof}   
The first claim follows    from Propositions \ref{prop: SDEMP},   \ref{prop: MP1} with $\, {\bm \sigma}(\cdot) \equiv 1 \,$;   the second claim, that $\, X(\cdot)\,$ is \textsc{Walsh} Brownian motion, is   a consequence of   Propositions {\ref{prop: SDEMP}}, \ref{prop: MP1}   and \ref{MartProb}.  With  $\,U(\cdot)\equiv B(\cdot) \,$ a standard Brownian motion,   Proposition \ref{prop: MART} shows that both processes below are       continuous  local martingales     $$ \,M_{(\varphi)}(\cdot) =  g_{(\varphi)}(X(\cdot))- g_{(\varphi)}( \mathrm{x} ) = \int^{\,\cdot}_{0}  
 h_{(\varphi)}(X(t)) \, {\mathrm d} B(t) \,, \quad \,N_{(\varphi)}(\cdot) =\big[M_{(\varphi) } (\cdot)\big]^{2} - \int^{\,\cdot}_{0} \big[h_{(\varphi)}(X(t))\big]^{2} {\mathrm d} t \,$$ 
(cf.\,Theorem 3.1 of \textsc{Barlow,  Pitman  \& Yor} (1989)). The third claim follows from Proposition \ref{Cor: MP1}. 
\end{proof}

 In the terminology adopted by \textsc{Mansuy \& Yor} (2006), and for a spinning measure ${\bm \mu}$ that does not concentrate on one or two   points in  $\,\mathfrak S\,$, this last property says that the natural filtration $ \mathbb{F}^X$ of the \textsc{Walsh} Brownian motion   is a  {\it weak Brownian filtration} (has the martingale representation property with respect to $U$) but {\it not a strong Brownian filtration} (cannot be generated by a Brownian motion of {\it any} dimension).

\section{Angular Dependence}  
\label{sec: AD}

Let us admit now  bounded, \textsc{Borel}-measurable co\"efficients $\,{\bm b} : \mathbb R \times [0, 2 \pi)\to \mathbb R\,$ and $\,{\bm a} : \mathbb R \times [0, 2 \pi) \to (0, \infty)\,$ which may depend  on the angular variable $\, \theta \in [0, 2\pi) \,$ in (\ref{eq: MP2}).  We assume also that $\,{\bm a} \,$ is bounded away from zero, and consider the local  martingale problem of subsection \ref{sec_512}   
but now with the infinitesimal generator re-defined as  
\begin{equation} 
\label{eq: L2g}
\mathcal L^{*} g(x) \, :=\, {\bm b} \big( \lVert x \rVert , \text{arg} (x) \big)\, G^{\prime}(x) + \frac{1}{\, 2\, } \, {\bm a} \big( \lVert x \rVert , \text{arg} (x) \big)\, G^{\prime\prime}(x) \, ; \, \, \, \, \,\,\,x \in \mathbb R^{2}\, , \, \,\,\,g \in \mathfrak D_{} \, .
\end{equation}

For every given, fixed $\, \theta \in [0, 2\pi) \,$, we set $\, {\bm \sigma}_{\theta}(r)   :=  {\bm \sigma}(r, \theta)  \,$ as well as  $\, {\bm a} (r, \theta) = [{\bm \sigma} (r, \theta)]^{2}\,$, and   define the scale function $\, {\bm p}_{\theta}(\cdot)\, $ by
\[
{\bm p}_{\theta}(r) \, =\, {\bm p}(r, \theta) \, :=\,  \int^{r}_{0} \exp \Big( - 2 \int^{\xi}_{0} \frac{ {\bm b}(\zeta, \theta)\, }{\,   {{\bm a}}(\zeta, \theta)\, } {\mathrm d} \zeta \Big) {\mathrm d} \xi \, ,  \qquad r \in [0, \infty)\,, 
\]
as well as its inverse $\, {\bm q}_{\theta}(r)= {\bm q} (r, \theta) \,$ in the radial component with $\, {\bm q}_{\theta}( {\bm p}_{\theta}(r) ) =  r\,$. We note that  these functions satisfy $\, {\bm p}_{\theta}( {0}) = 0 = {\bm q}_{\theta}({ 0}) \, $ and $\,{\bm p}_{\theta}^{\prime}( {0+}) = 1 = {\bm q}_{\theta}^{\prime}({ 0+})$;   that  $\, {\bm p}_{\theta}(\cdot) \,$ has an absolutely continuous, strictly positive derivative;    that the second derivative $\, {\bm p}_{\theta}^{\prime\prime}(\cdot)\,$ exists almost everywhere; and that   both of these derivatives are bounded. Therefore, by the generalized \textsc{It\^o} rule, we see that Theorem \ref{Gen_FS} holds   also for the function $\, {\bm p}_{\theta}(\cdot) \,$, which may not be in the class $\,\mathfrak D\,$; the same is true for the function $\, {\bm q}_{\theta}(\cdot)\,$.   
 
Let us consider   an auxiliary diffusion co\"efficient 
\begin{equation}
\label{eq: sigma^tilde}
\widetilde{{\bm \sigma}}_{\theta}(r) \, \equiv \, \widetilde{{\bm \sigma}} ( r, \theta)  \,:= \,   {\bm p}^{\prime}_{\theta}( {\bm q}_{\theta}( r)) \, {\bm \sigma}_{\theta} ( {\bm q}_{\theta}(r)) \,, \qquad  0 < r < \infty
\end{equation}
 and $\,\theta \in [0, 2\pi) \,$, and write $\, \widetilde{{\bm \sigma}} (y) \equiv  \widetilde{{\bm \sigma}} ( r, \theta)\,$ for $\, y = (r, \theta) \in \R^2.$  We introduce also  the stochastic clock
\[
{\mathcal Q}(\cdot):=
\int^{\, \cdot}_{0} \frac{ {\mathrm d} u}{\, \big[ \widetilde{ {\bm \sigma}} \big( \lVert X(u) \rVert , \Theta(u) \big)\big]^{2}\, } \qquad \text{and its inverse} \quad 
{\mathcal T}(t)   :=  \inf \big \{ v \ge 0 :  {\mathcal Q}(v)   > t \big \}  \, ; \quad 0 \le t < \infty \,.  
\, 
\]
Here $\,X(\cdot) = Z(\cdot) S(\cdot) \,$ is a \textsc{Walsh} semimartingale as constructed as in (\ref{eq: ZX}),   starting from a one-dimensional  Brownian motion $\, U(\cdot) = B(\cdot) $ in Proposition \ref{prop: ID}. In particular, $  X(\cdot)$ is   \textsc{Walsh} Brownian motion; whereas $\,\Theta(\cdot)= \text{arg} ( X   (\cdot))  \,$ is as in (\ref{eq: Theta}).  We consider now the time-changed, rescaled version  $\,Y(\cdot)   =  (Y_{1}(\cdot), Y_{2}(\cdot))^{\prime}\,$ of this \textsc{Walsh} Brownian motion $ X(\cdot)  $,     defined in  polar co\"ordinates  via  
\begin{equation}  \label{eq: Y}
\lVert Y(\cdot) \rVert  \, :=\,  {\bm q} \big( \lVert X ( {\mathcal T}(\cdot)) \rVert , \text{ arg} ( X( {\mathcal T} (\cdot))) \big) \, , \qquad  \text{arg} (Y(\cdot) ) \, :=\, \text{arg} ( X( {\mathcal T} (\cdot))) \, =\, \Theta( {\mathcal T} (\cdot)) \, . 
\end{equation}
In terms of this rescaling, we have the representation 
\begin{equation}  
\label{eq: inv_time_change}
\mathcal{T} (\cdot) =  \int_0^{\, \cdot} \Big(  {\bm p}^{\prime}_{\theta}( r) \, {\bm \sigma}_{\theta} ( r) \Big)^2\Big\vert _{\,\theta = \text{arg}(Y(t)), \,r = \lVert Y(t)\rVert}  \, \, \mathrm{d} t\, 
\end{equation}
   for the inverse clock. The resulting process $\,Y(\cdot) \,$  turns out to be a \textsc{Walsh} semimartingale with  angular dependence in its local characteristics $\,({\bm \sigma}, {\bm b}, {\bm \mu})\,$, as  follows.

\begin{prop} 
\label{prop: AD} 
The process $\,Y(\cdot) \,$ defined in (\ref{eq: Y})     
satisfies the   integral equations 
\begin{equation} 
\label{eq: dynY}
Y(\cdot)  \, =\,  Y(0) + \int^{\, \cdot}_{0} {\mathfrak f}(Y(t)) \big[ \, {{\bm \sigma}}  \big( \lVert Y (t) \rVert , \text{arg} (Y(t)) \big)\, {\mathrm d} W(t)
+  {\bm b}  \big(\lVert Y (t) \rVert , \text{arg} (Y(t)) \big) {\mathrm d} t \, \big] + {\bm \gamma} \,  L^{ \lVert Y \rVert}(\cdot) ~~
\end{equation}
\begin{equation*} 
\label{eq: ||Y||0} 
\lVert Y(\cdot) \rVert \, =\,  \lVert Y(0) \rVert + \int^{\cdot}_{0} {\bf 1}_{\{ \lVert Y(t) \rVert > 0 \}} \, \big( {\bm b} ( \lVert Y(t) \rVert, \text{arg} ( Y(t)) ) {\mathrm d} t + {\bm \sigma} ( \lVert Y(t) \rVert , \text{arg} ( Y(t))) {\mathrm d} W(t)\big) + L^{||Y||} (\cdot) \,  
\end{equation*} 
as well as  the properties   
 $\, 
\int^{\cdot}_{0}{\bf 1}_{\{ Y(t) =  0 \}} {\mathrm d} t  \equiv  0 \,$ and $\,L^{R^{A}_{\ast}} (\cdot)   \equiv   {\bm \nu}(A) \, L^{ \lVert Y \rVert}(\cdot) \,, \,\, \forall \,\,\, A \in \mathcal B ([0, 2\pi))\,. 
$   
   Furthermore, it   induces a solution to the local martingale problem associated with the triple $\, ({\bm \sigma}, {\bm b}, {\bm \mu}) \,$ and $\,\mathcal L^{\ast}\,$    in (\ref{eq: L2g}).

In the above expressions $\,{\mathfrak f} = ({\mathfrak f}_{1}, \mathfrak f_{2})^{\prime}\,$
is defined in (\ref{eq: f}), $ \,W(\cdot)\,$ is one-dimensional   Brownian motion, and the ``thinned'' process $\, \,R^{A}_{\ast}(\cdot)  :=  \lVert Y (\cdot) \rVert \cdot {\bf 1}_{A}(\text{arg}(Y(\cdot)) \, \,$ is defined for $\, A \in \mathcal B ([0, 2\pi))\,$. 
\end{prop}
  \newpage  

\noindent
{\it Proof:} 
Applying the \textsc{Freidlin-Sheu} formula in Theorem \ref{Gen_FS} to $\, {\bm q}(X(\cdot))\,$, we obtain 
\[
\lVert Y(\cdot) \rVert \, =\, {\bm q} ( X({\mathcal T}(\cdot))) \, =\,  {\bm q} ({\mathrm x}) + \int^{{\mathcal T}(\cdot)}_{0} {\bm Q}^{\prime}(X(u)) \,  {\bf 1}_{\{ X(u) \neq {\bm 0} \}} 
{\mathrm d} B(u) + \Big( \int^{2\pi}_{0} {\bm q}^{\prime}_{\theta}(0+) {\bm \nu} ({\mathrm d} \theta) \Big) \cdot L^{ \lVert X\rVert }({\mathcal T}(\cdot)) 
\]
\begin{equation} \label{eq: VertY}
~~~~~~~~~~~~~~~~~~~~~~~~~~~~~~~~~~~~~~~~~~~ + \frac{1}{\, 2\, } \int^{{\mathcal T}(\cdot)}_{0} {\bm Q}^{\prime\prime}(X(u)) \, {\bf 1}_{\{ X(u) \neq 0 \}} {\mathrm d} u \,  .  
\end{equation}

 \smallskip
 \noindent
Here by direct calculation 
\begin{equation} \label{eq: derivatives}
\,{\bm Q}^{\prime}(x) \, :=\, {\bm q}_{\theta}^{\prime}(r) \, =\, \frac{1}{ \, {\bm p}_{\theta}^{\prime}\big({\bm q}_{\theta}(r)\big)\,} \, , \qquad {\bm Q}^{\prime\prime} (x) \, :=\, {\bm q}_{\theta}^{\prime\prime}(r) \, =\, \frac{ 2\, {\bm b}({\bm q}_{\theta}(r), \theta)}{ \,     \,{\bm a} \big({\bm q}_{\theta}(r), \theta\big) \cdot \big( {\bm p}_{\theta}^{\prime}({\bm q}_{\theta}(r)) \big)^2\,   \,}
\end{equation}
hold for every $\, x = (r, \theta) \, $, where $\,r\,$ is not in a set of Lebesgue measure zero that depends on $\,\theta \in [0, 2\pi)\,$. 
  Thanks to the \textsc{P. L\' evy} Theorem, the continuous local martingale 
\begin{equation*} 
\label{eq: W0}
W(\cdot) \, :=\, \int^{{\mathcal T}(\cdot)}_{0} \frac{{\mathrm d} B(u)}{\,\widetilde{\,  {\bm \sigma}}( \lVert X(u) \rVert , \text{arg} ( X(u)))\,\, }  \, 
\end{equation*}
is one-dimensional standard Brownian motion. Since $\, \, \text{Leb} (\{ t: X(t) =  {\bm 0} \}) \, =\, \text{Leb} ( \{t: S(t) =  0\}) = 0 \,  $ 
 a.s. and $\,{\bm q}({\bm 0}) \, =\,  0 \,$ from the construction, we obtain 
\begin{equation} \label{eq: Leb}
\,  \text{Leb} \big(\{t: \lVert Y(t) \rVert   = {\bm q} (X({\mathcal T}(t)))  =   {\bm 0}   \} \big) \, =\, \text{ Leb} \big( {\mathcal T}^{-1} \{ t : X(t) \, =\,  {\bm 0}  \}\big) \, =\,  0 \, \,  \text{ a.s. }
\end{equation}
In conjunction with the definitions (\ref{eq: sigma^tilde}) and (\ref{eq: Y}),   we obtain now the representations
$$
 \int^{{\mathcal T}(\cdot)}_{0} {\bm Q}^{\prime}(X(u)) \,  {\bf 1}_{\{ X(u) \neq {\bm 0} \}} {\mathrm d} B(u) \, =\, \int^{\cdot}_{0} {\bm Q}^{\prime}(X( {\mathcal T}(u) ) )   \,  {\bf 1}_{\{ X({\mathcal T}(u)) \neq {\bm 0} \}} \,{\mathrm d} B({\mathcal T}(u)) 
 $$
\begin{equation} 
\label{eq: Qp}
\, =\, \int^{\cdot}_{0}  {\bf 1}_{\{ \lVert Y(u) \rVert > 0 \}} \, {\bm \sigma} \big( \lVert Y(u) \rVert, \text{arg} (Y(u))\big) \,{\mathrm d} W(u) 
\end{equation}
  \noindent
(on the strength of Proposition 3.4.8 in \textsc{Karatzas \& Shreve} (1991)), as well as 
$$
\int^{{\mathcal T}(\cdot)}_{0} {\bm Q}^{\prime\prime}(X(u)) \, {\bf 1}_{\{ X(u) \neq 0 \}} \,{\mathrm d} u  \, =\, \int^{\cdot}_{0} {\bm Q}^{\prime\prime}(X({\mathcal T}(u)))\, \frac{ \, {\mathrm d} {\mathcal T} (u)\, }{ {\mathrm d} u}  \,  {\bf 1}_{\{ X({\mathcal T}(u)) \neq {\bm 0} \}} \, {\mathrm d} u
 $$
\begin{equation} \label{eq: Qpp}
  \, =\,  2\, \int^{\cdot}_{0}  {\bf 1}_{\{ \lVert Y(u) \rVert > 0 \}} \,{\bm b} \big( \lVert Y(u) \rVert, \text{arg}(Y(u)\big)\, {\mathrm d} u  
\end{equation}
(by time-change). From these considerations and (\ref{eq: Leb}) we also obtain the identification of local time 
\begin{equation} \label{eq: L^|| Y||}
\, L^{|| Y||} (\cdot)\, =\int_0^{\,\cdot} \mathbf{ 1}_{ \{ Y(u) =0\} } \, \mathrm{d} || Y|| (u) \,=\, \Big( \int^{2\pi}_{0} {\bm q}^{\prime}_{\theta}(0+) {\bm \nu} ({\mathrm d} \theta) \Big) \cdot L^{ \lVert X\rVert }({\mathcal T}(\cdot)) \, \,,
\end{equation}
  thus also the dynamics for the radial part of the process $\, Y(\cdot)\,$, namely 
$$
\lVert Y(\cdot) \rVert \, =\,  \lVert Y(0) \rVert + \int^{\cdot}_{0} {\bf 1}_{\{ \lVert Y(t) \rVert > 0 \}} \, \big( {\bm b} ( \lVert Y(t) \rVert, \text{arg} ( Y(t)) ) {\mathrm d} t + {\bm \sigma} ( \lVert Y(t) \rVert , \text{arg} ( Y(t))) {\mathrm d} W(t)\big) + L^{||Y||} (\cdot).
$$

\noindent
$\bullet~$ Recalling \eqref{eq: Y}, and applying the \textsc{Freidlin-Sheu} formula in Theorem \ref{Gen_FS} to the process 
$\,Y_{i} (\cdot) \, =\, {\bm q}( X( \mathcal T(\cdot)) ) \,\mathfrak f_{i}(X( \mathcal T(\cdot)))\,$, 
we obtain 
\[
Y_{i}(\cdot) \, =\, {\mathrm y}_{i} + \int^{ \mathcal T(\cdot)}_{0} {\bm Q}^{\prime} (X(u)) \mathfrak f_{i}(X(u)) \, {\bf 1}_{\{X(u) \neq {\bm 0}\}} {\mathrm d} B(u) + \frac{1}{\,2\,} \int^{ \mathcal T(\cdot)}_{0} {\bm Q}^{\prime\prime}(X(u)) \mathfrak f_{i}(X(u)) {\bf 1}_{\{X(u) \neq {\bm 0}\}} {\mathrm d} u \,  
\]
\[
+ \Big( \int^{2\pi}_{0} {\bm q}_{\theta}^{\prime}(0+) \cos \Big( \theta -  \frac{\, \pi\, }{2} (i-1)\Big )   
 {\bm \nu}({\mathrm d} \theta) \Big)  L^{ \lVert X\rVert}(\mathcal T(\cdot)) \, ; \quad i \, =\, 1, 2 \, .  
\]
 
\smallskip
\noindent
Hence, combining   with  (\ref{eq: Qp})-(\ref{eq: L^|| Y||}) and $\, {\mathfrak f}({\bm 0}) =  0\,$ and $\, {\bm q}_{\theta}^{\prime}({ 0+})=1\,$, we obtain the   dynamics (\ref{eq: dynY}). 
 \newpage  

\smallskip

\noindent $\bullet\,$ Furthermore, for every $\, g \in {\mathfrak D}\,$  by another application of the \textsc{Freidlin-Sheu} formula in Theorem \ref{Gen_FS} to $\, 
\,g(Y(\cdot)) \, =\,  g({\bm q}(r, \theta), \theta)\big \vert_{\,r \,=\, \lVert X(\mathcal T (\cdot)\rVert\, , \, \theta \,=\, \text{arg} (X(\mathcal T(\cdot)))} \, 
\, $ with $\,{\bm q}_{\theta}(0+) \, =\,  0\,$, we derive 
\[
g(Y(T)) \, =\,  g({\mathrm y}) + \int^{T}_{0} {\bf 1}_{\{ Y(t) \neq {\bm 0} \}}\, \big( {\bm b} ( \lVert Y(t) \rVert, \text{arg} ( Y(t)) ) G^{\prime}(Y(t)) + \frac{1}{\, 2\, } {\bm a} ( \lVert Y(t) \rVert, \text{arg} ( Y(t)) ) G^{\prime\prime}(Y(t))  \big) {\mathrm d} t 
\]
\begin{equation} 
\label{eq: GenFS2}
+ \int^{T}_{0} {\bf 1}_{\{ Y(t) \neq {\bm 0} \}}\, G^{\prime}(Y(t)){\bm \sigma} ( \lVert Y(t) \rVert , \text{arg} ( Y(t))) {\mathrm d} W(t) + \Big( \int^{2\pi}_{0} g^{\prime}_{\theta}(0+)\, {\bm \nu} ({\mathrm d} \theta)\Big) \cdot L^{ \lVert Y \rVert}(T) \, , \quad 0 \le T < \infty   \, ,
\end{equation}
in conjunction with (\ref{eq: derivatives})-(\ref{eq: L^|| Y||}) and $\, {\bm q}_{\theta}^{\prime}({ 0+})=1\,$.  When $\, g \in {\mathfrak D}^{{\bm \mu}} $, we can apply this to $\,M^{g}(\cdot \, ; Y) \,$ in (\ref{eq: MP2}) --      now  redefined with the operator $\,\mathcal L^{*} \,$ of (\ref{eq: L2g}) --    to conclude that   
  $\, M^{g}(\cdot \,; Y)\,$ is equal to the local martingale  
\[
 g(Y(\cdot)) - g({\mathrm y}) - \int^{\cdot}_{0} {\mathcal L}^{*} g(Y(t))  \, {\bf 1}_{\{Y(t) \neq \mathbf{ 0} \}} {\mathrm d} t \, =\, \int^{\cdot}_{0} G^{\prime}(Y(t)) \, {\bf 1}_{\{ Y(t) \neq \mathbf{ 0} \}} \, {\bm \sigma} \big( \lVert Y(t) \rVert , \text{arg} (Y(t)) \big) \, {\mathrm d} W(t). 
\]
     Therefore, the local martingale problem associated with the triple $\, ({\bm \sigma}, {\bm b}, {\bm \mu}) \,$ and the second-order differential operator $\,\mathcal L^{*}\,$    in (\ref{eq: L2g}), is seen to have a solution. The properties of $\,Y(\cdot)\,$  are now verified readily. \qed

 \begin{prop} 
\label{prop: AD_2} 
With the  assumptions and notation of this section,  the local martingale problem of subsection \ref{sec_512},  associated with the triple $\, ({\bm \sigma}, {\bm b}, {\bm \mu}) \,$   and   the  operator $\,\mathcal L^{*}\,$ in (\ref{eq: L2g}), is well-posed.   
\end{prop}

\noindent
{\it Proof:} Existence of a solution to this local martingale problem is established by Proposition \ref{prop: AD}.  

To prove uniqueness, we can reverse the steps of the construction in Proposition \ref{prop: AD}, as follows.  Consider {\it any} solution of  the local martingale problem of subsection \ref{sec_512}, associated with the triple and $\, ({\bm \sigma}, {\bm b}, {\bm \mu}) \,$   and   the  operator $\,\mathcal L^{*}\,$, and the co\"ordinate process $\, Y(\cdot) := \omega_2 (\cdot)\,$ on the canonical space for that problem.  We introduce the time change $\, 
\mathcal{T} (\cdot)\,$ as in (\ref{eq: inv_time_change}),  along with its inverse $\, \mathcal{Q} (\cdot)\,;$ as well as the time-changed, rescaled version  $\,X(\cdot)   =  (X_{1}(\cdot), X_{2}(\cdot))^{\prime}\,$ of the process  $ Y(\cdot)  $,     defined in  polar co\"ordinates  via  
\begin{equation}  \label{eq: X}
\lVert X(\cdot) \rVert  \, :=\,  {\bm p} \big( \lVert Y ( {\mathcal Q}(\cdot)) \rVert , \text{ arg} ( Y( {\mathcal Q} (\cdot))) \big) \, , \qquad  \text{arg} (X(\cdot) ) \, :=\, \text{arg} ( Y( {\mathcal Q} (\cdot)))   \, . 
\end{equation}
Using Proposition \ref{prop: SDEMP} (rather, its obvious generalization to co\"efficients with angular dependence) and Theorem \ref{Gen_FS}, we have for the  planar process $\, Y(\cdot)\,$ the appropriate \textsc{Freidlin-Sheu-}formula. With this at hand,   the planar process $\,X(\cdot)\,$    is seen to be a \textsc{Walsh} Brownian motion with spinning measure $\,\bm\mu\,$, in a  manner similar to that in the proof of Proposition \ref{prop: AD}. The path $\, t \mapsto X(t) \,$ is, with probability one, continuous in the topology induced by the tree metric (\ref{eq: tree}), and hence so is the path $\,t \mapsto Y(t)\,$. In terms of this   \textsc{Walsh} Brownian motion,  we can express the   time change $\, \mathcal{Q} (\cdot)\, $ as 
$$
\mathcal{Q} (\cdot)\,=\int^{\, \cdot}_{0} \frac{ {\mathrm d} u}{\, \big[ \widetilde{ {\bm \sigma}} \big( \lVert X(u) \rVert , \text{arg} (X( u) \big)\big]^{2}\, }\,.
$$

The crucial step now, is to note that the process $\,Y(\cdot) \,$ can be written as $\,Y(t) \, =\, \Psi_{t}(X(\cdot))\,$. Here $\,\Psi_{\cdot} \, $  
is a measurable mapping  defined by $\, \Psi_{t}(\omega_{2}) \, =\, {\bm q}\big(\Pi_{\mathcal T(t; \omega_{2})} (\omega_{2}) \big)\, $, in terms of the measurable projection mapping $\,\Pi_{t}(\omega_{2}) \, :=\, \omega_{2}(t)\,$ and the continuous time change 
\[
\mathcal T(t; \omega_{2}) \, :=\,  \inf \bigg \{ v \ge 0 \, :\, \int^{v}_{0} \frac{ {\mathrm d} u }{\,  \big[ \widetilde{\bm \sigma} ( \lVert \omega_{2} (u) \rVert, \text{arg} ( \omega_{2} (u)))\big]^{2}\, } > t \bigg \} \, , \qquad 0 \le t < \infty \,.  
\] 
Since the distribution of the \textsc{Walsh} Brownian motion $\, X(\cdot) \,$ is uniquely determined (see  section \ref{subsec: WBM}), the distribution of $\,Y(\cdot) \,$ is also determined uniquely from these considerations.

We conclude that the local martingale problem associated with the triple  $\,({\bm \sigma}, {\bm b}, {\bm \mu})\,$ is well-posed.  \qed
 \newpage

\section{The Time-Homogeneous Strong \textsc{Markov} Property}
 \label{SMP}

From section \ref{subsec: WBM},   we know that the unique solution to the well-posed local martingale problem associated with the triple $\, (\mathbf{ 1}, \mathbf{ 0},  {\bm \mu}) \,$ induces a \textsc{Walsh} Brownian motion, which is  a time-homogeneous strong \textsc{Markov} process as shown in   \textsc{Barlow,  Pitman  \& Yor} (1989). We   generalize this result in subsection \ref{subsec: LMP}, by showing that every solution to a well-posed local martingale problem as in subsection \ref{sec_512}, associated with a triple $\, (\bm \sigma , \bm b,  {\bm \mu})  ,$ induces a time-homogeneous strong \textsc{Markov} process.

\smallskip
Next, we   pick up the thread of  Proposition \ref{prop: SDEMP}(a), and try to see what we can say about   solutions to {\it the system of stochastic equations (\ref{eq: SDEMP}) for given $(\gamma_1, \gamma_2) \in \R^2$,  subject only to the non-stickiness condition (\ref{eq: zeroLeb}).}  We find that for some such  solutions there is no   spinning measure  $\, \bm \mu \,$ such   that  the ``thinning condition" (\ref{eq: localDist}) is satisfied. We show that  the time-homogeneous strong \textsc{Markov} property can be used to rule out these solutions. Then for every solution with an appropriate version of this  property,  we prove the existence of a spinning measure $\, \bm \mu \,$   for which the local martingale problem associated with the triple $\, (\bm \sigma , \bm b,  {\bm \mu}) \,$ is solved by the distribution of the state process $\,X(\cdot)\,$ in the solution. In this spirit we obtain in subsection \ref{subsec: THSMP} a similar conclusion as in Part (a) of Proposition \ref{prop: SDEMP}, but with the notable difference that here $\, \bm \mu \,$ is not given in advance; its existence is established in the proof of Theorem \ref{thm: Gen}, the next major result of this work. As a corollary of this result, we show in subsection \ref{subsec: THSMP_WBM} that with $\bm b = \mathbf{ 0} ,\, \bm {\sigma} = \mathbf{ 1}$   the equations (\ref{eq: SDEMP}), subject to (\ref{eq: zeroLeb}) and to the time-homogeneous strong \textsc{Markov} property, characterize \textsc{Walsh} Brownian motions. 

\smallskip
Throughout this section, we shall always refer to subsection \ref{sec_511} for local submartingale problems associated with   pairs $\, (\bm \sigma , \bm b) \,$ (corresponding to one-dimensional reflected diffusions), and   to subsection \ref{sec_512} for local martingale problems associated with   triples $\, (\bm \sigma , \bm b,  {\bm \mu}) \,$ (corresponding to planar diffusions). 

\subsection{On Well-posed Local Martingale Problems}   \label{subsec: LMP}

\begin{definition}
\label{def: THSM}
 Given a filtered probability space $\,( \Omega, \mathcal {F}, \mathbb P) \,$, $\, \mathbb F= \big\{ \mathcal F({t})\big\}_{0 \le t < \infty} \,$, we say that a progressively measurable process $X(\cdot)$ with values in some Euclidean space $\,\mathbb R^{d}\,$ is {\it time-homogeneous strongly Markovian} with respect to it  if, for every stopping time $\, T \,$ of $\,\mathbb F\,$, real number $\,t \ge 0\,$, and set $\,\Gamma \in \mathcal B (\mathbb R^{d})\,$,  
\[
\mathbb P \big(     X( T + t) \in \Gamma \, \vert \, \mathcal F (T)  \big) \, =\, \mathbb P  \big( X(T+t) \in \Gamma  \, \vert  \, X(T) \big) \, =\, \mathfrak g \big( X(T)\big)    \quad \text{holds}\,\,\mathbb P-\text{a.e. on } \{T < \infty\} \, . 
\]
Here $\,\mathfrak g: \mathbb R^{d} \to \mathbb R \,$ is a bounded measurable function that may depend  on $\, t \,$ and $\, \Gamma \,$, but not on $\, T$.    \qed
 \end{definition}
 
It is clear that every strong \textsc{Markov} process with a one-parameter transition semigroup is time-homoge- neous strongly Markovian. Also, a diffusion, or a strong \textsc{Markov} family, is time-homogeneous strongly Markovian under every probability measure in the family (see Definition 5.1, Chapter IV of \textsc{Ikeda \& Watanabe} (1989), and Definition 2.6.3 in \textsc{Karatzas \& Shreve} (1991)). We show here that every solution to a well-posed local martingale problem associated with the triple $\, (\bm \sigma , \bm b,  {\bm \mu}) \,$ induces a time-homogeneous strongly Markovian process. This is an extension of Theorem 5.4.20 in \textsc{Karatzas \& Shreve} (1991) in the context of subsection \ref{sec_512}. Its proof given here  is in the same context.

\begin{prop}
\label{SM,Unique}
Suppose that the local martingale problem associated with the triple $(\bm \sigma, \bm b, \bm \mu)$ is well-posed, and let $\, {\mathbb Q^{\rm x}} \,$ be its solution  with $\,\omega_{2}(0)=\rm x\,$, $\, \mathbb Q^{\rm x}-$a.e. Then for every stopping time $T$ of $\,\mathbb F_{2}\,$, $C \in \mathcal F_{2}\,$, and $\,\rm x \in \mathbb R^{2}$, the process $\,\omega_{2}(\cdot)\,$ satisfies the property 
\[
\mathbb Q^{\rm x} \big( \theta_{T}^{-1} C\, \big \vert \,\mathcal F_{2}(T) \big) (\omega_{2})= \mathbb Q^{\,\omega_{2}(T)} \big( C \big) \, , \,\,\,\, ~~~ \mathbb Q^{\rm x}-\textit{a.e.} \,\,\, \textit{on} \,\,\{ T < \infty \} ,
\]
where $\theta_{T}$ is the shift operator  $(\theta_{T}(\omega_{2}))(\cdot):=\omega_{2}(T(\omega_{2})+\cdot)\,$. In particular, $\,\omega_{2}(\cdot)\,$ is time-homogeneous strongly Markovian with respect to $\,(\Omega_{2}, \mathcal F_{2}, \mathbb Q^{\rm x})\,$ and the filtration $\,\mathbb F_{2}\,$, \,for every $\,\rm x \in \mathbb R^{2}$. 
\end{prop}
 
We shall need a countable determining class for our local martingale problem, so we introduce it next. A  crucial result in this regard, Lemma \ref{Count} below, is proved in an Appendix, section \ref{sec: App_Lem}. 
\newpage

\begin{definition}
 \label{Class_E}
We shall denote by $\, \mathfrak E \subseteq  \mathfrak D^{\bm \mu}_{+}\,$   the collection   that consists  of  

\noindent
(i)\, the functions $\,g_{A}(x)\,: =\,    \lVert x   \lVert \,\big( {\bf 1}_A ( \text{arg} (x) ) - {\bm \nu}(A) \big) \,$ as in (\ref{eq: G3}), where $\, A \subset [0, 2 \pi)\,$ is of the form $\, [a, b) \,$ and $\, a, b \,$ are rational numbers; and of 

\noindent
(ii)\, the following functions in $\,   \mathfrak D^{\bm \mu}_{+}\,$   used in the proof of Part (b) of Proposition \ref{prop: SDEMP}: namely, $\, g_{1}, \, g_{2},\,g_{i,k},\, 1\leq i, k \leq 2 \,$  in (\ref{eq: G1}); $  \, g_{1,1}^{\circ}, \, g_{2,2}^{\circ}, \, g_{3} \,$ in (\ref{eq: G2});  as well as, for every rational $\,c_{1} > 0\,$, a function  $\, g_{4}  \in \mathfrak D^{\bm \mu} \,$   of the form $\, g_{4}(r, \theta) \, =\, \psi(r)\,$ where $\, \psi : [0, \infty) \to [0, \infty) \,$ is smooth with $\,\psi(r) \, =\, r\,$ for $\,r \geq c_{1}\,$.

 In particular,   $\,\mathfrak E \,$ is    a {\it  countable collection.}  \qed
\end{definition}

\begin{lm}
\label{Count} 
    Suppose $\,\mathbb Q\,$ is a probability measure  on $\,(\Omega_{2}, \mathcal F_{2})\,$ with $\, \omega_{2}(0) = \rm x\,,$ $\, \mathbb Q-$a.e., under which   
    $\,M^{g} (\cdot\,;\omega_2)\,$ is a continuous local martingale (resp., submartingale) of the filtration $\, \mathbb F_{2} \,$ for every function $\,g \in  \mathfrak D^{\bm \mu} \cap \mathfrak E$ (resp., $    \mathfrak E$). Then this is also true for every function $\,g \in  \mathfrak D^{\bm \mu}$ (resp., $\mathfrak D^{\bm \mu}_{+} $).
\end{lm}

\noindent 
{\it Proof of Proposition \ref{SM,Unique}:} We   proceed as in \textsc{Karatzas \& Shreve} (1991), proof of Theorem 5.4.20, including Lemma 5.4.18 and Lemma 5.4.19. It is easy to check that all the arguments there apply to our context (with   some standard localization   and application of optional sampling   to submartingales), except for the final step of the proof of Lemma 5.4.19. To get through it, we only need to find a countable collection $\,\mathfrak E\,\subset   \mathfrak D^{\bm \mu}_{+}\,$  with the property that, in order to show that  $\,M^{g} (\cdot\,;\omega_2)\,$ is a continuous local martingale (resp., submartingale) for every function $\,g \in  \mathfrak D^{\bm \mu}$ (resp., $\mathfrak D^{\bm \mu}_{+}$), it suffices to have these properties for  all functions in $\,\mathfrak E\,$. 
We appeal now to Lemma \ref{Count}, and the proof of Proposition \ref{SM,Unique} follows. \qed

\subsection{Time-homogeneous Strongly Markovian Solutions to   (\ref{eq: SDEMP}), under only (\ref{eq: zeroLeb})}   
\label{subsec: THSMP}

Let us recall Part (a) of Proposition \ref{prop: SDEMP}. Suppose that we do not specify a measure $\, \bm \mu \,$ in advance, and that the ``thinning" condition (\ref{eq: localDist}) is not imposed. In particular, {\it with given \textsc{Borel}-measurable functions $\,{\bm b} : [0, \infty) \to \mathbb R\,$, $\,{\bm \sigma} : [0, \infty) \to \R \setminus \{ 0\}\,$ and real numbers $\, \gamma_{i},\, i=1,2\,$, we consider the system of stochastic equations (\ref{eq: SDEMP}) subject only to the  non-stickiness  condition (\ref{eq: zeroLeb}).}     
 
From Part (b) of Proposition \ref{prop: SDEMP} we know that, for a probability measure $\, \bm \mu \,$ on $\,(\mathfrak S,\, \mathcal B(\mathfrak S)) \,$ with  
\begin{equation}
\label{gamma0}
\gamma_i = \int_{ \mathfrak S}  \mathfrak f_i (z) \,  {\bm \mu} (\mathrm{d} z)\,, \qquad i=1,2\, ,
\end{equation}
  every solution to the local martingale problem associated with the triple $\, (\bm \sigma, \bm b, \bm \mu) \,$ induces a solution to the system (\ref{eq: SDEMP}), subject to (\ref{eq: zeroLeb}). {\it But can we obtain all the solutions of (\ref{eq: SDEMP}), (\ref{eq: zeroLeb}) in this way? }

The answer is negative: There are usually several probability measures $\, \bm \mu \,$ satisfying (\ref{gamma0}), so  we can construct a   solution to (\ref{eq: SDEMP}) that satisfies  (\ref{eq: zeroLeb}) and features two different ``spinning measures", both satisfying (\ref{gamma0}). Then this solution is not related to that of a local martingale problem associated with the triple $\, (\bm \sigma, \bm b, \bm \mu) \,$, for {\it any} $\, \bm \mu \,$. The construction will be given in detail at the end of this subsection (Remark \ref{Mix_Ex}).

Interestingly, if we restrict our scope to solutions with some appropriate time-homogeneous strong \textsc{Markov} properties, then each solution to (\ref{eq: SDEMP}) subject to the  non-stickiness  condition (\ref{eq: zeroLeb})  {\it is}  related to that of a local martingale problem associated with the triple $\, (\bm \sigma, \bm b, \bm \mu) \,$, for some $\, \bm \mu\,$ that depends on this solution.    This is the main result of the present subsection, Theorem \ref{thm: Gen} below.

Before stating this result, we   note that Proposition 2.6.6 ($c^{\prime}$) in \textsc{Karatzas \& Shreve} (1991), proved for strong \textsc{Markov} families, admits a version for continuous, time-homogeneous strongly Markovian processes with exactly the same proof.  We state this version here; it will be used several times in what follows.

\begin{prop}
\label{prop: THSMP}
Suppose $\, X(\cdot) \,$ is continuous and time-homogeneous strongly Markovian, in the sense of Definition \ref{def: THSM}. 
Then  for every set $\, B \in \mathcal B(C[0, \infty)^{d})$  we have
\[
\mathbb P \big(  \,  X( T + \cdot) \in B \, \vert \, \mathcal F (T) \, \big) \, =\, \mathbb P  \big( X(T+\cdot) \in B \, \vert \, X(T) \big) \, =\, \mathfrak h ( X(T))   \, , \quad \mathbb P-\text{a.e. on } \{T < \infty\} \, .
\]
for some bounded, measurable function $\,\mathfrak h: \mathbb R^{d} \to \mathbb R \,$   that may depend    on the set $\, B\,$, but not on $\, T$.
\end{prop}

We   present now the main result of this section. Its proof is given in an Appendix, section \ref{sec: App}. 

  \newpage

\begin{thm} {\bf Strongly Markovian Solutions of   (\ref{eq: SDEMP}),   (\ref{eq: zeroLeb}):}
\label{thm: Gen}
Let us consider a weak solution $\, (X(\cdot), W(\cdot)),$ $( \Omega, \mathcal {F}, \mathbb P)$, $ \mathbb F= \big\{ \mathcal F({t})\big\}_{0 \le t < \infty}\,$    of the system  (\ref{eq: SDEMP}) for some given constants $\, \gamma_1, \,\gamma_2\,,$    namely 
\begin{equation*} 
X_{i}(\cdot) \, =\, X_{i}(0) + \int^{\,\cdot}_{0} {\mathfrak f}_{i}(X(t)) \Big[ {\bm b} ( \lVert X(t) \rVert ) {\mathrm d} t\, +\, {\bm \sigma} ( \lVert X(t) \rVert) {\mathrm d} W(t) \Big] + 
\gamma_i\, L^{ \lVert X\rVert}(\cdot) \,,  \quad  i = 1, 2  \,.
\end{equation*}
{\it (i)} \, The  radial part $\, \lVert X(\cdot) \rVert\,$ of the state process solves then the equation (\ref{eq: ||X||}), namely 
\begin{equation*}
\label{||X||2}
\lVert X(\cdot) \rVert \, =\,  \lVert X(0) \rVert + \int^{\cdot}_{0} {\bf 1}_{\{ \lVert X(t) \rVert > 0 \}} \, \Big( {\bm b} ( \lVert X(t) \rVert) {\mathrm d} t + {\bm \sigma} ( \lVert X(t) \rVert) {\mathrm d} W(t)\Big) + L^{||X||} (\cdot) \, . 
\end{equation*}
{\it (ii)} \, If both $\, X(\cdot)\,$ and $\, \Vert X(\cdot) \Vert \,$ are time-homogeneous, strongly Markovian processes with respect to   $\,\mathbb F^{X}=\big\{ \mathcal F^{X}({t})\big\}_{0 \le t < \infty}\,$ and  (\ref{eq: zeroLeb}) holds, then there exists a probability measure $\,\bm \mu\,$ on $\,(\mathfrak S,\, \mathcal B(\mathfrak S)) \,$     such that   $\, X(\cdot)\,$ induces a solution to the local martingale problem associated with the triple $\, (\bm \sigma, \bm b, \bm \mu) \,$. 

\smallskip
\noindent
{\it (iii)} \,If,  in addition, the state process of this weak solution satisfies the analogue 
\begin{equation}
\label{eq: Pos_Loc_Time_2}
 \mathbb{P} \big( L^{\Vert X\Vert} (\infty) >0 \big) \, >\, 0
\end{equation}
of the condition (\ref{eq: Pos_Loc_Time}),  then the measure $\,\bm \mu\,$ in (ii) is uniquely determined by $\, X(\cdot)\,$ and must satisfy (\ref{gamma0}).  
\end{thm}
 
\begin{remark}
\label{remark: 9.1}
The existence and determination of $\,\bm\mu\,$ in Theorem \ref{thm: Gen}(ii)(iii)  are reminiscent of what happens for the skew BM, where one can ``read off"  from the  \textsc{Harrison \& Shepp} (1981) equation -- e.g.\,\,(\ref{eq: HS}) below -- what the skewness parameter is. The measure $\,\bm\mu\,$ here, however, cannot be decided only through equation (\ref{eq: SDEMP}), as one can usually find many $\, \bm \mu$'s satisfying (9.1), given $\, \gamma_1, \,\gamma_2$.  Rather, $\,\bm\mu\,$ can be observed from the paths of a given solution, as (\ref{eq: UniqMU}) and Proposition \ref{Invariant and IID}  will show in the proof of Theorem \ref{thm: Gen}.
 \end{remark}
 
Under appropriate conditions on    $\, (\bm \sigma, \bm b) \,$   in (\ref{eq: SDEMP}), we will only need $\, X(\cdot) \,$ itself to be time-homogeneous and strongly Markovian with respect to $\,\mathbb F^{X}\,$ in Theorem \ref{thm: Gen}{\it (ii)}. The following lemma guarantees this. 

\begin{lm}
\label{LM ||X||}
With the setting and   assumptions of Theorem \ref{thm: Gen}, 
suppose  that the local submartingale problem of subsection \ref{sec_511} associated with the pair $\,(\bm \sigma, \bm b)\,$   is well-posed. 

Then $\Vert X(\cdot) \Vert$ is time-homogeneous strongly Markovian with respect to $\,\mathbb F^{X}\,$.
\end{lm}

\noindent
{\it Proof of Lemma \ref{LM ||X||}:}
We obtain the     equation   (\ref{eq: ||X||}) for the radial  part $\,\Vert X(\cdot) \Vert \,$ from Theorem \ref{thm: Gen}(i).  Applying \textsc{It\^o}'s formula to it  in the context of subsection \ref{sec_511}, we see  that for every function $\,\psi \in C^{2} ( [0, \infty); \mathbb R) \,$ with $\,\psi^{\prime}(0+) \ge 0\,$ (resp. $\,\psi^{\prime}(0+) = 0\,$) , the process $\, K^{\psi}  (\cdot\, ; \Vert X(\cdot) \Vert)\,$ is a continuous local submartingale (resp. martingale) with respect to the filtration $\, \mathbb F\,$. This process is also adapted to $\,\mathbb F^{X}\,$ and $\, \mathcal F^{X}(t) \subseteq \mathcal F(t) \,$ holds for all $\, t \ge 0\,$, so the statement in the last sentence   still holds with $\, \mathbb F\,$ replaced by $\,\mathbb F^{X}\,$. 

\smallskip
Following the idea of Lemma 5.4.18 and Lemma 5.4.19 in \textsc{Karatzas \& Shreve} (1991), we denote by $\, \mathbb Q_{\omega}(A)=\mathbb Q(\omega ; A): \, \Omega \times \mathcal F \mapsto [0,1]   \,$ the regular conditional probability for $\, \mathcal F \,$ given $\, \mathcal F^{X}(T) \,$, where $\, T \,$ is a bounded stopping time of $\,\mathbb F^X\,$.  For every $\, \omega \in \Omega \,$, we define the probability measure $\, \mathbb P_{\omega} \,$ on  $\, \big(C[0, \infty), \mathcal B\big(C[0, \infty)\big)\big)\,$  via $\, \,\mathbb P_{\omega}(F):= \mathbb Q_{\omega}\big(\Vert X(T+ \,\cdot\,) \Vert \in F\big) , \, \,\,\forall \,\, F \in \mathcal B(C[0, \infty))$.

 \smallskip

With this notation  and the conclusion in the first paragraph of this proof,  we can follow   the arguments in the aforementioned two lemmas  to show that for a.e.$\,\, \omega \in \Omega $, the probability measure $\, \mathbb P_{\omega} \,$   solves the local submartingale problem associated with the pair $\,(\bm \sigma, \bm b)\,$, starting at $\, \Vert X(T, \omega) \Vert$. Combining this with the well-posedness of the  local submartingale problem, we prove  Lemma \ref{LM ||X||} by applying the proof of Theorem 5.4.20 in \textsc{Karatzas \& Shreve} (1991).    \qed

\begin{remark}
Just as in the proof of Proposition \ref{SM,Unique},   the above argument   needs a ``countable representatives" result like Lemma \ref{Count}. Here it suffices to take   functions of the form $ f(x)=x \,$, $ \, g(x)=x^{2} \,$, and for every $\, n\, \in \mathbb{N} \,$ a function $\, f_{n}(\cdot) \,$ such that $\, f_n' (0+) =0\,$ and $\, f_n (x) =x\,$ for $\, x \ge (1/n)\,$. 
\end{remark}
 
In conjunction with Lemma \ref{LM ||X||}, Theorem \ref{thm: Gen} has the following corollary.  
 \newpage

\begin{cor}
\label{cor: Gen}
Suppose that the conditions (\ref{eq: zeroLeb}), (\ref{eq: Pos_Loc_Time_2}) are  satisfied by a weak solution $\, (X(\cdot), W(\cdot)), \,( \Omega, $ $\mathcal {F}, \mathbb P), \, \mathbb F$ $= \big\{ \mathcal F({t})\big\}_{0 \le t < \infty}\,$   of the system of  equations (\ref{eq: SDEMP}),  for some given real numbers $\, \gamma_1, \,\gamma_2\,.$ Suppose also  that   the local submartingale problem associated with the pair $\,(\bm \sigma, \bm b)\,$ is well-posed.  

\smallskip
If the state process $\, X(\cdot)\,$ of this weak solution is time-homogeneous and strongly Markovian with respect to $\,\mathbb F^{X}$,      then    it determines   a probability measure $\,\bm \mu\,$ on $\,(\mathfrak S,\, \mathcal B(\mathfrak S)) \,$ which satisfies (\ref{gamma0}), and is such that $\, X(\cdot)\,$ induces a solution to the local martingale problem associated with the triple $\, (\bm \sigma, \bm b, \bm \mu) \,$. 
\end{cor}

\subsection{The case of \textsc{Walsh} Brownian Motion}  
 \label{subsec: THSMP_WBM}

Let us specialize now the system of equations (\ref{eq: SDEMP}) to  the case $\,\bm b = \mathbf{ 0} ,\, \bm {\sigma} = \mathbf{ 1}\,$ as in Proposition \ref{prop: ID}, namely 
\begin{equation}
\label{eq: SDEBM}
X_{i}(\cdot) \, =\,  {\rm x}_{i} + \int^{\, \cdot}_{0} \mathfrak f_{i}\big(X(t)\big) \,{\mathrm d} W(t) +   \gamma_{i} \, L^{ \,||X||}(\cdot) \, ,  \qquad i=1,2  \,.
\end{equation}
\noindent
 We shall  show that when $\,\gamma_{1}^{2}+\gamma_{2}^{2}\leq 1\,$   this system,  coupled with the time-homogeneous strong \textsc{Markov} property,  characterizes  \textsc{Walsh} Brownian motions under the  non-stickiness  condition (\ref{eq: zeroLeb}). 
We note that in the statement and proof of the next proposition, neither $\, (\gamma_{1}, \gamma_{2})\,$ nor $\, \bm \mu\,$ are specified in advance; rather, we view them as related via (\ref{gamma0}).   

\begin{prop}{\bf  A New Characterization of \textsc{Walsh} Brownian Motions  
:} 
\label{SIEC: BM}
Assume that $\, Z (\cdot)\,$ is a continuous, adapted planar process on some filtered probability space $\,( \widetilde{\Omega}, \widetilde{\mathcal F}, \widetilde{\mathbb P}) \,$, $\, \widetilde{\mathbb F} = \big\{ \widetilde{\mathcal F}(t)\big\}_{0 \le t < \infty} \,$.
Then the following two assertions are equivalent: 

\smallskip
\noindent
(i) $\, \,Z (\cdot)\,$ is a \textsc{Walsh} Brownian motion, defined via the semigroup (\ref{eq: FSEMIG}), for some spinning measure $\, \bm \mu\,$. 

\smallskip
\noindent
(ii) \,For some   
pair of real numbers $\, (\gamma_{1}, \gamma_{2})\,$ with $\, \gamma_{1}^{2}+\gamma_{2}^{2}\leq 1\,$,   there exists a weak solution $\, (X(\cdot), W(\cdot)),\, ( \Omega,$ $\mathcal {F}, \mathbb P),$ $\mathbb F= \big\{ \mathcal F({t})\big\}_{0 \le t < \infty}\,$ to the system of equations (\ref{eq: SDEBM}),   such that  $\, X(\cdot):$  is time-homogeneous strongly Markovian with respect to $\,\mathbb F^{X};$  satisfies the   condition (\ref{eq: zeroLeb}); and   has the same distribution as $\, Z(\cdot)\,$.
 
 \smallskip
When these assertions hold,   the   measure $\, \bm \mu\,$ of the statement {\it (i),} and the co\"efficients $\, \gamma_{1}, \gamma_{2}\,$ of the statement {\it (ii),} satisfy the relationship (\ref{gamma0}).   
\end{prop}

\noindent
{\it Proof
: (i) $\Rightarrow$ (ii).} On the strength of Propositions \ref{MartProb} and \ref{prop: SDEMP}, the process $\, Z (\cdot)\,$ induces a weak solution of (\ref{eq: SDEBM}) subject to (\ref{eq: zeroLeb}), where $\, \gamma_{1}, \gamma_{2}\,$ are given by (\ref{gamma0}) and therefore satisfy $\, \gamma_{1}^{2}+\gamma_{2}^{2}\leq 1\,$. Since $\, Z (\cdot)\,$ is time-homogeneous strongly Markovian with respect its own  filtration, so is this solution.  
 
\smallskip
\noindent
{\it (ii) $\Rightarrow$ (i).}  Appealing to Proposition \ref{prop: MP1}, we see that the local submartingale problem associated with the pair $\, (\mathbf 1, \mathbf 0) \,$ is well-posed. By Theorem \ref{thm: Gen}(i) we obtain that   the radial part of $\, X(\cdot)\,$ satisfies
\begin{equation} 
\label{eq: ||X||BM}
\Vert X(\cdot) \Vert = \Vert X(0) \Vert + \int_{0}^{\cdot} {\bf 1}_{ \{ \Vert X(t) \Vert >0 \}} {\mathrm d} W(t) + L^{ \Vert X \Vert} (t) = \Vert X(0) \Vert + W(\cdot)+ L^{ \Vert X \Vert} (\cdot)\,,
\end{equation}
with the help of the non-stickiness condition (\ref{eq: zeroLeb}). Therefore, $\,\Vert X(\cdot) \Vert\,$ is the \textsc{Skorokhod} reflection of the Brownian motion $\,\Vert X(0) \Vert + W(\cdot)\,$,
and satisfies $\,\mathbb{P} \big( L^{\Vert X\Vert} (\infty) = \infty \big) =1\,$, so the condition (\ref{eq: Pos_Loc_Time_2}) follows.

Now from Corollary \ref{cor: Gen}, the weak solution posited in {\it (ii)}  induces a solution to the local martingale problem associated with the triple $\, (\mathbf{ 1}, \mathbf{ 0},  {\bm \mu}) \,$, 
for some probability measure $\, \bm \mu\,$ that satisfies (\ref{gamma0}).   Propositions \ref{MartProb} and \ref{prop: MP1} show that  $\, X(\cdot)\,$ is a \textsc{Walsh} Brownian motion with spinning measure $\,\bm \mu\,$, and so is $\, Z(\cdot)\,$. \qed

\smallskip
\begin{remark} {\it Similarities and Differences:} 
Proposition  \ref{SIEC: BM} and Theorem \ref{IFF} show  that the system of equations (\ref{eq: SDEBM}), with the condition $\,\gamma_{1}^{2}+\gamma_{2}^{2}\leq 1\,$ on the co\"efficients, is a two-dimensional analogue of the  \textsc{Harrison \& Shepp} (1981) equation for the skew Brownian motion. But with the following caveat:   
 
The equations (\ref{eq: SDEBM}), (\ref{eq: zeroLeb}) characterize \textsc{Walsh} Brownian motions only when we restrict attention to time-homogeneous strongly Markovian processes. {\it If  this restriction is not imposed, there will be solutions to the system (\ref{eq: SDEBM})  that are not \textsc{Walsh} Brownian motions.}  Such solutions are  discussed in the next remark.

Furthermore, (\ref{eq: SDEBM}) does not describe a unique \textsc{Walsh} Brownian motion, but may be satisfied by many such  motions with different spinning measures (cf. Remark \ref{remark: 9.1}). 
By contrast,   we can read off  the flipping probability from the co\"efficient in the equation for the one-dimensional skew Brownian motion. The construction in Remark \ref{Mix_Ex} right below is actually based on this observation. 
\end{remark}

\begin{remark}
\label{Mix_Ex}
{\it A solution to the system of equations (\ref{eq: SDEBM}) that features two different spinning measures:}  Consider the system of equations (\ref{eq: SDEBM}) with $\, \gamma_{1}=\gamma_{2}=0\,$ and $\,\rm x=(0,0)$, and note that  both   measures $$\bm \mu_{1} = \frac{1}{\,2\,}\, \delta_{(1,0)} + \frac{1}{\,2\,}\, \delta_{(-1,0)} \qquad \text{and} \qquad \,\bm \mu_{2} = \frac{1}{\,2\,}\, \delta_{(0,1)} + \frac{1}{\,2\,}\, \delta_{(0,-1)} $$    satisfy (\ref{gamma0}).   Let $\, X(\cdot)\,$ be a \textsc{Walsh} Brownian motion that solves the system (\ref{eq: SDEBM}) with $\, X(0)=(0,0)\,$, $\,\gamma_{1} = \gamma_{2} = 0\,$, spinning measure $\,\bm \mu_{1}\,$ and driving Brownian motion $\, B(\cdot)\,$. Let $\, Y(\cdot) \,$ be another \textsc{Walsh} Brownian motion that solves  (\ref{eq: SDEBM}) with $\,  Y(0)=(1,0)\,$, $\,\gamma_{1} = \gamma_{2} = 0\,$, spinning measure $\,\bm \mu_{2}\,$ and driver  $\,\widetilde B(\cdot) := B(\tau_{(1,0)} + \cdot)\,, $ another Brownian motion. Now define $\, \tau_{(1,0)} := \inf \{ t \geq 0 : X(t) = (1,0) \}\,$ and
 $$\, 
 Z(t)\, := \,X(t),\,\, \,\,0\,\leq\,t < \tau_{(1,0)}    \, ,\qquad \text{and} \qquad   Z(\tau_{(1,0)} + t) := Y (t), \,\,\,\,\forall \, \,\,t \geq 0 \,.
 $$   
The so-defined process $\, Z(\cdot)\,$ solves (\ref{eq: SDEBM}) with $Z(0)=(0,0)$ , $\gamma_{1} = \gamma_{2} = 0\,$ and driving Brownian motion $\, B(\cdot)\,$, but is {\it not} a \textsc{Walsh} Brownian motion: it switches from $\,\bm \mu_{1}\,$ to $\,\bm \mu_{2}\,$ after time $\,\tau_{(1,0)}\,$. It is also not time-homogeneous strongly Markovian,  by virtue of either Proposition \ref{SIEC: BM} or   elementary observations. 
\end{remark}

 \section{Examples}  
 \label{Ex} 

\begin{example} {\bf \textsc{Walsh}'s Brownian Motion and   Spider Martingales:} 
 \label{rem: spider} 
When the spinning  measure $\,{\bm \mu} \,$ in Theorem \ref{prop: skewTanaka}  is a discrete probability charging a  finite number of rays that pass through the origin,  and the  driving semimartingale $\,U(\cdot)\,$ is   Brownian motion, the process $\,X(\cdot)\,$ becomes the original \textsc{Walsh}  Brownian motion $\,{\bm W}(\cdot)\,$ with roundhouse singularities in multipole fields as in Proposition \ref{prop: ID}.    Given a finite number  $\,m  \ge 2  \,$ of distinct angles $\, \{ \theta_{\ell} \in [0, 2\pi), \,\ell = 1, \ldots, m\}\,$, let us consider $\,m\,$  rays emanating from the origin,    
$$
\,\mathcal{I}_{\ell} \, :=\, \big\{ x \in \mathbb R^{2} \setminus \{{\bm 0}\} : \text{arg} (x) \, =\, \theta_{\ell} \big\}\,,\qquad \ell \, =\, 1, \ldots ,m\,,
$$ 
and assign a discrete probability measure $\,{\bm \mu} \,$ with weights $\,p_{\ell} \in (0,1)\,$, $\, \sum_{\ell=1}^{m} p_{\ell} = 1\,$, such that   
\begin{equation}
\label{eq: discrete}
\, {\bm \mu} \big( \{ (\cos ( \theta_{\ell}), \sin (\theta_{\ell}))\}\big) \, =\, \mathbb P \big( \text{arg}({\bm \xi}_{1}\big) \, =\, \theta_{\ell}) \, =\, p_{\ell} \,,\qquad \ell \, =\, 1, \ldots ,m\,.
\end{equation}
 Using \textsc{Markov} semigroups  and   excursions, \textsc{Barlow, Pitman \& Yor} (1989)  study  \textsc{Walsh}'s Brownian motion $\, {\bm W}(\cdot) \,$ on the collection of rays $\, \bigcup_{\ell=1}^{m} \mathcal{I}_{\ell} \,\cup \{ {\bm 0}\} \,$. Their approach has been  generalized to ``multiple spider martingales" by \textsc{Yor} (1997), and has been studied   by \textsc{Tsirel'son} (1997), \textsc{Barlow, \'Emery, Knight, Song \& Yor} (1998), \textsc{Watanabe} (1999) and \textsc{Mansuy \& Yor} (2006), pp.\,103-116. 
\end{example}

\begin{example} {\bf The Case of Two Rays:} Let us consider the setup of the previous example with $\, m=2\,$ and $\, \theta_1=0\,$, $\, \theta_2 = \theta \in (0,   \pi]\,$, as well as $\,\mathbb P \big( \text{arg}({\bm \xi}_{1}\big)= \theta ) = p \in (0,1)\,$. The equations of (\ref{eq: skewTanaka_2}) 
become   
$$ 
X_{1}(\cdot) \, =\,  {\rm x}_{1} + \int^{\,\cdot}_{0} \text{cos} \big( \text{arg}\big(X(t)\big) \big) \,{\mathrm d} U(t) + 
\big( 1 - p + p \cos (\theta) \big)
\, L^{ ||X|| }(\cdot) \, ,  
$$
$$ 
X_{2}(\cdot) \, =\,  {\rm x}_{2} + \int^{\,\cdot}_{0} \text{sin} \big( \text{arg}\big(X(t)\big) \big) \,{\mathrm d} U(t) + 
p \,   \text{sin} (\theta)  \, L^{ ||X|| }(\cdot) 
$$
with $\, L^{ X_{1} }(\cdot) = \big( 1 - p + p \,   \text{cos}^+ (\theta) \big)\, L^{ ||X|| }(\cdot)\,$ and $\, L^{ X_{2} }(\cdot) =  p \,   \text{sin}^+ (\theta)  \, L^{ ||X|| }(\cdot)\,$.

   \smallskip
 \noindent
{\bf Case I:} With $\, \theta = \pi\,$, and with $\, x_2=0\,$ for simplicity, the second of these equations has the trivial solution $ X_2(\cdot) \equiv 0\,$, whereas the first can be cast in the form of   the celebrated \textsc{Harrison \& Shepp} (1981) equation 
\begin{equation}
\label{eq: HS}
X_{1} (\cdot) \, =\, x_{1} + V_{1}(\cdot) + \frac{\, 1 - 2\, p\, }{1 - p\, } L^{X_{1}}(\cdot) \, ,
\qquad \text{ driven by } \quad 
V_{1}(\cdot) \, :=  \int^{\cdot}_{0} \text{sgn} (X_{1}(t))  \, {\mathrm d} U(t) \,  . 
\end{equation}
 \newpage
 \noindent
   As these authors showed, when  $\, U(\cdot)\,$ is Brownian  motion  the above   equation has a pathwise unique, strong solution with respect to the Brownian motion $\,V_{1}(\cdot) $,  and in this case $\,X_1(\cdot)\,$ is   skew Brownian motion. 
   
   When written in terms of the original driver $\,U(\cdot)  $, the above equation for $\, X_1 (\cdot)\,$  is a skew version of the \textsc{Tanaka} equation.   In particular, Proposition 2.1 of \textsc{Ichiba \& Karatzas} (2014) establishes the filtration comparisons $\, \mathbb F^{U}(\cdot) \subsetneq   \mathbb F^{V_{1}}(\cdot)  = \mathbb F^{X_{1}}(\cdot) = \mathbb F^{(X_{1}, X_{2})}(\cdot)\,$ when $\,U(\cdot)\,$ is Brownian motion.  

\medskip
\noindent
{\bf Case II:} When $\, \theta \in (0, \pi) $, we assume  for simplicity $\, \text{arg}({\rm x}) \in \{ 0 , \, \theta\} \,$ and consider the  process
$$
\Upsilon(\cdot) \, :=\, \frac{ \, - X_{2}(\cdot) \, }{\sin (\theta)} \cdot {\bf 1}_{\{X_{2}(\cdot) > 0\}} + \frac{\, X_{1}(\cdot)\, }{\cos(0)} \cdot {\bf 1}_{\{X_{2}(\cdot)\, =\, 0\}} \,:
$$
we flatten the state space by rotating the ray. This process also satisfies a \textsc{Harrison-Shepp}-type equation  
\[
\Upsilon(\cdot) \, =\, \Upsilon(0) + V_{\bullet}(\cdot) + \frac{\, 1 - 2\, p\, }{1 - p\, } L^{\Upsilon}(\cdot) \, \qquad \text{ driven by } \quad  V_{\bullet}(\cdot) \, :=   \int^{\cdot}_{0} \text{sgn} (\Upsilon(t)) \, {\mathrm d} U(\cdot)\, ;
\]
and conversely,  the co\"ordinate processes are given in terms of $\,\Upsilon (\cdot) \,$ as 
$$
X_{1}(\cdot) \, =\, \Upsilon(\cdot) \cdot {\bf 1}_{\{\Upsilon(\cdot) > 0\}} - \Upsilon (\cdot) \cos (\theta) \cdot {\bf 1}_{\{\Upsilon (\cdot) \le 0 \}} \,, \qquad \,X_{2}(\cdot) \, =\,  - \Upsilon (\cdot) \sin (\theta) \cdot {\bf 1}_{\{\Upsilon (\cdot) \le 0 \}} \,.
$$
 If $\, U(\cdot)\,$ is standard Brownian motion, then so is $\,V_{\bullet}(\cdot)\,$; in this case  $\,\Upsilon (\cdot) \,$ becomes a skew Brownian motion, and  we obtain  as before the filtration comparisons
$\, 
\mathbb F^{U}(\cdot) \,  \subsetneq \, \mathbb F^{V_{\bullet}}(\cdot) \, =\, F^{\Upsilon}(\cdot) \, =\, \mathbb F^{(X_{1}, X_{2})}(\cdot)\,$. 
 
\smallskip
\noindent
$\bullet~$ 
We have shown that the filtration $\,\mathbb F^{(X_{1}, X_{2})}(\cdot)\,$ of a \textsc{Walsh} Brownian motion on two rays coincides with the filtration  generated by some standard Brownian  motion,  and   is   {\it strictly finer} than the filtration $ \mathbb F^{U}(\cdot) \,$ generated by its driving   Brownian motion. 

\smallskip
\noindent
$\bullet~$  
Suppose   the driver  $ U(\cdot) $ is a continuous local martingale with $U(0)=0$ and $\langle U \rangle (\infty) = \infty$, and consider its \textsc{Dambis-Dubins-Schwarz} representation 
 $\, U(\cdot) = {\bm \beta} (\langle U \rangle (\cdot) )\,$ with $\, {\bm \beta} (\cdot)\,$ a standard Brownian motion. 

From the above considerations and in conjunction with Proposition 2.2 in \textsc{Ichiba \& Karatzas} (2014) we see that, in the case of a spinning measure $\, {\bm \mu}\,$ that charges exactly two points on the unit circumference, uniqueness in distribution holds   for the system (\ref{eq: skewTanaka_1}) subject to (\ref{length}) and  (\ref{eq: localDist}),  provided that either

\smallskip
\noindent 
(i) $\, \,U(\cdot)\,$ is {\it pure} $($i.e., $\langle U \rangle ( t)$ is $\,{\cal F}^{{\bm \beta}}(\infty)-$measurable, for every $t \in [0, \infty))$; or that 

\smallskip
\noindent 
(ii) \,the quadratic variation process $\langle U \rangle (\cdot)  $ is adapted to a Brownian motion $\, \Gamma (\cdot) = ( \Gamma_1 (\cdot), \cdots, \Gamma_n (\cdot))'\,$ with values in some Euclidean space and independent of the   Brownian motion $\, {\bm \beta} (\cdot)\,$. 
\end{example}

\begin{example} {\bf \textsc{Tsirel'son}'s triple point:} 
When $\,\alpha_{i}^{(+)} = \alpha_{i}^{(-)} \, $ for $\,i \, =\, 1,2\,$, the equations (\ref{eq: skewTanaka}) and (\ref{eq: skewTanaka_2}) become, respectively,  
\[
X_{i}(T) \, =\,  x_{i} + \int^{T}_{0} \mathfrak f_{i}\big(X(t)\big)\, {\mathrm d} S(t)   \qquad \text{and} \qquad  X_{i}(T) \, =\,  x_{i} + \int^{T}_{0} \mathfrak f_{i}\big(X(t)\big)\, {\mathrm d} U(t) \, ; \quad i \, =\, 1, 2 \, .   
\]
This is the case when the common probability distribution $\,{\bm \mu} \,$ of the I.I.D. random variables $\,\{ {\bm \xi}_{1}, {\bm \xi}_{2}, \ldots \} \,$ in (\ref{eq: exp}) has zero expectation, namely $\, \mathbb E[ {\bm \xi}_{1} ] \, =\,  {\bm 0}\,$.     
For instance, when $\, {\bm \mu} \,$ assigns equal weights of $\,1   /   3\,$ to three points at angles $\,\theta_{0} + (2 \pi \ell / 3)\,$, $\,\ell \, =\, 0, 1, 2\,$ on the unit circumference $\,\mathfrak S\,$ that trisect it,  namely,    
\[
\mathbb P \Big({\bm \xi}_{1} = \big(\cos ( \theta_{0} + (2 \pi \ell / 3))\,, \sin (\theta_{0} + (2 \pi \ell / 3)) \big)^{\prime} \Big) \, =\,  1  / 3 \, ; \quad \ell \, =\, 0,1,2 
\]
 for some $\,\theta_{0} \in [0, 2 \pi)\,$. If, in  addition,  $ \, U(\cdot)   =  W(\cdot) \,$ is   Brownian motion,  and thus the \textsc{Skorokhod} reflection    $  \,S(\cdot)   =  W(\cdot) + \max_{\,0 \le s \le \,\cdot \,} (- W(s))^{+} 
 $ in (\ref{eq: S})  is  a reflecting Brownian motion,  we deduce from subsection \ref{La} that the corresponding planar process $\, X(\cdot) =(X_{1}(\cdot), X_{2}(\cdot))' \,$ is a martingale, to wit
$\,
X_{i}(T)=  {\mathrm x}_{i} + \int^{T}_{0} \mathfrak f_{i}\big(X(t)\big)\, {\mathrm d} W(t) \, , \,\,\, i \, =\, 1, 2 \, $.

\smallskip
It was conjectured by \textsc{Barlow, Pitman \& Yor} (1989),  and shown in the landmark paper by \textsc{Tsirel'son} (1997)  (cf. \textsc{Yor} (1997), \textsc{Mansuy \& Yor} (2006)),  
that the natural filtration of this martingale $\, X(\cdot)\,$ is not generated by {\it any}  Brownian motion of {\it any} dimension.  
\end{example}

\begin{example} \label{ex: driftWBM} {\bf \textsc{Walsh}'s Brownian motion with polar drifts:}    Let us look at the case $\, {\bm \sigma }(\cdot) \equiv 1\,$ and $\, {\bm c}(\cdot) \equiv  - \lambda \,$ for some $\, \lambda > 0\,$ in Corollary \ref{Cor: MP1}. The driving one-dimensional semimartingale $\,U(\cdot) \,$ for $\,X(\cdot)\,$ is Brownian motion with negative drift $\, - \lambda\,$ and with instantaneous reflection at the origin. It follows from  Theorem \ref{prop: skewTanaka} the process $\,X(\cdot) \, =\, (X_{1}(\cdot), X_{2}(\cdot))^{\prime} \,$ satisfies 
\[
X_{i}(T) \, =\,  {\mathrm x}_{i} + \int^{T}_{0} {\mathfrak f}_{i}(X(t)) \big( - \lambda \, {\mathrm d} t + {\mathrm d} W(t) \big) +  \gamma_i \,   
L^{ ||X || }(T) \, , \qquad 0 \le T < \infty \,  
\]
for $\,i \, =\, 1\, , \, 2\,$, where $\,W(\cdot)\,$ is one-dimensional standard Brownian motion. 

Moreover, following Proposition \ref{prop: AD}, we may replace the constant drifts by {\it drifts exhibiting angular dependence}. Suppose that $\, {\bm a}(r, \theta) \, =\,  1\,$ and $\, {\bm b} (r, \theta) \, =\, {\bm \lambda} (\theta) \,$ for some measurable function $\,{\bm \lambda}: [0, 2\pi) \to (0, \infty) \,$. The resulting process $\,Y(\cdot) \,$ in Proposition \ref{prop: AD} has the dynamics  
\[
Y(T) \, =\,  {\mathrm y} + \int^{T}_{0} {\mathfrak f}(Y(t)) \Big( - {\bm \lambda} ( \text{arg} \big(Y(t))\big) {\mathrm d} t + {\mathrm d} W(t)\Big) + {\bm \gamma} \, L^{ \lVert Y \rVert }(T) \, , \qquad 0 \le T < \infty \, . 
\]
Since the driving semimartingale is positive recurrent in $\,\mathbb R_{+}\,$, the degenerate planar process $\,X(\cdot) \,$ is positive recurrent.  Its stationary distribution is  expressed in polar co\"ordinates as 
\[
\bigg( \int^{2\pi}_{0} \frac{\bm \nu ({\mathrm d} u)}{\, 2 [ {\bm \lambda} (u)]^{2}\, }   \bigg)^{-1} \, \frac{ \, e^{-2 {\bm \lambda}(\theta)r}\, }{{\bm \lambda} (\theta)\,  } \,\,{\mathrm d} r  \, {\bm \nu } \, ({\mathrm d} \theta ) \,; \qquad r > 0 \, , \, \, \theta \in [0, 2\pi)   \,    
\]
by the distribution of occupation times and the excursion theory of \textsc{Salminen, Vallois \& Yor} (2007). \,  
If $\, {\bm \lambda} (\cdot) \equiv \lambda \,$(constant), then the stationary distribution reduces to $\, \big(2 \lambda e^{-2\, \lambda r} {\mathrm d} r \big)\, {\bm \nu} ({\mathrm d} \theta) \,$, $\, r > 0\,$, $\, \theta \in [0, 2\pi)\,$. 
\end{example}

\begin{example} {\bf \textsc{Walsh} semimartingale driven by \textsc{Bessel} processes:}   Suppose that $\,R^{2}(\cdot)\,$ is a squared  \textsc{Bessel} process with dynamics 
 $\, 
{\mathrm d} R^{2}(t)= {\delta} \, {\mathrm d} t + 2 \sqrt{ R^{2}(t)\,}\, {\mathrm d} W (t) \,  ,  
$  
where $\, {\delta} \in (1, 2) \,$ and $\, W(\cdot) \,$ is one-dimensional standard Brownian. 

We take the square root $\, \lvert  R(\cdot) \rvert \, $ of this process as the driving semimartingale, i.e., $\,U(\cdot) \, =\,  \lvert R(\cdot)\rvert \, =\,  S(\cdot) \,$ in Theorem \ref{prop: skewTanaka}. This process $\,S(\cdot) \,$ does not accumulate local time at the origin, i.e., $\,L^{S}(\cdot) \equiv 0 \,$ holds for $\,\delta \in (1, 2)\,$,  hence  the resulting planar process $\,X(\cdot)\,$ of Theorem \ref{prop: skewTanaka} has the dynamics  
\[
X_{i}(T) \, =\, {\mathrm x}_{i} + \int^{T}_{0} {\mathfrak f}_{i}(X(t)) \Big( \frac{ \, \delta - 1\, }{\, 2 \, \lVert X(t) \rVert\, } \cdot {\bf 1}_{\{ \lVert X(t) \neq 0 \rVert\}} \, {\mathrm d} t + {\mathrm d} W(t) \Big)  \, , \quad 0 \le T < \infty \,  
\]
for $\, i = 1  ,   2 $. Note that when $ \, \delta \, =\, 1 $, the process $\,X(\cdot)\,$  becomes \textsc{Walsh} Brownian motion; when $ \delta \in (0, 1)$, the semimartingale property is violated; when $ \delta \ge 2 $, the process $ R(\cdot)  $ never reaches the origin.

 \smallskip
Furthermore, and by analogy with Example \ref{ex: driftWBM}, given a measurable function $\, {\bm \delta} : [0, 2\pi) \to (1, 2)\,$ we may use the time-change technique with the dispersion $\,{\bm a}(r, \theta) \, =\,  4 \, r  \,$ and the drift $\,{\bm b}(r, \theta) \, =\,  {\bm \delta}( \theta) \,$ and consider the \textsc{Walsh} semimartingale $\,Y(\cdot) \,$ driven by angular dependent, squared-\textsc{Bessel} process   
\[
Y(T) \, =\,  {\mathrm y} + \int^{T}_{0} {\mathfrak f}(Y(t)) \Big( {\bm \delta}(\text{arg} (Y(t)) \, {\mathrm d} t + 2 \sqrt{ \,  \lVert Y(t) \rVert\, } \, {\mathrm d} W(t) \Big ) \, , \qquad 0 \le T < \infty . 
\]
Here,  the process  $\, \lVert Y (\cdot) \rVert\,$ does not accumulate local time at the origin. The corresponding scale function,   inverse function  and   stochastic clock  are given   by $\,{\bm p}_{\theta} (r) \, =\, r^{(2- {\bm \delta}(\theta))\, /\, 2} \,$, $\, {\bm q}_{\theta}(r) \, =\, r^{2 \, / \, (2 - {\bm \delta}(\theta))}\,$, and 
\[
{\mathcal T}(\cdot) \, =\, \int^{\cdot}_{0} \Big((2 - {\bm \delta}(\theta) )^{2} \, r^{- ({\bm \delta}(\theta) - 1)} \Big) \Big|_{r = \lVert Y (t) \rVert, \, \theta = \text{arg}(Y(t)) } \,{\mathrm d} t \, ,  
\]
respectively. It can be shown that the stochastic clock does not explode (cf. Lemma 3.1 of \textsc{Biane \& Yor} (1987), Proposition XI.1.11 of \textsc{Revuz \& Yor} (1999), pages 285-289 of \textsc{Rogers \& Williams} (2000) and Appendix A.1 of \textsc{Ichiba et al.} (2011)). From this process $\, Y(\cdot) \,$ we may define now the \textsc{Walsh} semimartingale 
$
\, \Xi(\cdot) = \big(\Xi_{1}(\cdot), \Xi_{2}(\cdot)\big)^{\prime}\,$ with $  \,  \Xi_{i}(\cdot)   :=  {\mathfrak f}_{i}(Y(\cdot)) \, \lVert Y(\cdot) \rVert^{1/2} \, , \, \,   i \, =\, 1, 2  \,
$
driven by a \textsc{Bessel} process with angular dependence, which  satisfies the vector integral equation derived from (\ref{eq: GenFS2}), namely, 
  \newpage
 \[
\Xi(T) \, =\, \Xi(0) + \int^{T}_{0} {\mathfrak f}(\Xi(t)) \Big(  \frac{ \, {\bm \delta}(\text{arg}(\Xi(t))) - 1\, }{ \, 2 \, \lVert \Xi (t) \rVert } \,{\bf 1}_{ \{ \lVert \Xi (t) \neq  0 \rVert\} }{\mathrm d} t + {\mathrm d} W(t) \Big) \, , \qquad 0 \le T < \infty   \,.
\] 
 \end{example}

\section{Appendix: The Proof  of Lemma \ref{Count}}  
\label{sec: App_Lem}

 \noindent
 We denote by $\,\widetilde{\mathfrak D^{\bm \mu}}$ (resp. $\widetilde{\mathfrak D^{\bm \mu}_{+}}$)   the collection of functions $\, g \,$ in $\,\mathfrak D^{\bm \mu}\,$ (resp. $\mathfrak D^{\bm \mu}_{+}$) such that $\, M^{g} (\cdot\,;\omega_2)\,$ is a continuous local martingale (resp. submartingale)  of the filtration $\, \mathbb F_{2} \,$, under $\, \mathbb Q \,$. Then we have $\, \widetilde{\mathfrak D^{\bm \mu}} \supseteq  \mathfrak D^{\bm \mu} \cap \mathfrak E  \,$ and $\,\widetilde{\mathfrak D^{\bm \mu}_{+}} \supseteq    
\mathfrak E\,$ by assumption. The goal here is to show $\, \widetilde{\mathfrak D^{\bm \mu}} = \mathfrak D^{\bm \mu} \,$ and $\, \widetilde{\mathfrak D^{\bm \mu}_{+}} = \mathfrak D^{\bm \mu}_{+} \,$.

Recalling that $\, \mathfrak E\,$ contains the  functions in Definition \ref{Class_E}(ii), we can follow the proof of Part (b) of Proposition \ref{prop: SDEMP} and show that there exists a one-dimensional standard Brownian motion $\, {W}(\cdot) \,$ on an extension of the filtered probability space $\,(\Omega_{2}, \mathcal F_{2}, \mathbb Q)$, $\mathbb F_{2}\,$ such that (\ref{eq: SDEMP}), (\ref{eq: ||X||}) hold  with $\, X(\cdot)\,$ given by (\ref{eq: XMPSDE}), or simply $\, X(\cdot):=\omega_{2}(\cdot)\,$.  It is clear, therefore, that $\int^{t}_{0} {\bf 1}_{\{ \lVert \omega_{2}(u) \rVert > 0 \}} \,\big( \vert {\bm b} ( \lVert \omega_{2}(u) \rVert) \vert + {\bm a} ( \lVert \omega_{2}(u) \rVert) \big) {\mathrm d} u \, < \infty$ holds for all $\, 0 \leq t < \infty$, $\,\mathbb Q-$a.s. We make now the following two observations. 

\smallskip
\noindent
{\bf First Observation:}   $\,\widetilde{\mathfrak D^{\bm \mu}}$  {\it  is a linear space.} This is obvious from the linearity of stochastic integrals, derivatives, and   local martingales. 

\smallskip
\noindent
{\bf Second Observation:} {\it  Suppose $\,\{ g_{n} \}_{n \in \mathbb N}  \subseteq \widetilde{\mathfrak D^{\bm \mu}}\, $ and $\,g \in \mathfrak D^{\bm \mu}$ satisfy the following: as $\, n \uparrow \infty$, $ \, g_{n}(x) \rightarrow g(x), \, \forall \, x \in \mathbb R^{2} \,$ and $\,G_{n}^{\prime}(x) \rightarrow G^{\prime}(x), \, G_{n}^{\prime\prime}(x) \rightarrow G^{\prime\prime}(x), \,\forall \, x \in \mathbb R^{2} \setminus \{ {\bm 0} \}, \,$ and all these functions ($g_{n},\, g,\, G_{n}^{\prime},\, G^{\prime},\, G_{n}^{\prime\prime},\, G^{\prime\prime}$) are uniformly bounded on every compact subset of $\,\mathbb R^{2}$. Then we have $\,g \in \widetilde{\mathfrak D^{\bm \mu}}$. }

To see this, we define stopping times  
$$\, 
T_{k} \,=\, \inf \left\{ t :   \int^{t}_{0} {\bf 1}_{\{ \lVert \omega_{2}(u) \rVert > 0 \}} \,\Big( \big \vert {\bm b} \big( \lVert \omega_{2}(u) \rVert \big) \big \vert + {\bm a} \big( \lVert \omega_{2}(u) \rVert \big) \Big) {\mathrm d} u \, \ge k \,\,\,\, \text{or} \,\,\, \,\lVert \omega_{2}(t) \rVert \ge k \right\}, \,\,\, \,\,\,k \in \mathbb N  \,,
$$  
and note that  $\, M^{g_{n}} (\cdot \wedge T_{k}\,;\omega_2),\, n \in \mathbb N \,$ are uniformly bounded local martingales, hence uniformly bounded martingales,  for all $\,k \in \mathbb N\,$; and that     $\, \lim_{n \to \infty } \, M^{g_{n}} (t \wedge T_{k}\,;\omega_2) =    M^{g} (t \wedge T_{k}\,;\omega_2)\,$ for any $\, t \in [0, \infty)$. Thus $\, M^{g} (\cdot \wedge T_{k}\,;\omega_2)\,$ is also a continuous   martingale, and the conclusion $\,g \in \widetilde{\mathfrak D^{\bm \mu}}\,$ follows.   

\smallskip
\noindent
$\bullet\,$ 
Returning to our argument, we know that for the functions of Definition \ref{Class_E}, the process $\, M^{g_{A}} (\cdot \, ;\omega_2)\,$ is a local martingale for any interval  $\, A \subseteq [0, 2\pi) \,$ of the form $\, [a, b) \,,$ where $\, a, b \,$ are rationals. Thus the same is true when $\, A \,$ is the disjoint union of such intervals, by linearity. These sets form an algebra. By the second observation and monotone class arguments, the same is also true for every \textsc{Borel} subset $\, A \,$ of $\, [0, 2 \pi) \,$. Now for any two disjoint \textsc{Borel} subsets $\, A, B    \,$ of $\,     [0, 2\pi) \,,$ we define 
$$\,
g_{A , B}(x) := \Vert x \Vert \, \big(\bm \nu (A) \, {\bf 1}_{\{\text{arg}(x) \in B \} } - \bm \nu (B) \, {\bf 1}_{ \{\text{arg}(x) \in A \} } \big) \quad \text{and note} \quad g_{A , B}(x) = \bm \nu (A) \, g_{B}(x) - \bm \nu (B) \, g_{A}(x),
$$ 
thus  $\, g_{A, B} \in \widetilde{\mathfrak D^{\bm \mu}}$ by linearity.  Starting from this and using linearity and induction, we  show that if $\,h: [0, 2 \pi) \to \R\,$ is simple and satisfies $\int_{0}^{2 \pi} \, h(\theta) \bm \nu ({\rm d} \theta) = 0 $, then the mapping $\, x \mapsto \Vert x \Vert \cdot h(\text{arg}(x))\,$ is in $\,\widetilde{\mathfrak D^{\bm \mu}}$. Using approximation and the second observation, we see that this statement is still true when ``simple" is replaced by ``bounded and measurable".

Let us recall now that, in the second paragraph of this section, we  obtained the existence of a one-dimensional standard Brownian motion $\, {W}(\cdot) \,$ on an extension of the filtered probability space $\,(\Omega_{2}, \mathcal F_{2}, \mathbb Q)$, $\mathbb F_{2}\,$, along with (\ref{eq: SDEMP}) and (\ref{eq: ||X||}), where $\, X(\cdot):=\omega_{2}(\cdot)\,$. By defining $\, S(\cdot):= \Vert X(\cdot) \Vert \,,$   we can follow the proof of Theorem \ref{Gen_FS} to establish for any given function  $\, g \in \mathfrak D^{\bm \mu} \,$ with $\, g_{\theta}^{\prime}(0+) \equiv 0\,$  the following \textsc{Freidlin-Sheu-}type semimartingale decomposition:  
\[
g( \omega_{2}(\cdot) )\,= \,g( {\rm x})+ \int_{0}^{\cdot} \, {\bf 1}_{\{ \lVert \omega_{2}(t) \rVert > 0 \}} \Big( {\bm b} ( \lVert \omega_{2}(t) \rVert ) \, G^{\prime}(\omega_{2}(t)) + \frac{1}{\, 2\, } \, {\bm a} ( \lVert \omega_{2}(t) \rVert ) \,G^{\prime\prime}(\omega_{2}(t)) \Big) {\rm d} t  
\]
\[
+ \int_{0}^{\cdot} \, {\bf 1}_{\{ \lVert \omega_{2}(t) \rVert > 0 \}} {\bm \sigma} ( \lVert \omega_{2}(t) \rVert) \, G^{\prime}(\omega_{2}(t))\, {\mathrm d} W(t)\,.
\]
The condition  (\ref{eq: localDist}) is {\it not} needed here; and neither are terms involving local time. 

This is because the use of (\ref{eq: localDist}) in the proof of Theorem \ref{Gen_FS} comes only when proving the convergence to local time as in (\ref{007}). 
But this property  holds here trivially, courtesy of $\,  g_{\theta}^{\prime}(0+) \equiv 0\,$. It follows from the above decomposition of \textsc{Freidlin-Sheu-}type  that, if $\, g \in \mathfrak D^{\bm \mu} \,$ satisfies $\, g_{\theta}^{\prime}(0+) \equiv 0\,$, then $\,g \in \widetilde{\mathfrak D^{\bm \mu}}$.

 \smallskip
Finally, we observe that every   $\, g \in \mathfrak D^{\bm \mu} \,$ can be decomposed as $\, g= g^{(1)}+g^{(2)} \,$, where the function $\, x \mapsto  g^{(1)}(x):=\Vert x \Vert \cdot g_{\theta}^{\prime}(0+)$ is in $\,\widetilde{\mathfrak D^{\bm \mu}} \,$ by the first paragraph of this bullet, and the function $\, g^{(2)} := g - g^{(1)} \in \mathfrak D^{\bm \mu} \,$ satisfies $\, \big(g^{(2)}_{\theta}\big)^{\prime}(0+) \equiv 0\,$. With  the considerations above, we see $\,g \in \widetilde{\mathfrak D^{\bm \mu}}$, thus $\, \widetilde{\mathfrak D^{\bm \mu}}=\mathfrak D^{\bm \mu} \,$. We decompose then every function $\, g \in \mathfrak D^{\bm \mu}_{+} \,$ as $\, g= g_{(1)}+g_{(2)} \,$, where $\, g_{(1)}(x):= c \, \Vert x \Vert  \,$  
 with a constant $\,c := \int_{0}^{2\pi} g_{\theta}^{\prime}(0+) {\bm \nu}({\mathrm d} \theta)\,\geq 0\,$ and $\, g_{(2)} :=  g - g_{(1)} \in \mathfrak D^{\bm \mu} \,$  (cf. Remark \ref{remark: 6.1}). Here $\, M^{g_{(2)}} (\cdot \, ;\omega_2)\,$ is a local martingale, and $\, M^{g_{(1)}} (\cdot \, ;\omega_2) = c \, M^{g_{3}}(\cdot \, ; \omega_{2})  \,$ is a local submartingale, since the mapping  $\, x \mapsto g_{3}(x) = \Vert x \Vert \,$ belongs to $\, \mathfrak E \subseteq  \widetilde{\mathfrak D^{\bm \mu}_{+}} \,$ (cf. Definition \ref{Class_E} (ii)). Thus $\, M^{g} (\cdot \, ;\omega_2) \,$ is also a local submartingale and $\,g \in \widetilde{\mathfrak D^{\bm \mu}_{+}}\,$. We conclude then $\, \widetilde{\mathfrak D^{\bm \mu}_{+}} = \mathfrak D^{\bm \mu}_{+}\,$, and the proof of Lemma \ref{Count} is complete.  \qed

\section{Appendix: The Proof   of Theorem \ref{thm: Gen}}  
 \label{sec: App}

 \noindent
We first identify the measure $\,\bm \mu\,$ from $\, X(\cdot)\,$, using the time-homogeneous strong \textsc{Markov} property of this process. Then we  establish a \textsc{Freidlin-Sheu-}type formula for $X(\cdot)$, so as to relate this process to a solution of the local martingale problem associated with the triple $\, (\bm \sigma, \bm b, \bm \mu) \,$.

\medskip
\noindent
$\bullet~$
By the result of Part (a) of Proposition \ref{prop: SDEMP},  
  we obtain the equation (\ref{eq: ||X||}), thus prove Part (i) of Theorem \ref{thm: Gen}.  We also know that the ``direction process" $\, \mathfrak f\big( X(\cdot) \big)= \big( \mathfrak f_{1}\big( X(\cdot)\big), \mathfrak f_{2}\big( X(\cdot)\big) \big) \,$ is constant on every excursion interval of $\,\Vert X(t) \Vert\,$,  by applying the idea in the argument at the beginning of section \ref{disc}.

\smallskip
\noindent
$\bullet~$ 
For every $\, \varepsilon > 0\,$, we define the stopping times $\,\big\{ \tau_{m}^{\varepsilon} \,, \, m \in \N_0\big\} \,$ as in (\ref{eq: tau rec}).  
 
 \medskip
 \noindent
{\bf PROOF OF THEOREM \ref{thm: Gen}(ii), PART A: }  
Let us start by assuming  that, with probability one, all these stopping times $\,\big\{ \tau_{m}^{\varepsilon} \,, \, m \in \N_0\big\} \,$ are finite. Then for every $\, \varepsilon > 0\,$, $\, \ell \in \mathbb N_{0} \,$, we define also the 
measure  
\begin{equation} 
\label{eq: UniqMU}
\, \bm \mu^{\varepsilon}_{\ell} (B)\,: =\, \mathbb P\, \big( \mathfrak f (X(\tau_{ 2 \ell +1}^{\varepsilon})) \in B \big) , \,\quad \forall \, \, B \in \mathcal B (\mathfrak S)\, .
\end{equation}

\begin{prop}
\label{Invariant and IID}
The measure 
$\bm \mu^{\varepsilon}_{\ell} \,$ just introduced does not depend on either   
 $\,\varepsilon \,$ or $\, \ell \,$, so we can define $\, \bm \mu := \bm \mu^{\varepsilon}_{\ell}, \, \, \forall \, \varepsilon > 0, \,\ell \in \mathbb N_{0} \,$.  Furthermore,     $\, \big\{ \,\mathfrak f (X(\tau_{2 \ell + 1}^{\varepsilon}))\, \big\}_{\ell \in \mathbb N_{0}} \,$ is a sequence of independent random variables  with common distribution $\, \bm \mu \,,$ \,for every fixed $\,\varepsilon > 0 $. 
\end{prop} 

 \noindent
{\it Proof
: Step 1.}
We shall show in this step that $\,\mathfrak f (X(\tau_{2 \ell + 1}^{\varepsilon}))\,$ is independent of $\,\mathcal F^{X}(\tau_{2 \ell}^{\varepsilon})\,$ for any $\,\varepsilon > 0, \,\ell \in \mathbb N_{0}\,$, and that the random variables $\, \{\mathfrak f (X(\tau_{2 \ell + 1}^{\varepsilon}))\}_{\ell \in \mathbb N_{0}} \,$ are I.I.D. for any fixed $\,\varepsilon > 0 $. 
  
\smallskip
By Proposition \ref{prop: THSMP},  we have for every $\,   \varepsilon > 0, \,\ell \in \mathbb N_{0}, \,  B \in \mathcal B ( \mathfrak S)$, the identity 
\[
\mathbb P \big( \mathfrak f (X(\tau_{2 \ell + 1}^{\varepsilon})) \, \in \, B \, \big \vert \, \mathcal F^{X}(\tau_{2 \ell}^{\varepsilon})\big) = \mathbb P \big( X(\tau_{2 \ell }^{\varepsilon} + \,\cdot) \, \in \, A_{1} \, \big \vert \, \mathcal F^{X}(\tau_{2 \ell}^{\varepsilon}) \big)
= \mathbb P \big( X(\tau_{2 \ell }^{\varepsilon} + \,\cdot) \, \in \, A_{1} \, \big \vert  \, X(\tau_{2 \ell }^{\varepsilon}) \big)  \, . 
\]
Here $$\, A_{1} : = \big\{ \omega \in C[0, \infty)^{2} : \mathfrak f (\omega (\tau_{1}^{\varepsilon}(\omega))) \in B, \, \omega (0) \, =\, 0 \big\} \in \mathcal B \big( C [0, \infty)^{2} \big) \,  ,$$ 
\noindent
and the above conditional probability also equals $\, \mathfrak h_{1}(X(\tau_{2 \ell }^{\varepsilon}))\,$, for some bounded measurable function $\,\mathfrak h_{1}: \mathbb R^{2} \to \mathbb R \,$   that depends only on $\, A_{1} \,$. Now because $X(\tau_{2 \ell}^{\varepsilon}) \equiv 0 \,$, this conditional probability is a constant that is irrelevant to $\,\tau^{\varepsilon}_{2\ell}\,$, in particular, to $\,\ell$. We deduce that   $\,\mathfrak f (X(\tau_{2 \ell + 1}^{\varepsilon}))\,$ is independent of $\,\mathcal F^{X}(\tau_{2 \ell}^{\varepsilon})\,$, and its distribution does not depend on $\, \ell \,$. Therefore,  the random variables in   $\, \big\{\mathfrak f (X(\tau_{2 \ell + 1}^{\varepsilon}))\big\}_{\ell \in \mathbb N_{0}} \,$ are I.I.D.  

\noindent
{\it Step 2:} 
On the strength of Step 1, we can define $\, \bm \mu^{\varepsilon}:=  \bm \mu^{\varepsilon}_{\ell}\,, \,\, \forall \,\,\ell \in \mathbb N_{0}\,$. We shall show in this step that $\,\bm \mu^{\varepsilon}\,$ does not depend on $\,\varepsilon \,$. Once this is done, we shall obtain Proposition \ref{Invariant and IID} by combining the results of the two steps. Let $\varepsilon_{1} > \varepsilon_{2} > 0$. We shall prove the claim 
 \newpage
$$
\bm \mu^{\varepsilon_{1}}(B) \,= \,\bm \mu^{\varepsilon_{2}}(B)\,,\,\,\, \, \,\,\,\forall \, \,B \in \mathcal B ( \mathfrak S)\,.
$$ 
Since $\Vert X(\tau_{1}^{\varepsilon_{1}}) \Vert = \varepsilon_{1} > \varepsilon_{2}$, and $\, \Vert X(\cdot) \Vert \leq \varepsilon_{2}\,$ on every $\, [\tau_{2 \ell}^{\varepsilon_{2}}, \tau_{2 \ell+1}^{\varepsilon_{2}}]\,$, we see that for a.e. $\,\omega \in \Omega\,$ there exists a unique $\ell_{2} \in \mathbb N_{0}\,$ (depending on $\omega$), such that $\,\tau_{2 \ell_{2} +1}^{\varepsilon_{2}} < \tau_{1}^{\varepsilon_{1}} < \tau_{2 \ell_{2} +2}^{\varepsilon_{2}}\,$. Then we can   partition $\,\Omega = \bigcup_{\ell \in \mathbb N_{0}} \{ \tau_{2 \ell +1}^{\varepsilon_{2}} < \tau_{1}^{\varepsilon_{1}} < \tau_{2 \ell +2}^{\varepsilon_{2}} \} $, where the right-hand side is a disjoint union. On   $\,\{ \tau_{2 \ell +1}^{\varepsilon_{2}} < \tau_{1}^{\varepsilon_{1}} < \tau_{2 \ell +2}^{\varepsilon_{2}} \}$, we note that $\,\tau_{1}^{\varepsilon_{1}}\,$ and  $\, \tau_{2 \ell +1}^{\varepsilon_{2}}\,$ are on the same excursion interval of $\, \Vert X(\cdot)\Vert\,$. Then from the considerations in the first bullet, we have $\, \mathfrak f (X(\tau_{1}^{\varepsilon_{1}})) = \mathfrak f (X(\tau_{2 \ell +1}^{\varepsilon_{2}}))\,$ on the event $\,\{ \tau_{2 \ell +1}^{\varepsilon_{2}} < \tau_{1}^{\varepsilon_{1}} < \tau_{2 \ell +2}^{\varepsilon_{2}} \}$.

On the strength of Lemma \ref{Descrip} below, we can write
\[
\bm \mu^{\varepsilon_{1}} (B) = \mathbb P \big( \mathfrak f (X(\tau_{ 1}^{\varepsilon_{1}})) \in B \big) = \sum_{\ell \in \mathbb N_{0}} \mathbb P \Big( \{ \mathfrak f (X(\tau_{ 1}^{\varepsilon_{1}})) \in B \} \bigcap \{\tau_{2 \ell +1}^{\varepsilon_{2}} < \tau_{1}^{\varepsilon_{1}} < \tau_{2 \ell +2}^{\varepsilon_{2}} \} \Big)
\]
\[
\,= \sum_{\ell \in \mathbb N_{0}} \mathbb P \Big( \{ \mathfrak f (X(\tau_{ 2 \ell + 1}^{\varepsilon_{2}})) \in B \} \bigcap \{ \tau_{2 \ell}^{\varepsilon_{2}} < \tau_{1}^{\varepsilon_{1}} \} \bigcap \Big\{ \max_{\tau_{2 \ell +1}^{\varepsilon_{2}} \leq t \leq \tau_{2 \ell +2}^{\varepsilon_{2}}} \Vert X(t) \Vert \geq \varepsilon_{1} \Big\} \Big)
\]
\[
= \bm \mu^{\varepsilon_{2}} (B) \sum_{\ell \in \mathbb N_{0}} \mathbb P \Big( \big\{ \tau_{2 \ell}^{\varepsilon_{2}} < \tau_{1}^{\varepsilon_{1}} \big\} \bigcap \Big\{ \max_{\tau_{2 \ell +1}^{\varepsilon_{2}} \leq t \leq \tau_{2 \ell +2}^{\varepsilon_{2}}} \Vert X(t) \Vert \geq \varepsilon_{1} \Big\} \Big)
\]
\[
=\,\bm \mu^{\varepsilon_{2}} (B) \sum_{\ell \in \mathbb N_{0}} \mathbb P \big( \tau_{2 \ell +1}^{\varepsilon_{2}} < \tau_{1}^{\varepsilon_{1}} < \tau_{2 \ell +2}^{\varepsilon_{2}} \big) \,=\,\bm \mu^{\varepsilon_{2}} (B). 
\]
 This way we complete Step 2, and Proposition  \ref{Invariant and IID} is proved. \qed

 \begin{lm}
\label{Descrip}
(a) We have the comparisons \, $\tau_{2 \ell +1}^{\varepsilon_{2}} < \tau_{1}^{\varepsilon_{1}} < \tau_{2 \ell +2}^{\varepsilon_{2}} \,$, if and only if $\,\tau_{2 \ell}^{\varepsilon_{2}} < \tau_{1}^{\varepsilon_{1}}$ and $\, \max_{\tau_{2 \ell +1}^{\varepsilon_{2}} \leq t \leq \tau_{2 \ell +2}^{\varepsilon_{2}}} \Vert X(t) \Vert \geq \varepsilon_{1} $ hold. 

\smallskip
\noindent
(b)\, $\,\forall \, B \in \mathcal B(\mathfrak S)\,$, the three events $\, \{ \mathfrak f (X(\tau_{ 2 \ell + 1}^{\varepsilon_{2}})) \in B \}, \, \{ \tau_{2 \ell}^{\varepsilon_{2}} < \tau_{1}^{\varepsilon_{1}} \} ,\, \{ \max_{\tau_{2 \ell +1}^{\varepsilon_{2}} \leq t \leq \tau_{2 \ell +2}^{\varepsilon_{2}}} \Vert X(t) \Vert \geq \varepsilon_{1} \} \,$ are independent. 
\end{lm}

 \noindent
{\it Proof of Lemma \ref{Descrip}:} (a)\, It is fairly clear that, if $\,\tau_{2 \ell +1}^{\varepsilon_{2}} < \tau_{1}^{\varepsilon_{1}} < \tau_{2 \ell +2}^{\varepsilon_{2}} \,$, then $\, \tau_{2 \ell}^{\varepsilon_{2}} <\tau_{2 \ell +1}^{\varepsilon_{2}} < \tau_{1}^{\varepsilon_{1}}\,$, and $\, \max_{\tau_{2 \ell +1}^{\varepsilon_{2}} \leq t \leq \tau_{2 \ell +2}^{\varepsilon_{2}}} \Vert X(t) \Vert \geq \Vert X(\tau_{1}^{\varepsilon_{1}}) \Vert = \varepsilon_{1} \,$.   

\smallskip
Conversely, if $\,\tau_{2 \ell}^{\varepsilon_{2}} < \tau_{1}^{\varepsilon_{1}} \,$, then since $\Vert X(t) \Vert \leq \varepsilon_{2}$ for $ t \in [ \tau_{2 \ell}^{\varepsilon_{2}} , \tau_{2 \ell +1 }^{\varepsilon_{2}} ] $, we have $\tau_{2 \ell +1 }^{\varepsilon_{2}} < \tau_{1}^{\varepsilon_{1}}$.
On the other hand, if $\, \max_{\tau_{2 \ell +1}^{\varepsilon_{2}} \leq t \leq \tau_{2 \ell +2}^{\varepsilon_{2}}} \Vert X(t) \Vert \geq \varepsilon_{1} \,$, then $\exists \,\,t \in (\tau_{2 \ell +1}^{\varepsilon_{2}} , \tau_{2 \ell +2}^{\varepsilon_{2}} ) \subset (\tau_{0}^{\varepsilon_{2}} , \tau_{2 \ell +2}^{\varepsilon_{2}} ) = (\tau_{0}^{\varepsilon_{1}} , \tau_{2 \ell +2}^{\varepsilon_{2}} )\,$, such that $\,\Vert X(t) \Vert \geq \varepsilon_{1} \,$. Thus $\,\tau_{1}^{\varepsilon_{1}} < \tau_{2 \ell +2}^{\varepsilon_{2}} $, concluding the  proof of Part (a) of Lemma \ref{Descrip}.

\smallskip
\noindent
(b)\, By Step 1, proof of Proposition  \ref{Invariant and IID}, $\,\{ \mathfrak f (X(\tau_{ 2 \ell + 1}^{\varepsilon_{2}})) \in B \}\, $ is independent of $ \, \mathcal F^{X} (\tau_{2 \ell}^{\varepsilon_{2}})$. But $\{ \tau_{2 \ell}^{\varepsilon_{2}} < \tau_{1}^{\varepsilon_{1}} \} \in \mathcal F^{X} (\tau_{2 \ell}^{\varepsilon_{2}}) $, so $\, \{ \mathfrak f (X(\tau_{ 2 \ell + 1}^{\varepsilon_{2}})) \in B \} \,$ and $\, \{ \tau_{2 \ell}^{\varepsilon_{2}} < \tau_{1}^{\varepsilon_{1}} \}\,$ are independent, and both belong to $\,\mathcal F^{X} (\tau_{2 \ell +1}^{\varepsilon_{2}})$.

 \smallskip
Let $\, A_{2}\,:=\, \big \{ \omega \in C[0, \infty) : \omega (\cdot) \,\,\text{hits} \,\, \varepsilon_{1} \,\,\, \text{before hitting} \,\, 0 \,\, \text{ with } \,\,\omega(0) = \varepsilon_{2}  \big\} \, \in \, \mathcal B ( C [0, \infty) )\,.$
   Proposition \ref{prop: THSMP} applied to $\, \Vert X(\cdot)\Vert \,$  gives  that 
$$\,
\mathbb P \Big( \max_{\tau_{2 \ell +1}^{\varepsilon_{2}} \leq t \leq \tau_{2 \ell +2}^{\varepsilon_{2}}} \Vert X(t) \Vert \geq \varepsilon_{1} \, \big| \, \mathcal F^{X} (\tau_{2 \ell +1}^{\varepsilon_{2}}) \Big) \, = \,\mathbb P \big( \Vert X(\tau_{2 \ell +1}^{\varepsilon_{2}}+ \, \cdot)   \Vert \in A_{2} \, \big|  \, \mathcal F^{X} (\tau_{2 \ell +1}^{\varepsilon_{2}}) \big)
$$ 
  equals  
$ \,  \mathbb P \big( \Vert X(\tau_{2 \ell +1}^{\varepsilon_{2}}+   \cdot) \Vert \in A_{2} \, \big| \,\Vert X(\tau_{2 \ell +1}^{\varepsilon_{2}}) \Vert \big) \, ,$ 
 a measurable function of $\,\Vert X(\tau_{2 \ell +1}^{\varepsilon_{2}}) \Vert\,$. But we have $\,\Vert X(\tau_{2 \ell +1}^{\varepsilon_{2}}) \Vert \equiv \varepsilon_{2}$, and therefore the event $\, \big\{ \max_{\tau_{2 \ell +1}^{\varepsilon_{2}} \leq t \leq \tau_{2 \ell +2}^{\varepsilon_{2}}} \Vert X(t) \Vert \geq \varepsilon_{1} \big\} \,$ is independent of $\,\mathcal F^{X} (\tau_{2 \ell +1}^{\varepsilon_{2}})$. Combining this observation with the last paragraph, we complete the argument for Part (b). 
 
 This concludes the proof of Lemma \ref{Descrip}. \qed 
 
 \medskip
\noindent
$\bullet~$
With $\,\bm \mu\,$ defined as in Proposition \ref{Invariant and IID}, let $\,\bm \nu\,$ be the ``angular measure" on $([0, 2 \pi), \mathcal B([0, 2 \pi)))$ induced by   $\,\bm \mu\,$ on $\,(\mathfrak S,\, \mathcal B(\mathfrak S))\, $, through (\ref{nu_mu}). Thus with $\, \Theta(\cdot):= \text{arg}(X(\cdot)) \,$,   for every fixed $\,\varepsilon > 0 \,$   the random variables $\,\{ \Theta(\tau_{2 \ell + 1}^{\varepsilon})\}_{\ell \in \mathbb N_{0}} \,$ are I.I.D. with common distribution $\bm \nu$, following Proposition \ref{Invariant and IID}. 
  \newpage
We   turn now to the proof of the \textsc{Freidlin-Sheu} formula for $\,X(\cdot)\,$ in this setting: {\it For every function $\, g:    \mathbb R^{2} \rightarrow \R\,$   in the class $\,\mathfrak{D}\,$, defined as in subsection \ref{FSa}, we have } 
  \[
g( X(\cdot) )= g( {\rm x})+ \int_{0}^{\cdot} \, {\bf 1}_{\{ \lVert X(t) \rVert > 0 \}} \Big( {\bm b} ( \lVert X(t) \rVert ) \, G^{\prime}(X(t)) + \frac{1}{\, 2\, } \, {\bm a} ( \lVert X(t) \rVert ) \,G^{\prime\prime}(X(t)) \Big) {\rm d} t ~~~~~~~~~~~~~~~~~~~~~~
\]
\begin{equation}
\label{GenFS: Diffusion}
~~~~~~~~~~~~~~~~+\,\int_{0}^{\cdot} \, {\bf 1}_{\{ \lVert X(t) \rVert > 0 \}} {\bm \sigma} ( \lVert X(t) \rVert) \, G^{\prime}(X(t))\, {\mathrm d} W(t) \,+ \Big(\int_{0}^{2\pi} g^{\prime}_\theta (0+) \,  {\bm \nu}({\mathrm d} \theta) \Big) L^{\Vert X \Vert}(T)\, .
\end{equation}
With the considerations   at the   start of this section and   $\, S(\cdot):= \Vert X(\cdot) \Vert \,$, we can proceed exactly as the proof of Theorem \ref{Gen_FS}, except for the step of proving (\ref{007}), because now we do not have the help of (\ref{eq: localDist}).

We claim this convergence   still  holds. Setting $\,N(T, \varepsilon) \, :=\, \sharp \big\{ \ell \in \mathbb{N} \,: \tau^{\varepsilon}_{2 \ell } < T \big\}\,$ and   $\, h(\theta):= g^{\prime}_{\theta}(0+)$,   
\[
\sum_{\{\ell \, :\, \tau^{\varepsilon}_{2\ell+1}  < T\}} \varepsilon \, g^{\prime}_{\Theta(\tau^{\varepsilon}_{2\ell + 1})}(0+) \, =\, \varepsilon \, N(T, \varepsilon) \, \cdot \, \frac{1}{\, N(T, \varepsilon)\, } \sum_{\ell =0}^{N(T, \varepsilon)-1} h(\Theta(\tau^{\varepsilon}_{2\ell + 1})) \, + O(\varepsilon). 
\]
First, we have $\, \varepsilon \, N(T, \varepsilon) \,\xrightarrow[\varepsilon \downarrow 0 ]{} \, L^{\Vert X \Vert} (T)$ in probability, by Theorem VI.1.10 in \textsc{Revuz \& Yor} (1999). 
 Next, by the strong law of large numbers, we have $\,\frac{1}{N } \sum_{\ell =0}^{N-1} h(\Theta(\tau^{\varepsilon}_{2\ell + 1})) \, \xrightarrow[ N \rightarrow \infty ]{}\, \int_{0}^{2\pi} h(\theta) {\bm \nu}({\mathrm d} \theta) \,,$ a.e., for any fixed $\,\varepsilon$. By the definition of limit, we have $ \,\sup_{n \geq N} \big| (\frac{1}{n } \sum_{\ell =0}^{n-1} h(\Theta(\tau^{\varepsilon}_{2\ell + 1}))) - \int_{0}^{2\pi} h(\theta) {\bm \nu}({\mathrm d} \theta) \big| \, \xrightarrow[N \rightarrow \infty ]{} 0  \,,$ a.e., so this convergence is   also valid in probability. Moreover, this convergence in probability is uniform in $\,\varepsilon\,$, because the distribution of the random variable $\,\sup_{n \geq N} \big| (\frac{1}{n } \sum_{\ell =0}^{n-1} h(\Theta(\tau^{\varepsilon}_{2\ell + 1}))) - \int_{0}^{2\pi} h(\theta) {\bm \nu}({\mathrm d} \theta) \big|\,$ does not depend on $\,\varepsilon \,$. Now it is not hard to see that we have the convergence in probability 
$$\, 
\frac{1}{\, N(T, \varepsilon)\, } \sum_{\ell =0}^{N(T, \varepsilon)-1} h(\Theta(\tau^{\varepsilon}_{2\ell + 1})) \,\,\xrightarrow[\varepsilon \downarrow 0 ]{} \,\int_{0}^{2\pi} h(\theta) {\bm \nu}({\mathrm d} \theta)\,, \qquad \text{on the event} \,\, \,\,\Big\{ N(T, \varepsilon) \xrightarrow[\varepsilon \downarrow 0 ]{} \infty \Big\}\,.
$$ 
  Thus our claim holds on this event. On the complement of this event  the terms $ \frac{1}{\, N(T, \varepsilon)\, } \sum_{\ell =0}^{N(T, \varepsilon)-1} h(\Theta(\tau^{\varepsilon}_{2\ell + 1}))$ stay bounded, and we have $\, \varepsilon \, N(T, \varepsilon) \,\xrightarrow[\varepsilon \downarrow 0 ]  \, L^{\Vert X \Vert} (T)=0$, thus  
$$
\,\sum_{\{\ell \, :\, \tau^{\varepsilon}_{2\ell+1}  < T\}} \varepsilon \, g^{\prime}_{\Theta(\tau^{\varepsilon}_{2\ell + 1})}(0+) \, \, \xrightarrow[\varepsilon \downarrow 0 ]{}  \,\,\Big(\int_{0}^{2\pi} g^{\prime}_\theta (0+) \,  {\bm \nu}({\mathrm d} \theta) \Big) L^{\Vert X \Vert}(T)\,=\,0\,, \quad \text{ in probability.} 
$$ 
This establishes our claim, and obtains the \textsc{Freidlin-Sheu} Formula (\ref{GenFS: Diffusion}) for the state process $\, X(\cdot)$ of the posited weak solution. With (\ref{GenFS: Diffusion}) just established, and (\ref{eq: zeroLeb}) valid by assumption, we see that $\,X(\cdot)\,$ generates a probability measure on $\, (C[0, \infty)^{2}, \mathcal B(C[0, \infty)^{2}))\,$  which solves the local martingale problem associated with the triple $\, (\bm \sigma, \bm b,  {\bm \mu}) \,$, where $\,\bm \mu\,$ is defined as in Proposition 
 \ref{Invariant and IID}.  This   proves Part (ii) of Theorem \ref{thm: Gen},   assuming that the stopping times $\,\{ \tau_{m}^{\varepsilon} \}_{m \in \mathbb N_{0},\, \varepsilon >0}\,$ are all finite with probability one.


 \medskip
 \noindent
{\bf PROOF OF THEOREM \ref{thm: Gen}(ii), PART B: }  
When the stopping times $\,\{ \tau_{m}^{\varepsilon} \}_{m \in \mathbb N_{0},\, \varepsilon >0}\,$ can be infinite, we   proceed as follows. 
 
\smallskip
\noindent
{\it Step 1:} If $\, \mathbb P(\tau_{0}^{\varepsilon}< \infty) =0 \,$, then $\,L^{\Vert X\Vert}(\cdot)\equiv 0\,$ and (\ref{GenFS: Diffusion}) holds for any $\,\bm\nu\,$. Thus the conclusion of Part (ii) of Theorem \ref{thm: Gen} is true for any probability measure $\,\bm\mu\,$ on $\,(\mathfrak S,\, \mathcal B(\mathfrak S)) \,$. 
If $\, \mathbb P(\tau_{0}^{\varepsilon} <\infty) >0 \,$, we know from (\ref{eq: zeroLeb}) that $\,X(\cdot)\,$ can reach the origin and leave it with positive probability. So we can pick up a $\,\varepsilon_{0}\,$ such that $\, \mathbb P (\tau_{1}^{\varepsilon_{0}} < \infty) > 0 \,$. Then for every $\, \varepsilon \in (0, \varepsilon_{0}]\,$, $\, \ell \in \mathbb N_{0} \,$, define the probability measure $\, \bm \mu^{\varepsilon}_{\ell} \,$ by 
$$\, \bm \mu^{\varepsilon}_{\ell} (B)\,: =\, \frac{\mathbb P \big( \mathfrak f (X(\tau_{ 2 \ell +1}^{\varepsilon})) {\bf 1}_{\{\tau_{ 2 \ell +1}^{\varepsilon} < \infty\} } \in B \big)}{\mathbb P \big(\tau_{ 2 \ell +1}^{\varepsilon} < \infty \big)} \,, \,\quad \forall \, \, B \in \mathcal B (\mathfrak S)\,.$$ 
This is well-defined for $\,\ell =0\,$, by our choice of $\,\varepsilon_{0}\,$. If $\,\mathbb P \big(\tau_{ 2 \ell +1}^{\varepsilon} < \infty \big)=0\,$ for some $\,\ell \geq 1\,$, we redefine $\, \bm \mu^{\varepsilon}_{\ell} \,$ by $\, \bm \mu^{\varepsilon}_{\ell} := \bm \mu^{\varepsilon}_{0} \,$. 
 \newpage

\noindent
{\it Step 2:} It is straightforward but heavier in notation, to follow the steps of Proposition \ref{Invariant and IID} and Lemma \ref{Descrip} and check that $\,\bm \mu^{\varepsilon}_{\ell} \,$ does not depend on either $\,\varepsilon \,$ or $\, \ell \,;$ so we can define $\, \bm \mu := \bm \mu^{\varepsilon}_{\ell}, \, \, \forall \, \varepsilon > 0, \,\ell \in \mathbb N_{0} \,$. Now we enlarge the original probability space by means of a countable collection of $\,\mathfrak S$-valued I.I.D. random variables $\,\{\bm \xi_{\ell}^{\varepsilon} \}_{\varepsilon \in \mathbb Q^{+}, \ell \in \mathbb N_{0}}\,$ with common distribution $\,\bm \mu\,$,   and independent of the $\,\sigma-$algebra $\,\mathcal F\,$. For every $\,\varepsilon \in \mathbb Q^{+}, \,\,\ell \in \mathbb N_{0}\,$, we define the $\,\mathfrak S$-valued   $\,\,\, \widetilde{\mathfrak f} \big(X(\tau_{2\ell +1}^{\varepsilon})\big):=  \mathfrak f \big(X(\tau_{2\ell +1}^{\varepsilon})\big) \, {\bf 1}_{\{\tau_{ 2 \ell +1}^{\varepsilon}<\infty\}} + \, \bm \xi_{\ell}^{\varepsilon}\, {\bf 1}_{\{\tau_{ 2 \ell +1}^{\varepsilon}=\infty\}} \, .$  

\smallskip
It is   straightforward but tedious, to check that for any $\,  \varepsilon \in \mathbb Q^{+}\,$, the random variables $\, \big\{ \,\widetilde{\mathfrak f} (X(\tau_{2 \ell + 1}^{\varepsilon})) \big\}_{\ell \in \mathbb N_{0}} \,$ are independent with common distribution $\, \bm \mu \,$. Then in the same way as in the last subsection, we can argue the convergence in (\ref{007}) 
along rationals. The proof of Part (ii) of Theorem \ref{thm: Gen} is now complete.

 \bigskip
 \noindent
{\bf PROOF OF THEOREM \ref{thm: Gen}(iii):}  
Finally, we argue Part (iii). Under the assumptions  for Parts (ii) and (iii), let $\, \bm\mu\,$ be some probability measure for which   the conclusion   (ii) holds. Then by   Proposition \ref{prop: SDEMP}(b), we know that $\, X(\cdot)\,$ also solves (\ref{eq: SDEMP}) with $\, \gamma_{i}\,$ replaced by $\,\int_{ \mathfrak S}  \mathfrak f_i (z) \,  {\bm \mu} (\mathrm{d} z)\,$. Thus we must have $\, \gamma_{i}= \int_{ \mathfrak S}  \mathfrak f_i (z) \,  {\bm \mu} (\mathrm{d} z)$, which is    (\ref{gamma0}), on the strength of $\,\mathbb{P} \big( L^{\Vert X\Vert} (\infty) >0 \big) > 0\,$. Moreover  (\ref{eq: localDist}) also holds, namely 
$$
L^{R^{A}}(\cdot) \, \equiv \, {\bm \nu} (A) \, L^{ \,\lVert X \rVert}(\cdot) \, , \,\,\,\, \,\,\,\, \, \forall \,\,\, A \in \mathcal B ([0, 2\pi)) \,,
$$
with   $\,R^{A}(\cdot) =   \lVert X (\cdot) \rVert \cdot {\bf 1}_A \big(\text{arg}\big(X(\cdot)\big) \big)  \,$. Thanks to $\,\mathbb{P} \big( L^{\Vert X\Vert} (\infty) >0 \big) > 0\,$ again, we see from the above relationship  that $\, X(\cdot)\,$ uniquely determines $\,\bm \nu $, thus also   $\,\bm\mu\,$.    The  proof of Theorem \ref{thm: Gen} is now complete. 
\qed

\medskip

\section*{Bibliography}

  \noindent \textsc{Barlow, M.T., \'Emery, M., Knight, F.B., Song, S. \& Yor, M.} (1998)  Autour d'un th\'eor\`eme de Tsirel'son sur des filtration browniennes. {\it S\'eminaire de Probabilit\'es XXXII}. Lecture Notes in Mathematics {\bf 1686}, 264-305.  Springer-Verlag, New York.

\medskip

\noindent \textsc{Barlow, M.T., Pitman, J.W. \& Yor, M.} (1989)  On Walsh's {B}rownian motions. {\it S\'eminaire de Probabilit\'es  XXIII}. Lecture Notes in Mathematics {\bf 1372}, 275-293. Springer-Verlag, New York.

 
 \medskip  
 \noindent \textsc{Biane, P. and Yor, M.} (1987)  Valeurs principales associ\'ees aux temps locaux browniens. {\it Bull. Sci. Math.}  {\bf 111},  23-101. 
 
 \medskip 
 
 \noindent \textsc{Chen, Z.-Q. \& Fukushima, M.} (2015) One-point reflection. {\it Stochastic Processes and Their Applications} {\bf 125},   1368-1393. 
  
\medskip
\noindent \textsc{Engelbert, H.J. \& Schmidt, W.} (1984) On one-dimensional stochastic differential equations with generalized drift. {\it Lecture Notes in Control and Information Sciences} {\bf 69}, 143-155. Springer-Verlag, NY.

\medskip 

\noindent \textsc{Evans, S.N. \& Sowers, R.B. } (2003) {Pinching and twisting Markov processes}. {\it Annals of Probability} {\bf 31}, 486-527.

\medskip 

\noindent \textsc{Fitzsimmons, P.J. \& Kuter, K.E.} (2014)  Harmonic functions on Walsh's {B}rownian motion. {\it Stochastic Processes and their Applications} {\bf 124}, 2228-2248. 

\medskip 

\noindent \textsc{Freidlin, M. \& Sheu, S.} (2000)  Diffusion processes on graphs: stochastic differential equations, large deviation principle. {\it Probability Theory and Related Fields} {\bf 116} 181-220.
 
\medskip 
\noindent \textsc{Hajri, H.} (2011)  Stochastic flows related to Walsh Brownian motion. {\it Electronic  Journal of  Probability} {\bf 16}, 1563-1599.  

\medskip

\noindent \textsc{Hajri, H. \& Touhami, W.} (2014) It\^o's formula for Walsh's Brownian motion and applications. {\it Statistics and Probability Letters} {\bf 87} 48-53. 

\medskip 

\noindent
\textsc{Harrison, J.M. \& Shepp, L.A.} (1981)  On skew Brownian motion. {\it Annals of Probability} {\bf  9}, 309-313.

\medskip 

\noindent \textsc{Ichiba, T. \& Karatzas, I.} (2014)  Skew-unfolding the {S}korokohod reflection of a continuous semimartingale. {\it Stochastic Analysis and Applications: In Honour of T.J. Lyons} (Dan Crisan et al., eds), 349-376. 

\medskip 

\noindent
\textsc{Ichiba, T., Karatzas, I. \& Prokaj, V.} (2013) Diffusions with rank-based characteristics and values in the nonnegative quadrant. {\it Bernoulli} {\bf 19}, 2455-2493. 

\medskip 

\noindent 
\textsc{Ichiba, T., Papathanakos, V., Banner, A., Karatzas, I., Fernholz, E.R.}  (2011) Hybrid Atlas models. {\it Ann. Appl.
Probab.} {\bf 21}, 609-644. 

\medskip 
 
\noindent \textsc{Ikeda, N. \& Watanabe, S.} (1989) {\it Stochastic Differential Equations and Diffusion Processes.} Second edition, North Holland Publ. Co., Amsterdam Oxford New York; Kodansha Ltd., Tokyo. 
 
 \medskip  
 
 \noindent\textsc{It\^o, K. \& Mc$\,${K}ean, H.P., Jr.}  (1963)  Brownian motions on a half-line. {\it Illinois J.   Math.}  {\bf 7}, 181-231.  

\medskip 

\noindent \textsc{Jacod, J.} (1977) A general theorem of representation for martingales. In {\it ``Probability"} (Univ. of Illinois at Urbana, March 1976; J.L. Doob Editor).  {\it Proceedings of Symposia in Pure Mathematics} {\bf   XXXI}, 37-53. American Mathematical Society, Providence, R.I.

\medskip 

\noindent 
\textsc{Kang, W. \& Ramanan, K.}  (2014) On the submartingale problem for reflected diffusions in domains with piecewise smooth boundaries. Preprint available at $~${\it http://arxiv.org/abs/1412.0729}.

\medskip 
 \noindent    \textsc{Karatzas, I. \& Shreve, S.E.}  (1991)  {\it Brownian Motion and Stochastic Calculus.}   Springer-Verlag, NY.

\medskip 

\noindent \textsc{Mansuy, R. \& Yor, M.} (2006)  Random times and enlargements of filtrations in a {Brownian} setting. {\it Lecture Notes in Mathematics} {\bf 1873}. Springer-Verlag, New York.


\medskip 
\noindent \textsc{Picard, J.} (2005) Stochastic calculus and martingales on trees. {\it Annales de l' Institut Henri Poincar\'e (S\'er. B: Probabilit\'es et Statistique) } {\bf 41}, 631-683.

\medskip 
\noindent \textsc{Prokaj, V.} (2009)  Unfolding the Skorokhod reflection of a semimartingale.   {\it Statistics and Probability Letters} {\bf 79}, 534-536.

 \medskip 
 \noindent
  \textsc{Revuz, D. \& Yor, M.} (1999) {\it Continuous Martingales and Brownian Motion.} Third Edition, Springer-Verlag, New York. 

\medskip 

\noindent \textsc{Rogers, L.C.G. \& Williams, D.} (2000)  {\it Diffusions, Markov Processes, and Martingales. Vol.  2: It\^o Calculus.} Cambridge University Press, Cambridge. 

\medskip 

\noindent \textsc{Salminen, P., Vallois, P. \& Yor, M.} (2007)  On the excursion theory for linear diffusions. {\it Japanese Journal of Mathematics} {\bf 2}, 97-127. 

\medskip
\noindent \textsc{Schmidt, W.} (1989)  On stochastic differential equations with reflecting barriers. {\it Mathematische Nachrichten} {\bf 142}, 135-148.

\medskip 

\noindent \textsc{Stroock, D.W. \& Varadhan, S.R.S.} (1971) Diffusion processes with boundary conditions. {\it Communications in Pure and Applied Mathematics} {\bf 24}, 147-225. 

\medskip 
\noindent \textsc{Tsirel'son, B.} (1997). Triple points: From non-Brownian filtration to harmonic measures. {\it Geometric and Functional Analysis} {\bf 7}, 1096-1142.

\medskip

\noindent \textsc{Walsh, J.B.} (1978) A diffusion with a discontinuous local time. {\it Ast\'erisque} {\bf 52-53}, 37-45.

\medskip 

\noindent \textsc{Watanabe, S.} (1999) The existence of a multiple spider martingale in the natural filtration of a certain diffusion in the plane. {\it S\'eminaire de Probabilit\'es  XXXIII}. {Lecture Notes in Mathematics} {\bf 1709} 277-290. 

\medskip

\noindent \textsc{Yor, M.} (1997) {\it Some Aspects of Brownian Motion. Part II: Some Recent Martingale Problems.} Birkh\"auser, Basel and Boston.

\end{document}